\RequirePackage{fix-cm}
\documentclass[envcountsect,smallextended]{svjour3}       % onecolumn (second format)
\smartqed  % flush right qed marks, e.g. at end of proof
\usepackage{amsmath,amssymb}
\usepackage[dvipsnames]{xcolor}
\usepackage{enumerate}
\usepackage{graphicx}
\usepackage{epstopdf}
\usepackage{relsize}
\usepackage{url}
\usepackage{multirow}
\usepackage[left=1in,top=1in,right=1in,bottom=1in,letterpaper]{geometry}
\usepackage{mathrsfs}
\usepackage{algorithm}
\usepackage{algpseudocode}
\usepackage{framed}
\usepackage{tikz-cd}
\usepackage{booktabs}
\newtheorem{example-numbered}{Example}[section]

\newcommand{\Hol}{\mathrm{Hol}}
\newcommand{\hol}{\mathrm{hol}}
\newcommand{\Hom}{\mathrm{Hom}}
\newcommand{\GL}{\mathrm{GL}}
\newcommand{\Aut}{\mathrm{Aut}}
\newcommand{\PSL}{\mathrm{PSL}}
\newcommand{\SO}{\mathrm{SO}}
\newcommand{\SU}{\mathrm{SU}}
\newcommand{\SL}{\mathrm{SL}}

\newcommand{\rank}{\mathrm{rank}}

\DeclareMathOperator{\im}{im}

\allowdisplaybreaks

\def\makeheadbox{\relax}
\begin{document}

\title{The Geometry of Synchronization Problems and Learning Group Actions\thanks{TG gratefully acknowledges partial support from Simons Math+X Investigators Award 400837, DARPA D15AP00109, NSF IIS 1546413, and an AMS-Simons Travel Grant; JB would like to acknowledge the support for this work by the EPSRC grants EP/I016945/1 and EP/N014189/1; SM would like to acknowledge support from NSF DMS 16-13261,  NSF IIS 1546331, NSF DMS-1418261, NSF IIS-1320357, NSF DMS-1045153, and HFSP RGP0051/2017.}
}


\author{Tingran Gao         \and
  Jacek Brodzki       \and
  Sayan Mukherjee
}

\authorrunning{Tingran Gao,  Jacek Brodzki, and Sayan Mukherjee} % if too long for running head

\institute{Tingran Gao \at
              Committee on Computational and Applied Mathematics,\\
              Department of Statistics,\\
              University of Chicago, Chicago, IL 60637, USA\\
              \email{tingrangao@galton.uchicago.edu}           %  \\
           \and
           Jacek Brodzki \at
           Department of Mathematical Sciences,\\
           University of Southampton, Southampton SO17 1BJ, \\
           \email{j.brodzki@soton.ac.uk}
           \and
           Sayan Mukherjee \at
           Departments of Statistical Science, Mathematics, Computer Science, and Bioinformatics \& Biostatistics, \\
           Duke University, Durham, NC 27708, USA\\
           \email{sayan@stat.duke.edu}
}

\date{Received: 16 April 2018 / Accepted: 29 April 2019}

\maketitle

\begin{abstract}
We develop a geometric framework, based on the classical theory of fibre bundles, to characterize the cohomological nature of a large class of \emph{synchronization-type problems} in the context of graph inference and combinatorial optimization.
%--- registering or aligning a collection of objects in a consistent manner ---
We identify each synchronization problem in topological group $G$ on connected graph $\Gamma$ with a flat principal $G$-bundle over $\Gamma$, thus establishing a classification result for synchronization problems using the representation variety of the fundamental group of $\Gamma$ into $G$. We then develop a twisted Hodge theory on flat vector bundles associated with these flat principal $G$-bundles, and provide a geometric realization of the \emph{graph connection Laplacian} as the lowest-degree Hodge Laplacian in the twisted de Rham-Hodge cochain complex. Motivated by these geometric intuitions, we propose to study the problem of \emph{learning group actions} --- partitioning a collection of objects based on the local synchronizability of pairwise correspondence relations --- and provide a heuristic synchronization-based algorithm for solving this type of problems. We demonstrate the efficacy of this algorithm on simulated and real datasets.
\end{abstract}

\section{Introduction}
\label{sec:introduction}

Over the past century, concepts from differential geometry have had a strong impact on probability theory, statistical inference, and machine learning \cite{Chentsov66,Cramer46,Fisher1922,Malhanobis36,Rao45}. Two central geometric concepts used in these fields have been differential operators (e.g. the Laplace--Beltrami operator \cite{BelkinNiyogi2004ML}) and Riemannian metrics (e.g. Fisher information \cite{Fisher1922}). In particular, the research program of manifold learning studies dimension reduction through the lens of differential-geometric quantities and invariants, and designs data compression algorithms that preserve intrinsic geometric information such as geodesic distances \cite{isomap}, affine connections \cite{lle}, second fundamental forms \cite{donoho03}, and heat kernels \cite{LapEigMaps2003,CoifmanLafonLMNWZ2005PNAS1,CoifmanLafonLMNWZ2005PNAS2}. The underlying hypothesis of these techniques is that the data lie approximately on a smooth manifold (often embedded in an ambient Euclidean space), a scenario facilitating inference due to smoothly controllable transitions between observed and unseen data. For practical purposes, discrete analogues of the inherently smooth theory of differential geometry have also been explored in fields ranging from geometry processing \cite{DDG2008,Crane:2013:DGP}, finite element methods \cite{DRR2006FEEC}, to spectral graph theory \cite{Chung1997} and diffusion geometry \cite{CoifmanLafon2006,SingerWu2012VDM}.

Beyond the manifold assumption, geometric objects can be handled with ``softer'' tools such as topology: topological data analysis techniques \cite{CZCG2005,EH2010Book} have been developed to study datasets based on their persistent homology. For smooth manifolds, it is well known that the singular cohomology and de Rham cohomology are isomorphic, indicating that some topological information can be read off from the differential structure of geometric objects. Carrying the de Rham theory beyond the manifold setting has attracted the interest of geometers and physicists: synthetic differential geometry \cite{Kock1982,Kock2006SDG} defines group-valued differential forms on ``formal manifolds'' (generalized notion of smooth spaces for which infinitesimal neighborhoods are specified axiomatically), based on which an analog of the classical de Rham theory can be established \cite{FelixLavendhomme1990}; noncommutative differential geometry \cite{Connes1985,Connes2000NG,Madore1999Book} builds upon the observation that much of differential geometry can be formulated in terms of the algebra of smooth functions defined on smooth manifolds, and replaces this algebra with noncommutative ones --- differential forms can then be extended to ``noncommutative spaces'' along with homology and cohomology of much more general objects. Discrete analogs of the Hodge Laplacian, a second order differential operator closely related to de Rham theory, have been proposed for simplicial complexes and graphs \cite{JLYY2011,Lim2015,PR2016,PRT2015,SKM2014}; its non-commutative counterpart for $1$-forms on graphs have recently been explored in \cite{Majid2013NRGG}.

Bridging recent developments applying differential geometry and topology in probability and statistical sciences, the problem of \emph{synchronization} \cite{BSS2013,WangSinger2013} arise in a variety of fields in computer science (e.g. computer vision \cite{BCSZ2014} and geometry processing \cite{KKSBL2015}), signal processing (e.g. sensor network localization \cite{CLS2012}), combinatorial optimization (e.g. non-commutative Grothendieck inequality \cite{BKS2016}), and natural sciences (e.g. cryo-electron microscopy \cite{BCS2015NUG,SS2012,SZSH2011} and geometric morphometrics \cite{Gao2015Thesis}). The data given in a synchronization problem include a connected graph that encodes similarity relations within a collection of objects, and pairwise correspondences --- often realized as elements of a transformation group $G$ --- characterizing the nature of the similarity between a pair of objects linked directly by an edge in the relation graph. The general goal of the problem is to adjust the pairwise correspondences, which often suffer from noisy or incomplete measurements, to obtain a globally consistent characterization of the pairwise relations for the entire dataset, in the sense that unveiling the transformation between a pair of objects far-apart in the relation graph can be done by composing transformations along consecutive edges on a path connecting the two objects, and the resulting composed transformation is independent of the choice of the path. (A precise definition of a synchronization problem will be provided below; see Section~\ref{sec:fibre-bundle-interpr}.) This paper stems from our attempt to gain a deeper understanding of the geometry underlying synchronization problems.
%, and its relation to the problem of learning group actions.
Whereas previous works \cite{SingerWu2012VDM,SingerWu2013} in this direction build upon manifold assumptions, the point of view we adopt here is synthetic and noncommutative: we will see that inference is possible due to rigidity rather than smoothness.

The remainder of this section gives a formal definition of synchronization problems, as well as a geometric interpretation in the language of fibre bundles. The fibre bundle interpretation is elementary but has not been presented in the literature of synchronization problems, to our knowledge. We then state the main results, discuss related works, and describe the organization of the paper.

\begin{table}[htbp]
\caption{Notations used throughout this paper}
\begin{center}
\begin{tabular}{c c p{12cm} }
\toprule
$\Gamma$ & & graph\\
$V$ & & vertex set of $\Gamma$\\
$E$ & & edge set of $\Gamma$\\
$n$ or $\left| V \right|$ & & number of vertices of the graph $\Gamma$\\
$m$ or $\left| E \right|$ & & number of edges of the graph $\Gamma$\\
$w_{ij}$ & & weight on edge $\left( i,j \right)\in E$\\
$d_i$ & & weighted degree on vertex $i\in V$, defined as $d_i=\sum_{j:\left( i,j \right)\in E}w_{ij}$\\
$G$ & & topological group\\
$e$ & & identity element of $G$\\
$G_{\delta}$ & & group $G$ equipped with discrete topology\\
$\mathbb{K}$ & & scalar field $\mathbb{R}$ or $\mathbb{C}$\\
$F$ & & vector space on $\mathbb{K}$ that is a representation space of $G$\\
$d$ or $\dim F$ & & dimension of the vector space $F$\\
$\left\langle \cdot,\cdot \right\rangle_F$ & & inner product on $F$\\
$\mathfrak{U}=\left\{ U_i\mid 1\leq i\leq \left| V \right| \right\}$ & & open cover of $\Gamma$ in which $U_i$ is the star of vertex $i\in V$\\
$C^0 \left( \Gamma; G \right)$ & & $G$-valued $0$-cochain on $\Gamma$, or the set of all vertex potentials on $\Gamma$\\
$C^1 \left( \Gamma; G \right)$ & & $G$-valued $1$-cochain on $\Gamma$, or the set of all edge potentials on $\Gamma$\\
$C^0 \left( \Gamma; F \right)$ & & $F$-valued $0$-cochain on $\Gamma$\\
$\mathscr{B}_{\rho}$ & & synchronization principal bundle (a flat principal $G$-bundle on $\Gamma$) associated with $\rho\in C^1 \left( \Gamma; G \right)$\\
$\mathscr{B}_{\rho}\left[ F \right]$ & & flat associated $F$-bundle of $\mathscr{B}_{\rho}$\\
$\hol_{\rho}$ & & holonomy homomorphism on $\mathscr{B}_{\rho}$, from $\pi_1 \left( \Gamma \right)$ to $G$\\
$\Hol_{\rho} \left( \Gamma \right)$ & & holonomy of the synchronization principal bundle $\mathscr{B}_{\rho}$\\
$\Omega_i^0 \left( \Gamma;\mathscr{B}_{\rho}\left[ F \right] \right)$ & & constant twisted local $0$-forms of $\mathscr{B}_{\rho}\left[ F \right]$ on $U_i$, i.e. constant local sections of $\mathscr{B}_{\rho}\left[ F \right]$ on $U_i$\\
$\Omega^0 \left( \Gamma;\mathscr{B}_{\rho}\left[ F \right] \right)$ & & locally constant twisted global $0$-forms of $\mathscr{B}_{\rho}\left[ F \right]$, i.e. locally constant global sections of $\mathscr{B}_{\rho}\left[ F \right]$\\
$\Omega_i^1 \left( \Gamma;\mathscr{B}_{\rho}\left[ F \right] \right)$ & & constant twisted local $1$-forms of $\mathscr{B}_{\rho}\left[ F \right]$ on $U_i$\\
$\Omega^1 \left( \Gamma;\mathscr{B}_{\rho}\left[ F \right] \right)$ & & locally constant twisted global $1$-forms of $\mathscr{B}_{\rho}\left[ F \right]$\\
$\left[ f \right]$ & & vector in $\mathbb{K}^{nd}$ representing $f\in C^0 \left( \Gamma; F \right)$ \\
$d_{\rho}$ & & $\rho$-twisted differential where $\rho\in C^1 \left( \Gamma;G \right)$, from $C^0 \left( \Gamma;F \right)$ to $\Omega^1 \left( \Gamma;\mathscr{B}_{\rho}\left[ F \right] \right)$\\
$\delta_{\rho}$ & & $\rho$-twisted codifferential where $\rho\in C^1 \left( \Gamma;G \right)$, from $\Omega^1 \left( \Gamma;\mathscr{B}_{\rho}\left[ F \right] \right)$ to $C^0 \left( \Gamma;F \right)$\\
$\Delta_{\rho}^{\left( 0 \right)}$& & $\rho$-twisted Hodge Laplacian of degree $0$\\
$\Delta_{\rho}^{\left( 1 \right)}$& & $\rho$-twisted Hodge Laplacian of degree $1$\\
$H_{\rho}^0 \left( \Gamma; \mathscr{B}_{\rho}\left[ F \right] \right)$ & & The $0$th twisted cohomology group for $\mathscr{B}_{\rho}\left[ F \right]$, where $\rho\in C^1 \left( \Gamma; G \right)$\\
$\nu \left( S \right)$ & & frustration of the subgraph of $\Gamma$ spanned by the vertex subset $S\subset V$\\
\bottomrule
\end{tabular}
\end{center}
\label{tab: notation}
\end{table}

\subsection{A Fibre Bundle Interpretation of Synchronization Problems}
\label{sec:fibre-bundle-interpr}

We begin with a standard formulation of the synchronization problem originated in a series of works by A. Singer and collaborators \cite{Singer2011a,SS2012,BSS2013,WangSinger2013,BSAB2014,BCS2015NUG}. Let $\Gamma= \left( V,E,w \right)$ be an undirected weighted graph with vertex set $V$, edge set $E$, and weights $w_{ij}$ for each $\left( i,j \right)\in E$. Assume $G$ is a topological group acting on a normed vector space $F$. Given a map $\rho:E\rightarrow G$ from the edges of $\Gamma$ to the group $G$ satisfying $\rho_{ij}=\rho_{ji}^{-1}$, the objective of a \emph{$F$-synchronization problem over $\Gamma$ with respect to $\rho$} is to find a map $f:V\rightarrow F$ satisfying the constraints
\begin{equation}
  \label{eq:vertex-pot-satisfies-edge-pot}
  f_i=\rho_{ij}f_j\qquad \forall \left( i,j \right)\in E.
\end{equation}
If no such map $f$ exists, the synchronization problem  consists of finding a map $f$ from $V$ to $F$ that satisfies the constraints as much as possible, in the sense of minimizing the \emph{frustration}
\begin{equation}
  \label{eq:frustration}
  \eta \left( f \right) = \frac{1}{2}\frac{\displaystyle\sum_{i,j\in V}w_{ij}\left\| f_i-\rho_{ij}f_j \right\|_F^2}{\displaystyle\sum_{i\in V} d_i \left\| f_i \right\|_F^2},
\end{equation}
where $\left\| \cdot \right\|_F$ is a norm defined on $F$, and $d_i=\sum_{j:\left( i,j \right)\in E}w_{ij}$ is the weighted degree at vertex $i$. In the terminology of \cite{BSS2013}, $\rho$ is an \emph{edge potential} and $f$ is a \emph{vertex potential}; a vertex potential is said to \emph{satisfy} a given edge potential if all equalities in \eqref{eq:vertex-pot-satisfies-edge-pot} hold. Varying the choice of group $G$ and field $F$ results in different realizations of the synchronization problem \cite{SingerWu2011ODM,SZSH2011,TSR2011,SS2012,WangSinger2013,BSS2013,BSAB2014}, as will be elaborated in Section~\ref{sec:other-related-work}.

Since we will frequently refer to the set of all edge and vertex potentials on a graph, let us introduce the following notations to ease our exposition: let $C^0 \left( \Gamma; G \right)$, $C^1 \left( \Gamma; G \right)$ denote respectively the set of all $G$-valued vertex and edge potentials on $\Gamma$, i.e.
\begin{equation}
\label{eq:group_valued_vert_edge_potentials}
  C^0 \left( \Gamma; G \right) := \left\{ f:V\rightarrow G \right\},\qquad C^1 \left( \Gamma; G \right) := \left\{ \rho:E\rightarrow G\mid \rho_{ij}=\rho_{ji}^{-1},\forall \left( i,j \right)\in E \right\}.
\end{equation}
For cohomological reasons that will become clear in Section~\ref{sec:geom-coho-sync}, we will also call $C^0 \left( \Gamma; G \right)$ and $C^1 \left( \Gamma; G \right)$ the \emph{$G$-valued $0$-} and \emph{$1$-cochains} on $\Gamma$, respectively. Similarly, let
\begin{equation}
  \label{eq:F-valued-vertex-potential}
  C^0 \left( \Gamma; F \right):=\left\{ f:V\rightarrow F \right\}
\end{equation}
denote the set of all $F$-valued vertex potentials on $\Gamma$. Throughout this paper, a $G$-valued edge potential $\rho\in\Omega^0 \left( \Gamma;G \right)$ is said to be \emph{synchronizable} if there exists a $G$-valued vertex potential $f\in\Omega^0 \left( \Gamma; G \right)$ satisfying $f_i=\rho_{ij}f_j,\,\,\forall \left( i,j \right)\in E$, i.e. \eqref{eq:vertex-pot-satisfies-edge-pot} is satisfied with $F=G$. Generally, an $F$-valued vertex potential satisfying \eqref{eq:vertex-pot-satisfies-edge-pot} will be referred to as a \emph{solution to the $F$-synchronizable problem over $\Gamma$ with respect to $\rho$}, or simply \emph{$F$-valued synchronization solution}. Clearly, $\rho$ is synchronizable if and only if a $G$-valued synchronization solution exists.

When $F=G$, i.e. when we consider the action of $G$ on itself, a synchronizable edge potential can be realized geometrically as a \emph{flat}\footnote{Recall (see, e.g. \cite[\S 2]{Tu1983Hodge}) that a fibre bundle $\pi:\mathfrak{B}\rightarrow X$, with total space $\mathfrak{B}$ and base space $X$, is said to be flat if it admits a system of local trivializations with locally constant bundle coordinate transformations.} principal bundle that is isomorphic to a product space in its entirety, i.e. a \emph{trivial}\footnote{Note that a flat bundle is not necessarily trivial (i.e. isomorphic to a product space) --- the fundamental group of the base space plays a central role in this development (see e.g. \cite[Chapter 2]{Morita2001GCC}).} flat principal bundle, as will be explained in Proposition~\ref{prop:synchronizability_flat_bundle} and Proposition~\ref{prop:synchronizability_flat_bundle_one_skelenton} below; this observation forms the backbone of the entire geometric framework we develop in this paper. When the fibre bundle is differentiable, this notion of flatness is equivalent to the existence of a flat connection on the bundle, which is essentially a special case of the Riemann-Hilbert correspondence \cite{Esnault1988}. The main results of this paper build upon extending further and deeper the analogy between the geometry of synchronization problems and fibre bundles.

Proposition~\ref{prop:synchronizability_flat_bundle} and Proposition~\ref{prop:synchronizability_flat_bundle_one_skelenton} characterize the basic building block for the geometric formulation of synchronization problems. We will develop the principal bundle in the generality of topological spaces that includes smooth structures as particular cases. Following Steenrod~\cite{Steenrod1951}, a fibre bundle is a quintuple $\mathscr{E}=\left( E,M,F,\pi,G \right)$ where $E$, $M$, $F$ are topological spaces, referred to as the \emph{total space}, \emph{base space}, and \emph{fibre space}, respectively; $\pi:E\rightarrow M$ is a continuous surjective map, called the \emph{bundle projection}, and $M$ adopts an open cover $\left\{ U_i \right\}$ with homeomorphisms $\phi_i:U_i\times F\rightarrow\pi^{-1}\left( U_i \right)$ between each $\pi^{-1}\left( U_i \right)\subset E$ and the product space $U_i\times F$, such that $\pi\big|_{\pi^{-1}\left( U_i \right)}$ is the composition of $\phi_i$ with $\mathrm{proj}_1:U_i\times F\rightarrow U_i$, the canonical projection onto the first factor of the product space. In other words, the following diagram is commutative:
\begin{center}
\begin{tikzcd}[column sep=small]
\pi^{-1}\left( U_i \right) \arrow{rr}{\phi_i} \arrow[swap]{dr}{\pi}& &U_{i}\times F \arrow{dl}{\mathrm{Proj}_1}\\
& U_i & 
\end{tikzcd} 
\end{center}
The open cover $\left\{ U_i \right\}$ and the homeomorphisms $\left\{\phi_i\right\}$ together provides a system of \emph{local trivializations} for the fibre bundle $\mathscr{E}$. Moreover, $G$ is a topological transformation group on $F$ encoding the compatibility of ``change-of-coordinates'' on $M$, with respect to the provided local trivializations, in the following sense: at every $x\in U_i\cap U_j\neq\emptyset$, the restriction of the composed map $\phi_i^{-1}\circ\phi_j:U_j\times F\rightarrow U_i\times F$ on $\left\{ x \right\}\times F$, which necessarily gives rise to a homeomorphism from $\left\{ x \right\}\times F$ to itself by definition, is canonically identified with a group element $g_{ij} \left( x \right)\in G$, and the map $g_{ij}:U_i\cap U_j\rightarrow G$ is continuous. The topological group $G$ is called the \emph{structure group} of the fibre bundle $\mathscr{E}$. The notation $F_x$ is often used to denote $\pi^{-1}\left( x \right)$ for $x\in M$, and referred to as the \emph{fibre over $x\in M$}. It is straightforward to check from these definitions that
\begin{align}
  g_{ii} \left( x \right)=e\quad&\forall x\in U_i\label{eq:bundle-coordinate-transformation-condition-1}\\
  g_{ij} \left( x \right)=g_{ji}^{-1}\left( x \right)\quad&\forall x\in U_i\cap U_j\label{eq:bundle-coordinate-transformation-condition-2}\\
  g_{ij}\left( x \right)g_{jk}\left( x \right)=g_{ik}\left( x \right)\quad&\forall x\in U_i\cap U_j\cap U_k\label{eq:bundle-coordinate-transformation-condition-3}
\end{align}
where $e$ is the identity element of the structural group $G$. The family of continuous maps $\left\{ g_{ij}:U_i\cap U_j\rightarrow G \right\}$ is called a system of \emph{coordinate transformations for the fibre bundle} $\mathscr{E}$. Interestingly, essentially all information for determining the fibre bundle $\mathscr{E}$ is encoded in the coordinate transformations, as the following theorem indicates:
\begin{theorem}[Steenrod \cite{Steenrod1951} \S 3.2]
  If $G$ is a topological transformation group of $F$, $U_j$ is an open cover of $M$, $\left\{ g_{ij} \right\}$ is a family of continuous maps from each non-empty intersection $U_i\cap U_j$ to $G$ satisfying \eqref{eq:bundle-coordinate-transformation-condition-1}, \eqref{eq:bundle-coordinate-transformation-condition-2}, \eqref{eq:bundle-coordinate-transformation-condition-3}, then there exists a fibre bundle $\mathscr{E}$ with base space $M$, fibre $F$, structural group $G$, and coordinate transformations $\left\{ g_{ij} \right\}$. Any two such fibre bundles are equivalent to each other.
\end{theorem}
The precise definition for two fibre bundles with the same base space, fibre space, and structural group to be equivalent can be found in \cite[\S 2.4]{Steenrod1951}, but we will also cover it in Section~\ref{sec:synchr-probl}. Notice that the conditions \eqref{eq:bundle-coordinate-transformation-condition-1}, \eqref{eq:bundle-coordinate-transformation-condition-2}, \eqref{eq:bundle-coordinate-transformation-condition-3} are reminiscent of the characterization for the synchronizability \eqref{eq:vertex-pot-satisfies-edge-pot} of a $G$-valued edge potential on a connected graph: if $\rho$ satisfies \eqref{eq:vertex-pot-satisfies-edge-pot} for a map $f:V\rightarrow G$, then $\rho_{ij}=f_if_j^{-1}$ on each edge $\left( i,j \right)\in E$, which certainly satisfies
\begin{equation}
\label{eq:synchronizability-characterization}
  \rho_{ii}=e\,\,\forall i\in V,\quad\rho_{ij}=\rho_{ji}^{-1}\,\,\forall \left( i,j \right)\in E,\quad\rho_{ij}\rho_{jk}=\rho_{ik}\,\,\forall \left( i,j \right),\left( j,k \right),\left( i,k \right)\in E.
\end{equation}
As the following Proposition~\ref{prop:synchronizability_flat_bundle} establishes, viewing the graph $\Gamma$ as a topological space, an appropriate open cover of $\Gamma$ can be found such that any synchronizable edge potential can be realized as coordinate transformations of a fibre bundle with base space $\Gamma$ and the topological group $G$ serving both as the fibre space and the structure group. A fibre bundle with its structural group as fibre type is called a \emph{principal bundle}. Moreover, any such principal bundle must also be flat, as the bundle coordinate transformations take constant values on every non-empty intersection of sets in the open cover. The following simple concepts from combinatorial graph theory and algebraic topology (see, e.g. \cite{BC2008,BottTu1982}) will be needed for the statement and proof of Proposition~\ref{prop:synchronizability_flat_bundle}:
\begin{enumerate}[1)]
\item The \emph{$n$-skeleton} of a simplicial complex $\mathcal{K}$ is the subcomplex of $\mathcal{K}$ consisting of all $j$-dimensional faces for $0\leq j\leq n$;
\item The \emph{support} of a simplicial complex $\mathcal{K}$ is the underlying topological space of $\mathcal{K}$;
\item The \emph{star neighborhood} of a vertex $v$ in a simplicial complex $\mathcal{K}$ is the union of all closed simplices in $\mathcal{K}$ containing $v$ as a vertex;
\item A \emph{clique complex} of a graph $\Gamma= \left( V,E \right)$ is the simplicial complex with all complete subgraphs of $\Gamma$ as its faces.
\end{enumerate}

\begin{proposition}
\label{prop:synchronizability_flat_bundle}
  Let $G$ be a topological group, $\Gamma=\left( V,E \right)$ a connected undirected graph, and $\rho:E\rightarrow G$ a map satisfying $\rho_{ij}=\rho_{ji}^{-1}$ for all $\left( i,j \right)\in E$. Denote $X$ for the $2$-skeleton of the clique complex of the graph $\Gamma$, $\mathcal{X}$ the support of $X$, and $\mathfrak{U}=\left\{ U_i\,\mid\,1\leq i\leq \left| V \right| \right\}$ for an open cover of $\mathcal{X}$ in which $U_i$ is the interior of the star of vertex $i$. Then $\rho$ is synchronizable over $G$ if and only if there exists a flat trivial principal fibre bundle $\pi:\mathscr{P}_{\rho}\rightarrow \mathcal{X}$ with structure group $G$ and a system of local trivializations defined on the open sets in $\mathfrak{U}$ with constant bundle transition functions $\rho_{ij}$ on non-empty $U_i\cap U_j$.
\end{proposition}
A proof of Proposition~\ref{prop:synchronizability_flat_bundle} can be found in \ref{sec:appendix-proof-fibre-bundle-sync}. The key idea is to view $\Gamma$ as the $1$-skeleton of its associated clique complex, and use the open cover consisting of star neighborhoods of each vertex. A similar construction of ``Cryo-EM complex'' has been used in \cite{YeLim2017} to classify data input to Cryo-EM problems, an important application of synchronization techniques.

However, it is important to notice that the converse to Proposition~\eqref{prop:synchronizability_flat_bundle} is not true in general; more precisely, an edge potential satisfying \eqref{eq:synchronizability-characterization}, which necessarily specifies a flat principal bundle over $\Gamma$, need not be synchronizable. For a simple example, consider a square graph $\Gamma$ consisting of a four vertices $1$, $2$, $3$, $4$ and four edges $\left( 1,2 \right)$, $\left( 2,3 \right)$, $\left( 3,4 \right)$, $\left( 4,1 \right)$, forming a closed simple loop but without any triangles enclosed by three edges. An edge potential satisfying $\rho_{ij}=\rho_{ji}^{-1}$ on all edges clearly satisfies all equalities in \eqref{eq:synchronizability-characterization} since no consistency needs to be checked on edge triplets, but it is easy to find $\rho$ violating the equality $\rho_{12}\rho_{23}\rho_{34}\rho_{41}=e$ which must be obeyed by any synchronizable edge potential, provided that the group $G$ is not trivial. The lesson is that the compatibility conditions \eqref{eq:synchronizability-characterization} are of a local nature, in the sense that the \emph{cycle-consistency} (borrowing a term from geometry processing of shape collections \cite{NBWYG2011,HuangGuibas2013} that describes a compatibility constraint analogous to the last equality in \eqref{eq:synchronizability-characterization}) is imposed only on triangles composed of edge triplets; in contrast, synchronizability requires a stronger notion of ``global'' cycle-consistency for the operation of composing group elements along loops of arbitrary length and topology on he graph. In a certain sense, fibre bundles are the geometric models realizing edge potentials that are ``locally synchronizable.''

Proposition~\ref{prop:synchronizability_flat_bundle} is our first attempt at understanding the geometric mechanism of synchronization problems. The assumption of the synchronizability of $\rho$ significantly restricts the range of applicability of this geometric analogy: in most scenarios of interest, the synchronizability of an edge potential is the goal rather than the starting point for a synchronization problem. Fortunately, it is possible to extend the fibre bundle analogy beyond the synchronizability assumption in Proposition~\ref{prop:synchronizability_flat_bundle}, by restricting the model base space from the $2$-skeleton of the clique complex of the graph to the $1$-skeleton, and adjust the open cover $\mathfrak{U}$ accordingly: if we define an open cover $\mathfrak{U}$ on the graph $\Gamma$ (which as a topological space is canonically identified with the $1$-skeleton of its clique complex) in which each open set $U_i$ covers only vertex $i$ and the interior of all edges adjacent to it, then $U_i\cap U_j\neq \emptyset$ if and only if $\left( i,j \right)\in E$, and any triple intersection of open sets in $\mathfrak{U}$ is empty. In consequence, any system of bundle coordinate transformations defined on $\mathfrak{U}$ by a $G$-valued edge potential $\rho$ automatically satisfies \eqref{eq:bundle-coordinate-transformation-condition-1}, \eqref{eq:bundle-coordinate-transformation-condition-2}, \eqref{eq:bundle-coordinate-transformation-condition-3}, and specifies a flat principal $G$-bundle over $\Gamma$, denoted as $\mathscr{B}_{\rho}$, regardless of synchronizability. This is also consistent with the definition of vector bundles on graphs in \cite{Kenyon2011}. Clearly, when $\rho$ is synchronizable, $\mathscr{B}_{\rho}$ is the restriction of the principal $G$-bundle $\mathscr{P}_{\rho}$ in Proposition~\ref{prop:synchronizability_flat_bundle} to the $1$-skeleton of the base space $\Gamma$, therefore trivial as well. Conversely, if $\mathscr{B}_{\rho}$ is trivial, by \cite[\S2.10 or \S4.3]{Steenrod1951}, there exists a map $f:\Gamma\rightarrow G$ assigning a constant value $f_i$ for all points $x\in U_i$ such that $\rho_{ij}=f_if_j^{-1}$ for all $U_i\cap U_j\neq \emptyset$, which gives rise to a map $f:V\rightarrow G$ by restriction to the vertex set $V$ of $\Gamma$; this verifies all constraints in \eqref{eq:vertex-pot-satisfies-edge-pot} and establishes the synchronizability of the edge potential $\rho$. Consequently, the triviality of $\mathscr{B}_{\rho}$ and $\mathscr{P}_{\rho}$ implies each other, both are equivalent to the synchronizability of $\rho$. We summarize these observations in Proposition~\ref{prop:synchronizability_flat_bundle_one_skelenton} and formally define the \emph{synchronization principal bundle} $\mathscr{B}_{\rho}$, which will be of central importance for the geometric framework we develop in the rest of this paper.

\begin{proposition}
\label{prop:synchronizability_flat_bundle_one_skelenton}
  Let $G$ be a topological group, $\Gamma=\left( V,E \right)$ a connected undirected graph, and $\rho:E\rightarrow G$ a map satisfying $\rho_{ij}=\rho_{ji}^{-1}$ for all $\left( i,j \right)\in E$. Write $\mathfrak{U}=\left\{ U_i\,\mid\,1\leq i\leq \left| V \right| \right\}$ for an open cover of $\Gamma$ in which $U_i$ is the union of the single vertex set $\left\{ i \right\}$ with the interior of all edges adjacent to the vertex $i$. Then $\rho$ defines a flat principal $G$-bundle $\mathscr{B}_{\rho}$ over $\Gamma$ with a system of local trivializations defined on the open sets in $\mathfrak{U}$ with constant bundle transition functions $\rho_{ij}$ on non-empty $U_i\cap U_j$. Furthermore, $\rho$ is synchronizable if and only if $\mathscr{B}_{\rho}$ is trivial.
\end{proposition}

\begin{definition}[Synchronization Principal Bundle]
\label{defn:sync-bundle}
  The fibre bundle $\mathscr{B}_{\rho}$ associated with the connected graph $\Gamma$ and edge potential $\rho$ as characterized in Proposition~\ref{prop:synchronizability_flat_bundle_one_skelenton} is called a \emph{synchronization principal bundle of edge potential $\rho$ over $\Gamma$}, or a \emph{synchronization principal bundle} for short.
\end{definition}

In practice, it is often more convenient to work with $\mathscr{B}_{\rho}$ rather than $\mathscr{P}_{\rho}$, not only since non-synchronizable edge potentials are much more prevalent, but also because noisy or incomplete measurements almost always cause the observed group elements $\rho_{ij}\in G$ to be non-synchronizable. Solving for a $G$-valued synchronization solution can thus be viewed as an approach to ``denoising'' or ``filtering'' those observed transformations $\rho_{ij}$, as was already implicit in many applications \cite{BSS2013,BKS2016}. In the sense of Proposition~\ref{prop:synchronizability_flat_bundle_one_skelenton}, these problems can be interpreted as inference on the structure of flat principal bundles.

% cocycle conditions even for the geometrically simpler case $F=G$, thus becomes an implicit goal . It is thus natural to model synchronization problems, whether $G=F$ or $G\neq F$, as a filtering problem for the transition functions of a flat fibre bundle. 
 
%There is a direct connection between Proposition~\ref{prop:synchronizability_flat_bundle} and Proposition~\ref{prop:synchronizability_flat_bundle_one_skelenton}, characterized by the following observation: Since flat $G$-bundles are equivalent to $G_{\delta}$-bundles, where $G_{\delta}$ is $G$ equipped with the discrete topology, a global section of $\mathscr{B}_{\rho}$ is necessarily constant on the boundary of any $2$-simplex $\sigma\in X$ and can thus be extended over $\sigma$. Therefore, any global section of $\mathscr{B}_{\rho}$ can be extended to a global section of $\mathscr{P}_{\rho}$; conversely, a global section of $\mathscr{P}_{\rho}$ \TG{cont.}
 
Most synchronization problems in practice \cite{SingerWu2011ODM,SZSH2011,SS2012,TSR2011} considers vertex potentials valued in $G$, the same topological group in which the prescribed edge potential takes value, pertaining to the principal bundle picture (i.e. $F=G$) discussed in Proposition~\ref{prop:synchronizability_flat_bundle} and Proposition~\ref{prop:synchronizability_flat_bundle_one_skelenton}. Our fibre bundle interpretation naturally includes more general synchronization problems in which the vertex potentials takes values in $F\neq G$ as well, by relating the synchronizability of an edge potential to the existence of global sections on an \emph{associated $F$-bundle} $\mathscr{P}_{\rho}\times_\eta F$ or $\mathscr{B}_{\rho}\times_\eta F$, where $\eta: G\times F\rightarrow F$ denotes the action of $G$ on $F$. Essentially, an associated $F$-bundle $\mathscr{P}_{\rho}\times_\eta F$ (or $\mathscr{B}_{\rho}\times_\eta F$) is constructed using the same procedure as the principal bundle $\mathscr{P}_{\rho}$ (or $\mathscr{B}_{\rho}$), but with fibre $F$ instead of $G$. The associated bundles $\mathscr{P}_{\rho}\times_\eta F$, $\mathscr{B}_{\rho}\times_\eta F$ are thus also flat (but not necessarily trivial), since their bundle coordinate transformations are equivalent to those of their principal bundle, up to the representation induced by the action $\eta$.

 A major difference between working with an associated bundle and the principal bundle is that the cocycle condition $\rho_{kj}\rho_{ji}=\rho_{ki}$ may still not be satisfied in the presence of a vertex potential $f:V\rightarrow F$ satisfying \eqref{eq:vertex-pot-satisfies-edge-pot}, as elements in $F$ can not be ``inverted'' in general; another important difference is the relation between triviality and global sections: whereas a principal bundle $\mathscr{P}_{\rho}$ or $\mathscr{B}_{\rho}$ is trivial if and only if \emph{one} global section exists, which amounts to finding \emph{one} solution to the synchronization problem over $\Gamma$ with respect to $\rho$, an associated bundle may admit one or more global sections yet still be non-trivial. For instance, a vector bundle always admits the zero global section, regardless of its triviality. It turns out that establishing the synchronizability of an edge potential through the triviality of an associated bundle requires finding ``sufficiently many'' global sections of an associated bundle, or in terms of synchronization problems, ``sufficiently many'' solutions satisfying \eqref{eq:vertex-pot-satisfies-edge-pot}. This is also reflected in the twisted Hodge theory we develop in Section~\ref{sec:geom-synchr-probl}. Even though finding enough global sections seems to be more work, in practice it could be much easier to find a set of global sections on the associated bundle than to find even only one global section on the principal bundle, as the action of $G$ on the space $F$ introduces additional information from both geometric and practical points of view. Since the flat $F$-bundle associated with $\mathscr{B}_{\rho}$ will be essential for Section~\ref{sec:geom-synchr-probl}, we introduce the following definition:

 \begin{definition}[Synchronization Associated Bundle]
   The flat $F$-bundle on $\Gamma$ associated with the flat principal bundle $\mathscr{B}_{\rho}$ by the action of $G$ on $F$  is called a \emph{synchronization associated $F$-bundle of edge potential $\rho$ over $\Gamma$}, or \emph{synchronization associated bundle} for short, denoted as $\mathscr{B}_{\rho}\left[ F \right]$.
 \end{definition}

We close this preliminary section drawing analogy between synchronization problems and fibre bundles by clarifying the relation among, and the geometric implications of, some variants of the optimization formulation of synchronization problems. Given a graph $\Gamma= \left( V,E \right)$ and a $G$-valued edge potential $\rho$, a direct translation of the goal of finding an $F$-valued vertex potential $f$ satisfying \eqref{eq:vertex-pot-satisfies-edge-pot} as much as possible is to solve
\begin{equation}
  \label{eq:sync_prob_gen_type_1_special}
  \min_{f:V\rightarrow F} \sum_{\left( i,j \right)\in E} \textrm{Cost}_F\left( \rho_{ij}f_j, f_i \right),
\end{equation}
where $\mathrm{Cost}_F: F\times F\rightarrow \left[ 0,\infty \right)$ is a cost function on $F$ (e.g., derived from a distance or a norm). When we seek multiple solutions to an $F$-synchronization problem over $\Gamma$ with prescribed edge potential $\rho$, it is natural to impose the additional constraints that the solutions should be sufficiently different from each other; in the presence of a Hilbert space structure on $F$, it is convenient to impose orthogonality constraints between pairs of solutions to obtain linearly independence. With additional normalization constraints to fix the issue of identifiability, this exactly translates into the spectral relaxation algorithm in \cite{BSS2013}. If the prescribed edge potential is synchronizable, its synchronizability will be confirmed once a sufficient number of synchronization solutions can be collected, where the actual number depends on the property of the group $G$ as well as its action on $F$; if not, a synchronizable edge potential can be constructed from sufficiently many $F$-valued ``approximate solutions'' that minimize the objective function in \eqref{eq:sync_prob_gen_type_1_special} as much as possible. The case $G=F$ corresponds to the optimization problem
\begin{equation}
\label{eq:sync_prob_gen_type_1}
  \min_{f:V\rightarrow G} \sum_{\left( i,j \right)\in E} \textrm{Cost}_G\left( \rho_{ij}f_j, f_i \right).
\end{equation}
which, in the case the $\mathrm{Cost}_G$ is $G$-invariant, is equivalent to
\begin{equation}
  \label{eq:sync_prob_gen_type_2_simplified}
    \min_{\substack{f:V\rightarrow G}}\sum_{\left( i,j \right)\in E} \mathrm{Cost}_G \left( \rho_{ij}, f_if_j^{-1}\right).
\end{equation}
%The connection between the geometric pictures of these optimization formulations can be summarized as follows:
If $\rho$ is synchronizable, a minimizer of \eqref{eq:sync_prob_gen_type_1} (resp. \eqref{eq:sync_prob_gen_type_1_special}) attaining zero objective value can be geometrically realized as a global section of the synchronization principal bundle $\mathscr{B}_{\rho}$ (resp. $\mathscr{P}_{\rho}$); such a minimizer implies the triviality of the principal bundle $\mathscr{B}_{\rho}$ (resp. $\mathscr{P}_{\rho}$), but not necessarily so in general for the associated bundle $\mathscr{B}_{\rho}\times_{\eta}F$ (resp. $\mathscr{P}_{\rho}\times_{\eta}F$). If $\rho$ is not synchronizable, the minimum values of \eqref{eq:sync_prob_gen_type_1}, \eqref{eq:sync_prob_gen_type_2_simplified}, and \eqref{eq:sync_prob_gen_type_1_special} are all greater than zero, and  minimizer of \eqref{eq:sync_prob_gen_type_1} or \eqref{eq:sync_prob_gen_type_2_simplified} can be viewed as a ``denoised'' or ``filtered'' version of a trivial flat principal bundle underlying the dataset.

\subsection{Main Contributions}
\label{sec:main-results}

In this section we give a brief overview of our main contribution. We will motivate the two ingredients of the geometric framework developed in Section~\ref{sec:geom-coho-sync}, namely, holonomy representation and Hodge theory, by demonstrating preliminary versions of our formulation that lead to weaker conclusions or incomplete geometric pictures, then sketch the full approaches adopted in Section~\ref{sec:geom-coho-sync}. Finally, we draw the link between the geometric framework and the proposal of the \emph{learning group actions} (LGA) problem.

\subsubsection{Holonomy of Synchronization Principal Bundles}
\label{sec:holon-synchr-bundl}
%\emph{Holonomy of Synchronization Principal Bundles.}
Consider $\mathscr{B}_{\rho}$, the synchronization principal bundle arising from a $G$-synchronization problem over a connected graph $\Gamma=\left( V,E \right)$ with respect to $\rho\in C^1 \left( \Gamma;G \right)$. Fix an arbitrary vertex $v\in V$, and denote the set of all $v$-based loops in $\Gamma$ (loops with $v$ as both the starting and ending vertex) as $\Omega_v$; $\Omega_v$ carries a natural group structure by the composition of $v$-based loops. Now the procedure of taking the product of the values of $\rho$ along the consecutive edges in the loop specifies a group homomorphism from $\Omega_v$ into $G$. Denote the image of this group homomorphism by $H_v$, which is necessarily a finitely generated subgroup of $G$ since $\Gamma$ is a finite graph. The simple but important observation here is that $H_v$ is the trivial subgroup of $G$ if and only if $\rho$ is synchronizable. The group $H_v$ is the analogy of the \emph{holonomy group based at $v$} in differential geometry, if we view $\rho_{ij}$ on edge $\left( i,j \right)$ as the \emph{parallel-transport} between fibres of $\mathscr{B}_{\rho}$ at $i,j\in V$.

Section~\ref{sec:synchr-probl} is devoted to a deeper and more systematic treatment of the group homomorphism from loops in $\Gamma$ to the structure group $G$. We will define $\Hol_{\rho}\left( \Gamma \right)$, the holonomy of the synchronization principal bundle $\mathscr{B}_{\rho}$ (independent of the choice of the base vertex $v$), as well as an equivalence relation on $C^1 \left( \Gamma; G \right)$ induced by a right action of $C^0 \left( \Gamma; G \right)$ (which is treated implicitly when solving synchronization problems in practice), and establish a correspondence between $\Hol_{\rho}\left( \Gamma \right)$ and the equivalence class in $C^1 \left( \Gamma; G \right)/C^0 \left( \Gamma; G \right)$ to which $\rho \in C^1 \left( \Gamma; G \right)$ belongs. In particular, trivial holonomy $\Hol_{\rho}\left( \Gamma \right)$ corresponds to the orbit in $C^1 \left( \Gamma; G \right)/C^0 \left( \Gamma; G \right)$ consisting precisely of all synchronizable edge potentials. This correspondence will be formulated in Theorem~\ref{lem:holonomy_homomorphism_bijectivity} as between $C^1 \left( \Gamma;G \right)/C^0 \left( \Gamma;G \right)$ and $\Hom \left( \pi_1 \left( \Gamma \right), G \right)/G$, the \emph{representation variety} of the fundamental group of $\Gamma$ into $G$.

\subsubsection{Twisted De Rham Cohomology of Synchronization Associated Vector Bundles}
\label{sec:twisted-de-rham}

The \emph{graph connection Laplacian} (GCL) for an $F$-synchronization problem over graph $\Gamma$ with respect to $\rho\in\Omega^1 \left( \Gamma; G \right)$ is a linear operator on $\Omega^0 \left( \Gamma; F \right)$ defined as
\begin{equation*}
  \left(L_1f\right)_i:=\frac{1}{d_i}\sum_{j:\left( i,j \right)\in E}w_{ij} \left( f_i-\rho_{ij}f_j \right),\quad\forall i\in V,\,\,\forall f\in C^0 \left( \Gamma; F \right).
\end{equation*}
If $F$ is a vector space and $G$ has a matrix representation on $F$, GCL can be written as a block matrix in which the $\left( i,j \right)$-th block is the matrix representation of $\rho_{ij}$, if $\left( i,j \right)\in E$. GCL essentially carries all information of a synchronization problem and is of central importance to the spectral and SDP relaxation algorithms for synchronization. Our motivation for Section~\ref{sec:geom-synchr-probl} was to provide a cohomological interpretation for GCL, in the hope of realizing it geometrically as a Hodge Laplacian in a cochain complex, inspired by a similar geometric realization of the \emph{graph Laplacian} (in the context of algebraic and spectral graph theory) as a Hodge Laplacian on degree-zero forms in discrete Hodge theory (see \ref{sec:graph-lapl-discr}). Note that GCL reduces to the graph Laplacian if the group $G$ is a scalar field.

In the literature of differential geometry, \emph{twisted differential forms} on a flat vector bundle $\mathscr{E}$ can be intuitively thought as bundle-valued differential forms on the base manifold. The twisted Hodge theory we develop in Section~\ref{sec:geom-synchr-probl} defines two discrete differential operators that are formal adjoints of each other between \emph{constant twisted local $0$-forms} and \emph{constant twisted local $1$-forms} on the synchronization associated $F$-bundle $\mathscr{B}_{\rho}\left[ F \right]$ , namely the \emph{$\rho$-twisted differential $d_{\rho}$} and the \emph{$\rho$-twisted codifferential $\delta_{\rho}$}, such that GCL can be written as the composition $\delta_{\rho}d_{\rho}$. Provisionally, by identifying each $f\in C^0 \left( \Gamma; F \right)$ naturally with a collection of constant twisted local $0$-forms --- one for each open set $U_i\in\mathfrak{U}$ --- a coarse approximation of our construction can be written as
\begin{equation*}
  \begin{aligned}
    &\left(d_{\rho}f\right)_{ij}\sim f_i-\rho_{ij}f_j,\quad\forall f\in C^0 \left( \Gamma; F \right),\\
    &\left(\delta_{\rho}\omega\right)_i\sim\frac{1}{d_i}\sum_{j:\left( i,j \right)\in E}w_{ij}\omega_{ij},\quad\forall\omega\in C^1 \left( \Gamma; F \right):=\left\{ \omega:E\rightarrow F \mid \omega_{ij}=-\omega_{ji}\,\,\forall \left( i,j \right)\in E\right\},
  \end{aligned}
\end{equation*}
from which it can be easily checked that $L_1=\delta_{\rho}d_{\rho}$ on $C^0 \left( \Gamma; F \right)$. The main conceptual difficulty with this natural formulation is that ``$d_{\rho}f$'' defined as such does not possess the skew-symmetry desired for $1$-forms, since in general
\begin{equation}
\label{eq:violation-skew-symmetry}
  f_j-\rho_{ji}f_j=-\rho_{ji}\left( f_i-\rho_{ij}f_j \right)\neq -\left( f_i-\rho_{ij}f_j \right).
\end{equation}

The framework we develop in Section~\ref{sec:geom-synchr-probl} circumvent this skew-symmetry issue with $1$-forms by defining $f_i-\rho_{ij}f_j$ as the representation of $d_{\rho}f$, a \emph{twisted global $1$-form} defined over the entire graph $\Gamma$, in the system of local trivializations of $\mathscr{B}_{\rho}\left[ F \right]$ over the open cover $\mathfrak{U}$. We then define the \emph{$\rho$-twisted codifferential $\delta_{\rho}$} that is the formal adjoint of $d_{\rho}$ with respect to inner products naturally specified on the space of twisted local $0$- and $1$-forms, and realize the graph connection Laplacian $L_1$ as the degree-zero Hodge Laplacian $\delta_{\rho}d_{\rho}:C^0 \left( \Gamma; G \right)\rightarrow C^0 \left( \Gamma; G \right)$ in the twisted de Rham cochain complex \eqref{eq:twisted_chain_complex_with_adjoint}. These constructions lead to two different characterizations of the synchronizability of $\rho$, one in Proposition~\ref{prop:twisted-diff-ker-global-sections} with a twisted de Rham cohomology group, and the other through a Hodge-type decomposition of $C^0 \left( \Gamma; F \right)$ following Theorem~\ref{thm:hodge_decomp}. This twisted Hodge theory also provides geometric insights for the GCL-based spectral relaxation algorithm and Cheeger-type inequalities in \cite{BSS2013}, as we will elaborate in Section~\ref{sec:cheeg-ineq-twist}.

\subsubsection{Learning Group Actions via Synchronizability}
\label{sec:learn-group-acti-1}

Fibre bundles are topological spaces that are product spaces locally but not necessarily globally. However, we can still look for maximal open subsets of the base space on which the fibre bundle is trivializable, and seek a decomposition of the base space into the union of such ``maximal trivializable subsets.'' This intuition motivated us to consider applying synchronization techniques to partition a graph into connected components, based on the synchronizability of a prescribed edge potential in addition to the connectivity of the graph. In Section~\ref{sec:motiv-gener-form}, we define the general problem of \emph{learning group actions} (LGA) for a set $X$, equipped with an action by group $G$, as searching for a partition of $X$ into a specified number of subsets and learning a new group action on $X$ that is cycle-consistent within each partition; the cycle-consistency need not be maintained for a cycle of actions across multiple partitions.
%If such a partition plus cycle-consistent group actions can be found, the new group actions within each partition generate subgroups of $G$.
The LGA problem is then specialized to the setting of synchronization problems (\emph{learning group actions by synchronization}, or LGAS), for which we define a quantity that measures the performance of graph partitions based on the synchronizability of a fixed edge potential on the entire graph $\Gamma$, motivated by the classical normalized graph cut algorithm. Finally, we propose in Section~\ref{sec:heur-algor-learn} a heuristic algorithm for LGAS, building upon iteratively applying existing synchronization techniques hierarchically and performing spectral clustering on the edge-wise frustration.

\subsection{Broader Context and Related Work}
\label{sec:other-related-work}

The synchronization problem has been studied for a variety of choices of topological groups $G$ and spaces $F$. Typically these formulations fall into our principal bundle setting and require an underlying  manifold structure. We give a brief summary for the correspondences between choices of $G$ and $F$ in our framework and the practical instances in the synchronization literature: In \cite{BSS2013} $G=F=O \left( d \right)$ and $G=O \left( d \right)$, $F=\mathbb{S}^{d-1}$ are studied; the case $G=F=\SO \left( d \right)$ is examined in \cite{BSAB2014,WangSinger2013}; orientation detection or $G=F=O \left( 1 \right)$ is considered in \cite{SingerWu2011ODM}; cryo-electron microscopy concerns $G=F=\SO \left( 2 \right)$ \cite{SS2012,SZSH2011}; globally aligning three-dimensional scans is the case where $G=F=\SO \left( 3 \right)$, and so is \cite{TSR2011}. 

Our formulation of the synchronization problem considers a broader class of geometric structure than what has been proposed in the literature. Specifically, we do not require a manifold assumption (as the problem is modeled on topological spaces), or a principal bundle structure (as we can work with any associated bundle), or compact and/or commutative structure groups. For comparison, the vector and principal bundle framework developed in \cite{SingerWu2012VDM,SingerWu2013} relies on manifold assumptions for the base, fibre, and total space, as well as an (extrinsic) isometric embedding into an ambient Euclidean space for locally estimating tangent spaces and parallel-transports; similarly for recent work \cite{HDM2016} extending this geometric framework to smooth bundles with general fibre types. Both Vector Diffusion Maps (VDM) \cite{SingerWu2012VDM} and Horizontal Diffusion Maps (HDM) \cite{HDM2016} can be viewed as attempts at combining the idea of synchronization with diffusion geometry \cite{CoifmanLafon2006,CoifmanLafonLMNWZ2005PNAS1,CoifmanLafonLMNWZ2005PNAS2}. The geometry underlying the synchronization problem related to cryo-electron microscopy \cite{SS2012,SZSH2011} can be described using the language of $\SO \left( 2 \right)$-principal bundles, as recently demonstrated in \cite{YeLim2017}, with a \v{C}ech cohomology approach through Leray's Theorem which depends essentially on the commutativity of the structure group $\SO \left( 2 \right)$, whereas most synchronization problems of practical interest involve noncommutative structure groups. The Non-Unique Games (NUG) and SDP relaxation framework established in \cite{BCS2015NUG,BCSZ2014} assumes the compactness of the structure group $G$, and resorts to a compactification procedure that maps a subset of $G$ to another compact group for synchronization problems over non-compact groups such as the Euclidean group in the motion estimation problem in computer vision \cite{MP2007,HTDL2013}.

The graph twisted Hodge theory we develop in Section \ref{sec:geom-synchr-probl} also has ties to recent developments in discrete Hodge theory \cite{JLYY2011,Lim2015,MS2016,PR2016,PRT2015,SKM2014}. In \cite{Kenyon2011} Laplacians on one- and two-dimensional vector bundles on graphs were used to understand the relation between graphs embedded on surfaces and cycle rooted spanning forests, generalizing the relation between spanning trees and graph Laplacians. Variants of the graph Laplacian have been used to relate ranking problems to synchronization problems in \cite{Cucuringu,Fanuel1}; a combinatorial Laplacians based on a discrete Hodge theory on directed graphs has been successfully applied to decompose ranking problems and games into ``gradient-like'' versus ``cyclic'' components in \cite{JLYY2011,Candogan}, and to visualize directed networks in \cite{Fanueletal}. Discrete Laplacians on simplicial complexes have been proposed, and spectral properties such as Cheeger inequalities and stationary distributions of random walks have been examined in a series of papers \cite{JLYY2011,Lim2015,MS2016,PR2016,PRT2015,SKM2014}. The cocycle conditions \eqref{eq:synchronizability-characterization} are also characterized in geometry processing and computer vision recently for analysis of shape or image collections \cite{NBWYG2011,HuangGuibas2013,WangHuangGuibas2013,HWG2014}, where they are also known as \emph{cycle-consistency conditions}.

% are also prevalent in theoretical and applied geometry, such as for developing a homological algebra for complexes of groups in \cite{Haefliger}

The geometric and topological tools we utilize in this paper, namely those involving the topology and geometry of fibre bundles, are covered in most standard textbooks, e.g. \cite{Steenrod1951,Taubes2011DG,BottTu1982}. After Milnor's seminal work on flat connections on a Riemannian manifold \cite{Milnor1958}, the relation between flat bundles and their holonomy homomorphisms became widely known \cite{KamberTondeur1967,Lue1976,Goldman1982,Vassiliou1983,Corlette1988} and is still attracting interests of modern mathematical physicists (e.g. Higgs bundles and representation of the fundamental group \cite{Hitchin1987,Simpson1992}). In a completely topological setup, flat bundles can be characterized as fibrations with a homotopy-invariant lifting property (a topological analogue of parallel-transport in differential geometry); essentially the same correspondence between flat bundles and holonomy homomorphisms is already known to Steenrod \cite{Steenrod1951} and referred to as characteristic classes of flat bundles \cite{Dupont1976,Lue1976,Goldman1982,Esnault1988}. In a broader context, the correspondence between flat bundles (integrable connections) and local systems (locally constant sheaves) is a special case of the \emph{Riemann-Hilbert correspondence}, a higher-dimensional generalization of Hilbert's twenty-first problem \cite{Bolibrukh1990,AnosovBolibruch1994}. This correspondence fostered important developments in algebraic geometry, including $D$-modules \cite{Kashiwara1979,Kashiwara1984,Mebkhout1980,Mebkhout1984} and Deligne's work on integrable algebraic connections \cite{Deligne1970}. Understanding the representation varieties of the fundamental groups of Riemann surfaces into Lie groups has been of interest to algebraic geometers, geometric topologists, and representation theorists in the past decades \cite{Labourie2013,Goldman2006}.

The rest of this paper is organized as follows. Section~\ref{sec:geom-coho-sync} establishes the geometric framework for synchronization problems, relating the synchronizability of an edge potential $\rho$ to (a) the triviality of the holonomy of the flat principal bundle $\mathscr{B}_{\rho}$, in Section~\ref{sec:synchr-probl}; (b) the dimension of the zero-th degree twisted cohomology group of a $\rho$-twisted de Rham cochain complex, as well as the dimension of the kernel of the zero-th degree twisted Hodge Laplacian, in Section~\ref{sec:geom-synchr-probl}. Section~\ref{sec:learn-group-acti} defines the problem of learning group actions, and proposes SynCut, a heuristic algorithm based on synchronization and graph spectral techniques. Numerical simulations indicating the effectiveness of SynCut is performed on synthetic datasets in Section~\ref{sec:numer-exper-synth} and on a real dataset of a collection of anatomical surfaces in Section~\ref{sec:example}. A few problems of potential interest are listed in Section~\ref{sec:concl-discu} for future exploration.

\section{Synchronization as a Cohomology Problem}
\label{sec:geom-coho-sync}

This section concerns two geometric aspects of the synchronization problem. Section~\ref{sec:synchr-probl} links the synchronizability of an edge potential to the triviality of the holonomy group of a flat principal bundle. Section~\ref{sec:geom-synchr-probl} establishes a discrete twisted Hodge theory that naturally realizes the graph connection Laplacian as the lowest-order Hodge Laplacian of a twisted de Rham cochain complex. The obstruction to synchronizability of an edge potential turns out to be a cohomology group in the twisted de Rham complex; the degeneracy of this cohomology group is reflected in the spectral information of the twisted Hodge Laplacian, which also provides a geometric interpretation for the relaxation techniques used in solving synchronization problems.

\subsection{Holonomy and Synchronizability}
\label{sec:synchr-probl}

The two main results in this section, Corollary~\ref{cor:cor-trivial-hol} and Theorem~\ref{lem:holonomy_homomorphism_bijectivity}, relate the synchronizability of an edge potential $\rho$ to the triviality of the holonomy group of the synchronization principal bundle $\mathscr{B}_{\rho}$. Our motivation is as follows. Recall from Proposition~\ref{prop:synchronizability_flat_bundle_one_skelenton} that $G$-synchronization problems on a fixed graph $\Gamma$ with different edge potentials are in one-to-one correspondence with flat principal $G$-bundles over $\Gamma$, and the synchronizability of an edge potential translates into the triviality of the bundle; whereas the one-to-one correspondence is stated at the level of local coordinates in Proposition~\ref{prop:synchronizability_flat_bundle_one_skelenton}, the triviality of the principal bundle is a property of the equivalence classes of flat principal G-bundles, which suggests the same level of abstraction for synchronizability. As will be precisely stated later in this subsection, (appropriately defined) equivalence classes of edge potentials form the \emph{moduli space} of flat $G$-bundles on $\Gamma$, and a given edge potential is synchronizable if and only if it belongs to the same equivalence class as the trivial edge potential that assigns each edge of $\Gamma$ the identity element $e\in G$. The \emph{holonomy group}, or the equivalence classes of \emph{holonomy homomorphisms} from the fundamental group of $\Gamma$ to the structure group $G$, is a faithful representation of the moduli space of flat principal $G$-bundles on $\Gamma$; we thus will be able to detect the synchronizability of an edge potential through the triviality of the associated holonomy group. This argument is reminiscent of classical classification theorems of (1) principal bundles with disconnected structure groups in topology (see e.g. \cite[\S13.9]{Steenrod1951}); (2) flat connections in differential geometry (see e.g. \cite[\S13.6]{Taubes2011DG}); and (3) holomorphic vector bundles of fixed rank and degree (see e.g. \cite[Appendix \S2.1]{Wells2007GTM65}) in complex geometry.

For $f\in C^0 \left( \Gamma; G \right)$, $\rho\in C^1 \left( \Gamma; G \right)$, we say that \emph{$f$ and $\rho$ are compatible on edge $\left( i,j \right)\in E$} if $f_i=\rho_{ij}f_j$, and that \emph{$f$ and $\rho$ are compatible on graph $\Gamma$} if they are compatible on every edge in $\Gamma$. Recall from Definition~\ref{defn:sync-bundle} that we write $\mathscr{B}_{\rho}$ for the synchronization principal bundle associated with $\rho\in C^1 \left( \Gamma;G \right)$, as described in Proposition~\ref{prop:synchronizability_flat_bundle_one_skelenton}. Equivalently, it is often convenient to view $\mathscr{B}_{\rho}$ as a $G_{\delta}$-bundle on $\Gamma$, where $G_{\delta}$ is the same group as $G$ but equipped with the discrete topology. For $\rho,\tilde{\rho}\in\Gamma$, the $G_{\delta}$-bundles $\mathscr{B}_{\rho}$, $\mathscr{B}_{\tilde{\rho}}$ are \emph{equivalent}, denoted as $\mathscr{B}_{\rho}\sim\mathscr{B}_{\tilde{\rho}}$, if a \emph{bundle map} (see \cite[\S2.5]{Steenrod1951}) exists between $\mathscr{B}_{\rho}$ and $\mathscr{B}_{\tilde{\rho}}$ that induces the identity map on the base space $\Gamma$. Since $\mathscr{B}_{\rho}, \mathscr{B}_{\tilde{\rho}}$ have the same base space, fibre, and structure group, recall from \cite[Lemma 2.10]{Steenrod1951} that they are equivalent if and only if there exist continuous functions $\lambda_i:U_i\rightarrow G_{\delta}$ defined on each $U_i\in\mathfrak{U}$ such that
\begin{equation*}
  \tilde{\rho}_{ij}=\lambda_i \left( x \right)^{-1}\rho_{ij}\,\lambda_j \left( x \right), \quad \forall x\in U_i\cap U_j\neq \emptyset.
\end{equation*}
Since the topology on $G_{\delta}$ is discrete and $U_i$ is connected, $\lambda_i$ is constant on $U_i$, and defines a vertex potential by setting $f_i:=\lambda_i \left( v_i \right)$, where $v_i\in U_i$ is the $i$th vertex of $\Gamma$. This proves the following lemma:
\begin{lemma}
\label{lem:equiv-class-characterization}
Two edge potentials $\rho,\tilde{\rho}\in C^1 \left( \Gamma;G \right)$ define equivalent flat principal $G$-bundles on $\Gamma$ if and only if there exists $f\in C^0 \left( \Gamma;G \right)$ such that
\begin{equation}
  \label{eq:bundle_equivalence}
  \tilde{\rho}_{ij}=f_i^{-1}\rho_{ij}\,f_j,\quad \forall \left( i,j \right)\in E.
\end{equation}
In other words, equivalence classes of flat principal $G$-bundles on $\Gamma$ determined by edge potentials (through Proposition~\ref{prop:synchronizability_flat_bundle_one_skelenton}) are in one-to-one correspondence with equivalence classes in the orbit space $C^1 \left( \Gamma;G \right)/C^0 \left( \Gamma;G \right)$, where the right action of $C^0 \left( \Gamma;G \right)$ on $C^1 \left( \Gamma;G \right)$ is defined as
\begin{equation}
  \label{eq:vertpot_action_edgepot}
  \left[f\left(\rho\right)\right]_{ij}:=f_i^{-1}\rho_{ij}\,f_j,\quad\forall \left( i,j \right)\in E.
\end{equation}
\end{lemma}

\begin{remark}
  The orbit space $C^1 \left( \Gamma;G \right)/C^0 \left( \Gamma;G \right)$ is exactly the \emph{first cohomology set} $\check{H}^1 \left( \left(\Gamma, \mathfrak{U}\right), \underline{G} \right)$ for the sheaf of germs of locally constant $G$-valued functions over $\Gamma$ with respect to the open cover $\mathfrak{U}$, where $G$ is possibly nonabelian. It is thus not surprising that the orbit space should identify naturally with isomorphism classes of flat principal $G$-bundles over $\Gamma$ (see e.g. \cite[Proposition 4.1.2]{Brylinski2007} or \cite[\S 8.1]{Michor2008}).
\end{remark}

A \emph{path} in $\Gamma$ is a collection of consecutive edges in $\Gamma$. If all edges in path $\gamma$ are oriented consistently, we say $\gamma$ is an oriented path. For any oriented path $\gamma$, define $\gamma^{-1}$ the \emph{reverse} of $\gamma$ as the path in $\Gamma$ consisting of the same consecutive edges in $\gamma$ listed in the opposite order and with all orientations reversed. For an oriented path $\gamma$ consisting of consecutive edges $\{\left( i_0,i_1 \right), \left( i_1, i_2 \right), \cdots, \left( i_{N-1}, i_N \right)\}$ set
\begin{equation}
\label{eq:holonomy_map_defn}
  \hol_{\rho} \left( \gamma \right) = \left(\rho_{i_N,i_{N-1}}\rho_{i_{N-1},i_{N-2}}\cdots\rho_{i_2,i_1}\rho_{i_1,i_0}\right)^{-1}=\rho_{i_0,_1}\rho_{i_1,i_2}\cdots\rho_{i_{N-2},i_{N-1}}\rho_{i_{N-1},i_N} \in G,
\end{equation}
then $\hol_{\rho}$ maps paths in $\Gamma$ to elements of group $G$. For two oriented paths $\gamma$, $\gamma'$ such that the ending vertex of $\gamma$ coincides with the starting vertex of $\gamma'$, define $\gamma\circ \gamma'$ as the oriented path constructed by concatenating $\gamma'$ with $\gamma$. It is then straightforward to verify by definition that
\begin{equation}
  \label{eq:path_group_op}
  \hol_{\rho} \left( \gamma^{-1} \right)= \hol_{\rho} \left( \gamma \right)^{-1},\qquad \hol_{\rho}\left( \gamma\circ\gamma' \right)= \hol_{\rho} \left( \gamma \right)\hol_{\rho} \left( \gamma' \right).
\end{equation}
If an oriented path starts and ends at the same vertex $v$, we call it an \emph{oriented loop based at vertex $v$}. Denote $\Omega_v$ for all loops based at $v\in V$ in $\Gamma$, including the single vertex set $\left\{ v \right\}$ viewed as the identity loop based at $v$. Clearly, $\Omega_v$ carries a group structure with the loop concatenation and reversion operations. The equalities in \eqref{eq:path_group_op} ensures $\hol_{\rho} \left( \left\{ v \right\} \right)=e$ and that $\hol_{\rho}:\Omega_v\rightarrow G$ is a group homomorphism. Moreover, since graph $\Gamma$ does not contain any $2$-simplices, two oriented loops based at $v$ are homotopic if and only if they differ by a collection of disconnected trees in $\Gamma$, in which every tree gets mapped to $e\in G$ under the map $\hol_{\rho}$; the map $\hol_{\rho}:\Omega_v\rightarrow G$ thus descends naturally to a map to $G$ from $\pi_1 \left( \Gamma, v \right)$, the \emph{fundamental group of $\Gamma$ based at $v$}. Unless confusions arise, we shall also denote the descended map as $\hol_{\rho}$ for simplicity of notation. Lemma~\ref{lem:holonomy_group} below summarizes these discussions.

\begin{lemma}
\label{lem:holonomy_group}
  The map $\hol_{\rho}:\pi_1 \left( \Gamma, v \right)\rightarrow G$ defined in \eqref{eq:holonomy_map_defn} is a group homomorphism. In particular, the image of this homomorphism is a subgroup of $G$. % $\Hol_{\rho} \left( \pi_1 \left( \Gamma,v \right) \right)$ is a subgroup of $G$.
\end{lemma}
We will refer to the group homomorphism $\hol_{\rho}:\pi_1 \left( \Gamma, v \right)\rightarrow G$ as the \emph{holonomy homomorphism at $v\in \Gamma$} for a $G$-synchronization problem with prescribed edge potential $\rho\in C^1 \left( \Gamma;G \right)$. Define the \emph{holonomy group at $v\in \Gamma$} of edge potential $\rho$ as the image
\begin{equation*}
  \Hol_{\rho}\left( v \right):=\hol_{\rho} \left( \pi_1 \left( \Gamma,v \right) \right).
\end{equation*}
From a different point of view, Lemma~\ref{lem:holonomy_group} assigns an element of $\Hom \left( \pi_1 \left( \Gamma,v \right), G \right)$ to each element of $C^1 \left( \Gamma;G \right)$, where $\Hom \left( \pi_1 \left( \Gamma,v \right), G \right)$ is the set of group homomorphisms from $\pi_1 \left( \Gamma,v \right)$ to $G$.

\begin{lemma}
\label{lem:lem-holonomy-iso}
If $\Gamma$ is connected, the holonomy groups $\Hol_{\rho}\left( v \right)$, $\Hol_{\rho}\left( w \right)$ at $v,w\in V$ are conjugate to each other as subgroups of $G$.
\end{lemma}
\begin{proof}
Let $\gamma$ be a path in $\Gamma$ connecting vertex $v$ to vertex $w$. The fundamental groups of $\Gamma$ based at $v,w$ are related by conjugation $\pi_1 \left( \Gamma,v \right)=\gamma^{-1}\pi_1 \left( \Gamma,w \right)\gamma$, thus
  \begin{equation*}
    \Hol_{\rho} \left( v \right) = \Hol_{\rho} \left( \gamma^{-1}\pi_1 \left( \Gamma,w \right)\gamma \right) = \Hol_{\rho} \left( \gamma^{-1}\right)\Hol_{\rho}\left(\pi_1 \left( \Gamma,w \right)\right)\Hol_{\rho}\left(\gamma \right) = \Hol_{\rho} \left( \gamma\right)^{-1}\Hol_{\rho}\left(w\right)\Hol_{\rho}\left(\gamma \right).
  \end{equation*}
\end{proof}
Define the \emph{holonomy} of $\rho\in C^1 \left( \Gamma; G \right)$ on a connected graph $\Gamma$ as the following conjugacy class (orbit of the action by conjugation) of subgroups of $G$:
\begin{equation}
  \label{eq:defn_holonomy_group_for_graph}
  \Hol_{\rho} \left( \Gamma \right):=\left\{ g^{-1}\Hol_{\rho} \left( v \right)g\,\big|\,\textrm{for all }g\in G, \textrm{and an arbitrarily chosen but fixed vertex }v\in V\right\}.
\end{equation}
By Lemma~\ref{lem:lem-holonomy-iso}, the definition of $\Hol_{\rho} \left( \Gamma \right)$ is independent of the choice of a fixed base $v\in V$. We say that the holonomy of $\rho\in C^1 \left( \Gamma;G \right)$ is \emph{trivial} on a connected graph $\Gamma$ if $\Hol_{\rho} \left( \Gamma \right)$ contains only the trivial subgroup $\left\{ e \right\}$ for all $g\in G$. Under the connectivity assumption of $\Gamma$, the triviality of the global invariant $\Hol_{\rho} \left( \Gamma \right)$ can be completely determined by its seemingly ``local'' counterparts; see Lemma~\ref{lem:triviality_global_holonomy} below. Of course, holonomy is not local in nature, as $\Hol_{\rho} \left( v \right)$ encodes the information of all oriented loops based at vertex $v$ and in principle ``touches'' the entire space $\Gamma$.
\begin{lemma}
\label{lem:triviality_global_holonomy}
  If $\Gamma$ is connected, the following statements are equivalent:
\begin{enumerate}[(i)]
\item $\Hol_{\rho} \left( \Gamma \right)$ is trivial;
\item $\Hol_{\rho} \left( v \right)= \left\{ e \right\}$ for some vertex $v\in V$;
\item $\Hol_{\rho} \left( v \right)= \left\{ e \right\}$ for all vertices $v\in V$.
\end{enumerate}
\end{lemma}

Similar to the definition of $\Hol_{\rho}\left( \Gamma \right)$ in \eqref{eq:defn_holonomy_group_for_graph}, the fundamental group $\pi_1 \left( \Gamma \right)$ of a connected graph $\Gamma$ is also determined by the fundamental group $\pi_1 \left( \Gamma,v_0 \right)$ at any vertex $v_0\in V$ up to conjugacy classes. Therefore, Lemma~\ref{lem:holonomy_group} and Lemma~\ref{lem:lem-holonomy-iso} together assign to each $\rho\in C^1 \left( \Gamma;G \right)$ an equivalence class in $\Hom \left( \pi_1 \left( \Gamma \right), G \right)/G$, in which $G$ acts on $\Hom \left( \pi_1 \left( \Gamma \right), G \right)$ by the inner automorphisms of $G$
\begin{equation*}
  \phi\mapsto g^{-1}\phi g,\quad\forall \phi\in \Hom \left( \pi_1 \left( \Gamma \right), G \right), g\in G.
\end{equation*}
In other words, Lemma~\ref{lem:holonomy_group} and Lemma~\ref{lem:lem-holonomy-iso} guarantee a well-defined map $\Hol:\Omega^1 \left( \Gamma; G \right)\rightarrow\Hom \left( \pi_1 \left( \Gamma \right), G \right)/G$. Furthermore, note in equation \eqref{eq:holonomy_map_defn} that $\Hol$ is invariant under the right action \eqref{eq:vertpot_action_edgepot} of $ C^{0}\left(\Gamma;G\right)$ on $C^1 \left( \Gamma; G \right)$,
% \begin{equation}
%   \rho_{ij}\mapsto f_i^{-1}\rho_{ij}f_j,\qquad \forall f\in C^{0}\left(\Gamma;G\right), \rho\in C^{1} \left( \Gamma;G \right),\left( i,j \right)\in E,
% \end{equation}
thus $\Hol$ naturally descends to a map from $C^{1} \left( \Gamma;G \right)/ C^{0}\left(\Gamma;G\right)$ to $\Hom \left( \pi_1 \left( \Gamma \right), G \right)/G$. The space $\Hom \left( \pi_1 \left( \Gamma \right), G \right)/G$ is known as the \emph{representation variety} of the fundamental group of $\Gamma$ (the free product of a finite number of copies of $\mathbb{Z}$) into $G$. To simplify the exposition, we shall use the same notation $\Hol$ to denote its quotient map induced by the canonical projection $C^{1} \left( \Gamma;G \right)\rightarrow C^{1} \left( \Gamma;G \right)/ C^{0}\left(\Gamma;G\right)$. Theorem~\ref{lem:holonomy_homomorphism_bijectivity} below establishes the bijectivity of the quotient map.

\begin{theorem}
\label{lem:holonomy_homomorphism_bijectivity}
If $\Gamma$ is connected, the map $\Hol:C^{1} \left( \Gamma;G \right)/ C^{0}\left(\Gamma;G\right)\rightarrow \Hom \left( \pi_1 \left( \Gamma \right), G \right)/G$ defined as $\Hol \left( \left[ \rho \right] \right)=\left[ \hol_{\rho} \right]$ is bijective. Moreover, $\Hom \left( \pi_1 \left( \Gamma \right), G \right)/G$ is in one-to-one correspondence with equivalence classes of flat principal $G$-bundles $\mathscr{B}_{\rho}$ with $\rho\in C^{1} \left( \Gamma;G \right)$.
\end{theorem}
\begin{proof}
  We construct an inverse of $\Hol$ from $\Hom \left( \pi_1 \left( \Gamma \right), G \right)/G$ back to $C^{1} \left( \Gamma;G \right)/ C^{0}\left(\Gamma;G\right)$. Fix an arbitrary vertex $v_0\in \Gamma$, and let $\chi:\pi_1 \left( \Gamma,v_0 \right)\rightarrow G$ be a group homomorphism. By the connectivity of $\Gamma$, each vertex $v_i\in V$ of $\Gamma$ is connected to $v_0$ through an oriented path $\gamma_{0i}$; we orient these paths so they all start at vertex $v_0$, and enforce $\gamma_{00}=\left\{ v_0 \right\}$. Assign to each edge $\left( i,j \right)\in E$ an element $\rho_{ij}$ of the group $G$ defined by
  \begin{equation}
  \label{eq:inverse_hol_homomorphism}
    \rho_{ij} := \chi \left( \gamma_{0i}\circ \left( i,j \right)\circ \gamma_{0j}^{-1} \right).
  \end{equation}
Clearly, $\rho_{ij}=\rho_{ji}^{-1}$ follows from the fact that $\chi$ is a group homomorphism; so does $\rho_{ii}=e$ for all vertices $v_i\in V$. Of course, an edge potential $\rho$ defined as in \eqref{eq:inverse_hol_homomorphism} depends on the choice of the oriented paths $\left\{ \gamma_{0i} \right\}$; this dependence is removed after passing to the orbit space $\left[ \rho \right]\in C^{1} \left( \Gamma;G \right)/ C^{0}\left(\Gamma;G\right)$. In fact, let $\left\{ \tilde{\gamma}_{0i} \right\}$ be an arbitrary choice of $\left| V \right|$ oriented paths connecting $v_0$ to each vertex of $\Gamma$ satisfying $\tilde{\gamma}_{00}=\left\{ v_0 \right\}$, then
\begin{equation*}
    \tilde{\rho}_{ij}=\chi \left( \tilde{\gamma}_{0i}\circ \left( i,j \right)\circ \tilde{\gamma}_{0j}^{-1} \right) =\chi \left( \tilde{\gamma}_{0i}\circ\gamma_{0i}^{-1} \right)\chi\left(\gamma_{0i}\circ \left( i,j \right)\circ \gamma_{0j}^{-1}\right)\chi \left( \gamma_{0j}\circ \tilde{\gamma}_{0j}^{-1}\right)=\chi \left( \gamma_{0i}\circ\tilde{\gamma}_{0i}^{-1} \right)\rho_{ij}\,\chi \left( \gamma_{0j}\circ \tilde{\gamma}_{0j}^{-1} \right),
\end{equation*}
i.e., as elements in $C^{1} \left( \Gamma;G \right)$, $\tilde{\rho}$ differs from $\rho$ by an action of the vertex potential $f\in C^{0}\left(\Gamma;G\right)$ defined as
\begin{equation*}
  f_i:=\chi\left( \gamma_{0i}\circ\tilde{\gamma}_{0i}^{-1} \right),\quad\forall v_i\in V.
\end{equation*}
Therefore, \eqref{eq:inverse_hol_homomorphism} uniquely specifies an element $\left[ \rho \right]$ in $C^{1} \left( \Gamma;G \right)/ C^{0}\left(\Gamma;G\right)$ for any $\chi\in\Hom \left( \pi_1 \left( \Gamma \right), G \right)$.

It remains to show that $\Hol \left(\left[ \rho \right]\right)$ differs from $\chi$ by an inner automorphism of $G$. To see this, let $\omega$ be an arbitrary oriented loop on $\Gamma$ based at $v_0$ consisting of consecutive edges $\left( v_0,v_{i_1} \right), \left( v_{i_1}, v_{i_2} \right), \cdots, \left( v_{i_{N}}, v_0 \right)$, where $N$ is some nonnegative integer. Using $\gamma_{00}=\left\{ 0 \right\}$ and $\gamma_{0i}^{-1}\circ \gamma_{0i}= \left\{ v_0 \right\}$ for any $v_i\in V$, we have
\begin{equation*}
%\label{eq:image_coincides_with_homomorphism}
  \begin{aligned}
    &\hol_{\rho}\left( \omega \right) = \left(\rho_{0,i_N}\rho_{i_N,i_{N-1}}\cdots\rho_{i_2,i_1}\rho_{i_1,0}\right)^{-1}=\rho_{0,i_1}\rho_{i_1,i_2}\cdots\rho_{i_{N-1},i_N}\rho_{i_N,0} \\
    & = \chi \left( \gamma_{00}\circ \left( v_0,v_{i_1} \right)\circ \gamma_{0,i_1}^{-1} \right)\chi \left( \gamma_{0,i_1}\circ \left( v_{i_1},v_{i_2} \right)\circ \gamma_{0,i_2}^{-1} \right)\cdots \chi \left( \gamma_{0,i_{N-1}}\circ \left( v_{i_{N-1}},v_{i_N} \right)\circ \gamma_{0,i_N}^{-1} \right)\chi \left( \gamma_{0,i_N}\circ \left( v_{i_N},v_{0} \right)\circ \gamma_{00}^{-1} \right) \\
    & = \chi \left( \left\{ v_0 \right\} \right)\chi \left( \left( v_0,v_{i_1} \right)\circ\left( v_{i_1}, v_{i_2} \right)\circ \cdots\circ\left( v_{i_{N}}, v_0 \right) \right)\chi \left( \left\{ v_0 \right\} \right)^{-1}=\chi \left( \omega \right).
  \end{aligned}
\end{equation*}
This calculation is clearly independent of the choices of oriented paths $\left\{ \tilde{\gamma}_{0i} \right\}$. Thus $\Hol$ maps $\left[ \rho \right]$ exactly to $\chi$, an element of $\Hom \left( \pi_1 \left( \Gamma,v_0 \right),G \right)$; the independence of $\Hol \left( \left[ \rho \right] \right)$ as an element of $\Hom \left( \pi_1 \left( \Gamma \right), G \right)/G$ with respect to the choice of the base point $v_0$ follows from an essentially identical argument as given in the proof of Lemma~\ref{lem:lem-holonomy-iso}. The last statement follows from Lemma~\ref{lem:equiv-class-characterization}.
\end{proof}

Theorem~\ref{lem:holonomy_homomorphism_bijectivity} is closely related to classification theorems of flat connections and principal bundles with disconnected structure group (see, e.g. \cite[\S13.6]{Taubes2011DG}) and \cite[\S13.9]{Steenrod1951}). The synchronizability of an edge potential $\rho$ on connected graph $\Gamma$, which is equivalent to the triviality of $\mathscr{B}_{\rho}$ (c.f. Proposition~\ref{prop:synchronizability_flat_bundle_one_skelenton}), can now be interpreted as the corresponding conjugacy class of $\Hom \left( \pi_1 \left( \Gamma \right),G \right)$. In fact, the conjugacy class corresponding to trivial bundles $\mathscr{B}_{\rho}$ is also trivial and reflects the triviality of the holonomy of $\Gamma$. The proof of Corollary~\ref{cor:cor-trivial-hol} further develops this observation.

\begin{corollary}
\label{cor:cor-trivial-hol}
  For a connected graph $\Gamma$ and topological group $G$, an edge potential $\rho\in C^{1} \left( \Gamma;G \right)$ is synchronizable if and only if $\Hol_{\rho} \left( \Gamma \right)$ is trivial.
\end{corollary}
\begin{proof}
  Note that $\rho\in C^{1} \left( \Gamma;G \right)$ is synchronizable (see \eqref{eq:vertex-pot-satisfies-edge-pot}) if and only if there exists $f\in C^0 \left( \Gamma;G \right)$ such that
  \begin{equation*}
    f_i^{-1}\rho_{ij}f_j=e\in G\quad\forall \left( i,j \right)\in E,
  \end{equation*}
where $e$ is the identify element of the structure group $G$. This is equivalent to saying that
  \begin{equation}
  \label{eq:syncable-equivalence}
    \left[ \rho \right]=\left[ \mathfrak{e} \right]\in C^{1} \left( \Gamma;G \right)/ C^{0}\left(\Gamma;G\right), \quad\textrm{where $\mathfrak{e}\in C^{1} \left( \Gamma;G \right)$ is defined as $\mathfrak{e}_{ij}=e\in G$ for all $\left( i,j \right)\in E$},
  \end{equation}
which by Theorem~\ref{lem:holonomy_homomorphism_bijectivity} implies
\begin{equation*}
  \Hol \left( \left[ \rho \right] \right)=\Hol \left( \left[ \mathfrak{e} \right] \right)=\mathrm{Id}_e\in \Hom \left( \pi_1 \left( \Gamma \right), G \right)/G,
\end{equation*}
where $\mathrm{Id}_e:\pi_1 \left( \Gamma \right)\rightarrow G$ is the constant map sending all oriented loops in $\Gamma$ to the identity element $e\in G$. The conclusion follows immediately by noting that
\begin{equation*}
  \begin{aligned}
    \Hol \left( \left[ \rho \right] \right)=\mathrm{Id}_e\quad&\Leftrightarrow\quad\textrm{for any $v\in V$,}\,\,\hol_{\rho} \left( \omega \right)=e\textrm{ for all oriented loops $\omega$ based at $v$}\\
    &\Leftrightarrow\quad \Hol_{\rho} \left( v \right)=\left\{ e \right\}\textrm{ for all vertices $v\in V$}\\
    &\Leftrightarrow\quad \Hol_{\rho} \left( \Gamma \right) \textrm{ is trivial}
  \end{aligned}
\end{equation*}
where for the last equivalence we invoked Lemma~\ref{lem:triviality_global_holonomy}.
\end{proof}

Corollary~\ref{cor:cor-trivial-hol} on its own can be derived from an elementary argument. In fact, without descending $\hol_{\rho}$ from $\Omega_v$ to $\pi_1 \left( \Gamma,v \right)$, we can still define $\Hol_{\rho} \left( v \right)$ as the image $\hol_{\rho} \left( \Omega_v \right)$, though $\hol_{\rho}$ is not injective as a group homomorphism from $\Omega_v$ to $G$. The triviality of $\Hol_{\rho}\left( v \right)$ still implies the existence of a vertex potential $f\in C^{0}\left(\Gamma;G\right)$ compatible with $\rho\in C^{1} \left( \Gamma;G \right)$ (simply by setting $f_i=e$ on an arbitrarily chosen $v_i\in V$ and progressively propagating values of $f$ to neighboring vertices), and \emph{vice versa}. The exposition in this section, centered around Theorem~\ref{lem:holonomy_homomorphism_bijectivity}, extends beyond this elementary argument and strives to unveil a complete geometric picture underlying the ``correspondence between trivialities'' discussed in Corollary~\ref{cor:cor-trivial-hol}. In future work we intend to pursue novel synchronization algorithms based on metric and symplectic structures on the moduli space of flat bundles (see, e.g. \cite{AtiyahBott1983,Weinstein1995,Hitchin1990}).

\subsection{A Twisted Hodge Theory for Synchronization Problems}
\label{sec:geom-synchr-probl}

In this section we relate synchronization to the first cohomology of a de Rham cochain complex on $\Gamma$ with coefficients twisted by a representation space $F$ of the structure group $G$. This can be interpreted as an instance of the standard de Rham cohomology of flat bundles (see e.g. \cite{Xia2012,GarciaRayan2015,Simpson1994,Simpson1995}). The fibre bundles considered in this section are vector bundles (with fibre type $F$) associated with the principal bundle studied in Section~\ref{sec:synchr-probl}. When the vector space $F$ is equipped with a metric, the vector bundle inherits a compatible metric, with which a twisted Hodge Laplacian can be constructed; special cases of this twisted Laplacian in the lowest degree include the connection Laplacian \cite{SingerWu2012VDM,BSS2013}. In this setting, synchronizability is realized as a condition on the dimension of the null space of the lowest degree twisted Hodge Laplacian, this is reminiscent of the classical Riemann-Hilbert correspondence between flat connections and locally constant sheaves. The spectral information of the twisted Hodge Laplacian serves as a quantification of 
the the level of obstruction to synchronizability.

\subsubsection{Flat Associated Bundles and Twisted Zero-Forms}
\label{sec:flat-assoc-bundles}

Let $\mathscr{B}_{\rho}$ be the synchronization principal bundle on $\Gamma$ associated with $\rho\in C^{1} \left( \Gamma;G \right)$, as in Proposition~\ref{prop:synchronizability_flat_bundle_one_skelenton}, and $F$ be a topological space on which $G$ acts on the left as a topological transformation group. Denote the action of $G$ on $F$ as $\tau:G\rightarrow \mathrm{Aut}\left( F \right)$. Consider the right action of $G$ on $\mathscr{B}_{\rho}\times F$ as
\begin{equation*}
  \left( p,v \right)\longmapsto \left( pg, \tau\left(g^{-1}\right)v \right).
\end{equation*}
The orbit space of this action, conventionally denoted as $\mathscr{B}_{\rho}\times_GF$ or $\mathscr{B}_{\rho}\left[ F \right]$, is referred to as the \emph{$F$-bundle associated with principal bundle $\mathscr{B}_{\rho}$}, or \emph{associated $F$-bundle} for short. We will denote the bundle projection as $\pi:\mathscr{B}_{\rho}\left[ F \right]\rightarrow \Gamma$, and denote $\mathscr{B}_{\rho}\left[ F \right]_x:=\pi^{-1}\left( x \right)$ for the \emph{fibre over $x\in \Gamma$}. Strictly speaking, the graph $\Gamma$ should be distinguished from its underlying topological space, but we use the same notation $\Gamma$ for both as long as the meaning is clear from the context.

The same open cover $\mathfrak{U}$ of $\Gamma$ that trivializes $\mathscr{B}_{\rho}$ also trivializes $\mathscr{B}_{\rho} \left[ F \right]$. In fact, the bundle transition function of $\mathscr{B}_{\rho} \left[ F \right]$ on any nonempty $U_i\cap U_j$ is the constant map $U_i\cap U_j\rightarrow \tau \left( \rho_{ij} \right)\in\mathrm{Aut}\left( F \right)$, where $U_i\cap U_j \rightarrow \rho_{ij}$ is the constant bundle transition function from $U_i$ to $U_j$ for $\mathscr{B}_{\rho}$. Consequently, the associated bundle $\mathscr{B}_{\rho} \left[ F \right]$ is also flat. Unless confusions arise, we shall refer to $\mathscr{B}_{\rho} \left[ F \right]$ as the \emph{flat associated $F$-bundle of $\mathscr{B}_{\rho}$}, and denote the local trivialization of the associated bundle over $U_i\in\mathfrak{U}$ using the same notation $\phi_i:U_i\times F\rightarrow \mathscr{B}_{\rho}\left[ F \right]$ as for the principal bundle $\mathscr{B}_{\rho}$.

In the context of synchronization problems, the most relevant associated bundles are those with fibre $F$ being a vector space and structure group $G$ being the general linear group $\GL \left( F \right)$. These types of fibre bundles are commonly referred to as \emph{vector bundles}. We will focus on flat associated vector bundles for the rest of the section, though the definition of \emph{fibre projections} and \emph{sections} extend literally to general fibre bundles. For simplicity of presentation, we will omit the notation $\tau$ and write the bundle transition functions again as $\rho_{ij}$ (instead of $\tau \left( \rho_{ij} \right)$), since its action on a vector space $F$ is simply matrix-vector multiplication.

We now focus on \emph{sections}, the analog of ``functions'' on smooth manifolds but with values in fibre bundles. For a general fibre bundle $\mathfrak{E}\rightarrow \mathfrak{B}$, a \emph{local section} $s:U\rightarrow \mathfrak{E}|_U$ of $\mathfrak{E}$ on an open set $U$ of the base space $\mathfrak{B}$ is a continuous map from $U$ to $\mathfrak{E}|_U$ such that $\pi\circ s$ is identified on $U$. A \emph{global section} of $\mathfrak{E}$ is a local section defined on the entire base space $\mathfrak{B}$. We shall encode the data of synchronization problems into the language of sections of flat associated bundles. The discrete nature of the problem naturally motivates us to consider special classes of local and global sections that are ``constant'' within each open set in $\mathfrak{U}$, in a sense to be made clear soon in local coordinates. The following notion of \emph{fibre projection} is introduced to simplify notations involving local coordinates.
\begin{definition}
For any $i\in V$, define the \emph{fibre projection} over $U_i\in\mathfrak{U}$, denoted as 
\begin{equation}
  \label{eq:trivialization_vertical_projection}
  p_i:\mathscr{B}_{\rho}\left[ F \right]\big|_{U_i}=\pi^{-1} \left( U_i \right)\longrightarrow F,
\end{equation}
as the composition of $\phi_i^{-1}:\mathscr{B}_{\rho}\left[ F \right]\big|_{U_i}\rightarrow U_i\times F$ with the canonical projection $U_i\times F\rightarrow F$. For any $x\in U_i$, the restriction of $p_i$ to the fibre $\mathscr{B}_{\rho}\left[ F \right]_x$, denoted as $p_{i,x}:\mathscr{B}_{\rho}\left[ F \right]_x\rightarrow F$, is (by definition) simultaneously a homeomorphism between topological spaces and an isomorphism between vector spaces.
\end{definition}

\begin{definition}[Constant Local Sections]
  A \emph{constant local section} $s:U_i\rightarrow \mathscr{B}_{\rho}\left[ F \right]\big|_{U_i}$ of the bundle $\mathscr{B}_{\rho}\left[ F \right]$ on open set $U_i\in\mathfrak{U}$ is a local section of $\mathscr{B}_{\rho}\left[ F \right]$ such that $p_{i,x} \left( s \left( x \right) \right)$ is a constant element of $F$ for all $x\in U_i$. We refer to the linear space of all constant local sections on $U_i\in\mathfrak{U}$ as \emph{constant twisted local $0$-forms on $U_i$}, denoted as $\Omega_i^0 \left( \Gamma; \mathscr{B}_{\rho}\left[ F \right] \right)$.
\end{definition}
Clearly, a constant local section $s\in\Omega_i^0 \left( \Gamma; \mathscr{B}_{\rho}\left[ F \right] \right)$ is unambiguously determined by evaluating $s$ at vertex $i$, or equivalently by reading off the fibre projection image $s_i:=p_i \left( s \left( i \right) \right)$. We denote this characterization of $s$ as
% =======
% as the composition of $\phi_i^{-1}:\mathscr{B}_{\rho}\left[ F \right]\big|_{U_i}\rightarrow U_i\times F$ with the canonical projection $U_i\times F\rightarrow F$. For any $x\in U_i$, the restriction of $p_i$ to the fibre $\mathscr{B}_{\rho}\left[ F \right]_x$ is a homeomorphism, denoted as $p_{i,x}:\mathscr{B}_{\rho}\left[ F \right]_x\rightarrow F$.
% \end{definition}

% Given an arbitrary vector $v\in F$, the fibre projection $p_i$ identifies $v$ with a point $p_{i,x}^{-1}\left( v \right)$ on the fibre $\mathscr{B}_{\rho}\left[ F \right]_x$. Such an identification is local in nature: for $x\in U_i\cap U_j$, the two fibre projections $p_i$, $p_j$ generally assign two different points on the fibre $\mathscr{B}_{\rho}\left[ F \right]_x$ to any $v\in F$. Therefore, fibre projections enable us to view linear spaces of $F$-valued maps on $V$ and skew-symmetric $F$-valued maps on $E$, denoted respectively as

% %\TG{begin to doubt whether the following are the right definitions --- in order to be compatible with the definition of twisted differential/codifferentials in section 2.2.2, and inner products in 2.2.3, we have to begin with definitions of Omega-0 and Omega-1 as (intrinsic) geometric quantities on the bundle, instead of a pack of local sections.....}

% >>>>>>> 8b762fa018fd7575da28d0be06fa99a69bc4fe58
\begin{equation}
\label{eq:characterize-constant-local-section}
  p_i \left( s \left( x \right) \right)\equiv s_i,\quad\forall x\in U_i,\,\,s\in\Omega_i^0 \left( \Gamma; \mathscr{B}_{\rho}\left[ F \right] \right).
\end{equation}
When we consider $x\in U_i\cap U_j$ where $U_i\cap U_j\neq \emptyset$ (i.e. when $\left( i,j \right)\in E$), it will be convenient to note that the fibre projection $p_j$ evaluates $s \left( x \right)$ to $p_j \left( s \left( x \right) \right)=p_j\circ p_i^{-1} \left( p_i \left( s \left( x \right) \right) \right)\equiv \rho_{ji}s_i$. This can be understood as a ``change-of-coordinates'' formula for constant local sections.

Let $C^0 \left( \Gamma; F \right):=\left\{ f:V\rightarrow F \right\}$ denote the linear space of \emph{$F$-valued $0$-cochains} on graph $\Gamma$. Every element $f$ of $C^0 \left( \Gamma; F \right)$ defines a collection of constant local sections $\left\{ f^{\left( i \right)}:U_i\rightarrow \pi^{-1}\left( U_i \right) \,\mid\,U_i\in\mathfrak{U} \right\}$, one for each $U_i\in \mathfrak{U}$ with
\begin{equation}
\label{eq:local_sections_assoc_omega_zero}
  f^{\left( i \right)} \left( x \right):=p_{i,x}^{-1}\left( f_i \right),\quad\forall x\in U_i.
\end{equation}
We thus have the canonical identification
\begin{equation}
\label{eq:zero-cochain-identification}
  C^0 \left( \Gamma; F \right) = \prod_{i\in V}\Omega^0_i \left( \Gamma; \mathscr{B}_{\rho}\left[ F \right] \right).
\end{equation}
Of course, the constant local sections $\left\{f^{\left( i \right)}\right\}$ specified by $f\in C^0 \left( \Gamma; F \right)$ generally do not give rise to a global section of the bundle $\mathscr{B}_{\rho}\left[ F \right]$, unless they ``patch together'' seamlessly on every nonempty intersection $U_i\cap U_j$, satisfying the condition 
\begin{equation}
\label{eq:constant-local-sections-patch-together}
  p_{i,x}^{-1} \left( f_i \right)=f^{\left( i \right)}\left( x \right)=f^{\left( j \right)}\left( x \right)=p_{j,x}^{-1}\left( f_j \right),\,\,\forall x\in U_i\cap U_j \quad \Leftrightarrow\quad f_i = p_{i,x}\circ p_{j,x}^{-1}\left( f_j \right)=\rho_{ij}f_j,\,\,\forall x\in U_i\cap U_j.
\end{equation}
The right hand side of \eqref{eq:constant-local-sections-patch-together} is recognized as a solution to the synchronization problem with prescribed edge potential $\rho$. We have thus proved the following Lemma.
\begin{lemma}
\label{lem:local-compatibility-form-global}
  The constant local sections specified by $f\in C^0 \left( \Gamma; F \right)$ define a global section on $\mathscr{B}_{\rho} \left[ F \right]$ if and only if
\begin{equation}
\label{eq:global_section}
  f_i=\rho_{ij}f_j,\quad\forall \left( i,j \right)\in E,
\end{equation}
i.e., if and only if the vertex potential $f:V\rightarrow F$ is a solution to the $F$-synchronization problem over $\Gamma$ with respect to the edge potential $\rho\in C^1 \left(\Gamma; G\right)$.
\end{lemma}

When condition \eqref{eq:global_section} is satisfied, the resulting global section constructed from constant local sections is special among all global sections of $\mathscr{B}_{\rho}\left[ F \right]$ in that its restriction to each $U_i$ has constant image under fibre projection over $U_i$. This type of global section will be of major interest in the remainder of this section.

\begin{definition}[Locally Constant Global Section]
\label{defn:locally_constant_global_section}
  A global section $s:\Gamma\rightarrow\mathscr{B}_{\rho}\left[ F \right]$ is said to be \emph{locally constant} if
  \begin{equation}
    \label{eq:defn-locally-constant-global-section}
    p_{i} \left( s \left( x \right) \right) \equiv \mathrm{const.}\quad\forall x\in U_i.
  \end{equation}
The linear space of all locally constant global sections on $\mathscr{B}_{\rho}\left[ F \right]$ will be called \emph{locally constant twisted global $0$-forms} on $\Gamma$, denoted as $\Omega^0 \left( \Gamma; \mathscr{B}_{\rho}\left[ F \right] \right)$.
\end{definition}

Naturally, $\Omega^0 \left( \Gamma; \mathscr{B}_{\rho}\left[ F \right] \right)$ embeds into $C^0 \left( \Gamma; F \right)$ by
\begin{equation}
  \label{eq:lcwg-zero-forms-embed-vert-potentials}
  \begin{aligned}
    \Omega^0 \left( \Gamma; \mathscr{B}_{\rho}\left[ F \right] \right) &\hookrightarrow \prod_{i\in V}\Omega^0_i \left( \Gamma; \mathscr{B}_{\rho}\left[ F \right] \right) = C^0 \left( \Gamma; F \right)\\
     s &\longmapsto \left( s|_{U_1}, \cdots, s|_{U_n} \right)
  \end{aligned}
\end{equation}
where $n= \left| V \right|$ stands for the total number of vertices in $\Gamma$. The objective of a $F$-synchronization problem over $\Gamma$ with respect to $\rho\in C^1 \left( \Gamma; G \right)$ can be interpreted in this geometric framework as searching for an element of $\Omega^0 \left( \Gamma; \mathscr{B}_{\rho}\left[ F \right] \right)$ in the feasible domain $C^0 \left( \Gamma; F \right)$.

The existence of global sections is crucial information for the structure of a fibre bundle. For principal bundles $\mathscr{B}_{\rho}$ considered in Section~\ref{sec:synchr-probl}, a single global section dictates the triviality of the bundle. Though the triviality of a principal bundle is equivalent to its associated vector bundle (see Proposition~\ref{prop:triviality_of_principal_and_associated_bundles}), $\mathscr{B}_{\rho}\left[ F \right]$ is trivial if and only if it admits $d=\dim F$ global sections $s^{1},\cdots,s^{d}$ that are \emph{linearly independent} in the sense that $s_x^{1},\cdots,s_x^{d}$ on each fibre $F_{x}$ are linearly independent as vectors in $F$ (c.f. \cite[Theorem 2.2]{MilnorStasheff1974}). A collection of linearly independent global sections are said to form a \emph{global frame} (see e.g. \cite[Chapter 5]{Lee2003GTM218}) for the vector bundle, since they define a basis (frame) for each fibre. As will be established in Proposition~\ref{prop:triviality_of_principal_and_associated_bundles}, the fact that the bundle $\mathscr{B}_{\rho}\left[ F \right]$ is flat further reduces its triviality to finding $d$ linearly independent locally constant global sections, for which linear independence only needs to be checked at the vertices of $\Gamma$. More precisely, adopting notation $s_i:=p_i\left(s \left( i \right)\right),\forall i\in V$ for an arbitrary section $s$ of $\mathscr{B}_{\rho}\left[ F \right]$, we define the linear independence of locally constant global sections of $\mathscr{B}_{\rho}\left[ F \right]$ as follows.

\begin{definition}
\label{defn:linear_indpendence_as_global_sections}
  A collection of $k$ ($1\leq k\leq d=\dim F$) locally constant global sections $s^{1},\cdots,s^{k}\in\Omega^0 \left( \Gamma; \mathscr{B}_{\rho}\left[ F \right] \right)$ are said to be \emph{linearly independent} if $s^{1}_i,\cdots,s^{k}_i$ are linearly independent as vectors in $F$ at every vertex $i\in V$.
\end{definition}

By the embedding \eqref{eq:lcwg-zero-forms-embed-vert-potentials}, any $s\in\Omega^0 \left( \Gamma; \mathscr{B}_{\rho}\left[ F \right] \right)$ can be equivalently encoded into a vector of dimension $nd$, where $d=\dim F$ and $n=\left| V \right|$ stands for the number of vertices of $\Gamma$. In fact, just as for any vertex potential in $C^0 \left( \Gamma; F \right)$, one simply needs to vertically stack the column vectors $\left\{ s_i=p_i s \left( i \right) \mid i\in V \right\}$, the fibre projection images of $s$ at each vertex. We shall write $\left[ s \right]$ for such a vector of length $nd$ that encodes $s\in\Omega^0 \left( \Gamma; \mathscr{B}_{\rho}\left[ F \right] \right)$, and refer to the vector as the \emph{representative vector} of the section. The linear independence of locally constant global sections is easily seen to be equivalent to the linear independence of the representative vectors of length $nd$, as the following Lemma clarifies.

\begin{lemma}
\label{lem:equivalence-linear-independence}
  A collection of $k$ ($1\leq k\leq d=\dim F$) locally constant global sections $s^{1},\cdots,s^{k}$ are linearly independent if and only if $\left[ s^1 \right], \cdots, \left[ s^k \right]$ are linear independent as vectors of length $nd$.
\end{lemma}
\begin{proof}
Since $\mathscr{B}_{\rho}\left[ F \right]$ is a flat bundle and the graph $\Gamma$ is assumed connected, the linear independence of locally constant global sections $s^{1},\cdots,s^{k}$ is equivalent to the linear independence of vectors $s^1_i,\cdots,s^k_i$ at any vertex $i$: since $\rho_{ji}\in G$ are all invertible, vectors $s^1_i,\cdots,s^k_i$ are linearly independent if and only if $s^1_{j}=\rho_{ji}s^1_i,\cdots,s^k_j=\rho_{ji}s^k_i$ are linearly independent. For definiteness, let us fix $i=1$. Write $S$ for the $nd$-by-$k$ matrix with $\left[ s^j \right]$ as its $j$th column, $S_1$ for the $d\times k$ matrix with $s^j_1$ as its $j$th column, $e=\left[ 1,\cdots,1 \right]^{\top}$ for the column vector of length $d$ with all entries equal to one, and $P$ for the $nd$-by-$nd$ block diagonal matrix with $\rho_{j1}$ at its $j$th diagonal block (adopting the convention $\rho_{11}=I_{n\times n}$). The conclusion follows from the matrix identity
\begin{equation*}
  S = P \left( e\otimes S_1 \right)
\end{equation*}
and
\begin{equation*}
  \rank \left( S \right)=\rank \left( \left[ 1,\cdots,1 \right]^{\top} \right)\cdot\rank \left( S_1 \right)=\rank \left( S_1 \right).
\end{equation*}
\end{proof}

\begin{remark}
\label{rem:linear-indpendence-subtlety}
Note that the equivalence of two notions of linear independence only holds if we already know that $s^1,\cdots,s^k$ are global sections. For general $f^1,\cdots,f^k\in C^0 \left( \Gamma; F \right)$ that are linearly independent as $nd$-vectors, their corresponding representatives in $\prod_{i\in V}\Omega^0_i \left( \Gamma; \mathscr{B}_{\rho}\left[ F \right] \right)$ do not necessarily define global sections, nor are they in general linearly independent as constant local sections on each $U_i$. A simple example is to consider a graph $\Gamma$ consisting of two vertices $V = \left\{ v_1, v_2 \right\}$ and only one edge connecting them, $F=\mathbb{R}^2$, and $G$ is the trivial group consisting of only the $2\times 2$ identity matrix: vectors $\left[ f^1 \right]=\left( 1,0,1,0 \right)^{\top}$ and $\left[ f^2 \right]=\left( 1,0,0,1 \right)^{\top}$ are linearly independent as vectors in $\mathbb{R}^4$ but do not define linearly independent constant local sections on $U_1$.
\end{remark}

With all essential concepts presented, we are ready to establish our main observation in this subsection.
\begin{proposition}
\label{prop:triviality_of_principal_and_associated_bundles}
Let $G$ be a topological transformation group acting on a (real or complex) $d$-dimensional vector space $F$ on the left, $\Gamma= \left( V,E \right)$ be a connected undirected graph, and $\rho\in C^{1} \left( \Gamma;G \right)$ a $G$-valued edge potential. The following statements are equivalent:
\begin{enumerate}[(i)]
\item\label{item:I} $\mathscr{B}_{\rho}$ is trivial;
\item\label{item:II} $\mathscr{B}_{\rho}$ admits a global section;
\item\label{item:III} $\mathscr{B}_{\rho}\left[ F \right]$ is trivial;
\item\label{item:IV} $\mathscr{B}_{\rho}\left[ F \right]$ admits $d=\dim F$ linearly independent locally constant global sections.
\end{enumerate}
\end{proposition}
\begin{proof}
  The equivalence of \eqref{item:I} and \eqref{item:III} follows from \cite[Theorem 4.3]{Steenrod1951}. The equivalence \eqref{item:I}$\Leftrightarrow$\eqref{item:II} follows from standard differential geometry, see e.g. \cite{Steenrod1951,MilnorStasheff1974}; similarly standard is the equivalence between \eqref{item:III} and the existence of $n$ linearly independent global sections on $\mathscr{B}_{\rho}\left[ F \right]$. To show the equivalence \eqref{item:IV}$\Leftrightarrow$\eqref{item:III}, it suffices to prove that a trivial flat vector bundle $\mathscr{B}_{\rho}\left[ F \right]$ admits $d$ linearly independent global sections that are also locally constant. To see this, recall from Proposition~\ref{prop:synchronizability_flat_bundle_one_skelenton} and Corollary~\ref{cor:cor-trivial-hol} that
  \begin{equation*}
    \begin{aligned}
      \mathscr{B}_{\rho}\left[ F \right]\textrm{ is trivial}\,\,&\Leftrightarrow\,\,\mathscr{B}_{\rho}\textrm{ is trivial}\,\,\Leftrightarrow\,\,\Hol \left( \left[ \rho \right] \right)\textrm{ is trivial}\\
      &\Leftrightarrow\exists \,g\in\Omega_{\rho}^0\left( \Gamma; G \right)\,\textrm{ s.t. } \,g_i^{-1}\rho_{ij}g_j=e\,\,\forall \left( i,j \right)\in E\\
      &\Leftrightarrow\exists \,g\in\Omega_{\rho}^0\left( \Gamma; G \right)\,\textrm{ s.t. } \,\rho_{ij}=g_ig_{j}^{-1}\,\,\forall \left( i,j \right)\in E.
    \end{aligned}
  \end{equation*}
Let $\left\{ e_1,\cdots,e_d \right\}$ be a basis for $F$, and for each $k=1,\cdots,d$ define $s^k:\Gamma\rightarrow \mathscr{B}_{\rho}\left[ F \right]$ as
\begin{equation*}
  s^k \left( x \right)=p_{i,x}^{-1} \left( g_i e_k \right)\quad\forall x\in U_i.
\end{equation*}
It is straightforward to verify by definition that $s^k$ is a well-defined global section and is locally constant. That $s^1,\cdots,s^d$ are linearly independent as global sections follows from the fact that $p_{i,x}: F\rightarrow \mathscr{B}_{\rho}\left[ F \right]_x$ are isomorphisms between vector spaces.
\end{proof}

\begin{remark}
  The global section in \eqref{item:II} is also ``locally constant'' in a sense analogous to Definition~\ref{defn:locally_constant_global_section} but for principal bundles; we do not introduce this definition here since global sections on principal bundles will not be pursued directly in this work. 
\end{remark}

Proposition~\ref{prop:triviality_of_principal_and_associated_bundles} points out an alternative approach for determining the synchronizability of an edge potential $\rho\in C^{1} \left( \Gamma;G \right)$, at least when $G$ is a matrix group $\GL \left( F \right)$: it suffices to check the existence of $d=\dim F$ linearly independent locally constant global sections on the flat associated vector bundle $\mathscr{B}_{\rho} \left[ F \right]$. Such existence can be stated as a cohomological obstruction. We will pursue such a formulation in the next section. In Section~\ref{sec:twisted-laplacians} we will utilize the inner product structure on $F$ to reduce the structure group of a $\GL \left( F \right)$-bundle to $O \left( d \right)$ or $U \left( d \right)$, as commonly seen in synchronization problems \cite{BKS2016,BSS2013,CKS2015}. If the underlying fibre bundle is orientable, the same procedures further reduce the structure group to $\SO \left( d \right)$ or $\SU \left( d \right)$, corresponding to synchronization problems considered in \cite{SZSH2011,WangSinger2013}.

\subsubsection{Twisted One-Forms and De Rham Cohomology}
\label{sec:twisted-de-rham-coho}

In a smooth category, sections on a fibre bundle can be differentiated by a \emph{covariant derivative}. The resulting object is a skew-symmetric ``direction-dependent'' section on the same bundle, or equivalently a section of a new fibre bundle which is the tensor product of the original fibre bundle with the bundle of $1$-forms on the base manifold. We shall generalized this picture to the discrete/combinatorial setting for flat associated bundles $\mathscr{B}_{\rho}\left[ F \right]$ that naturally arise in synchronization problems.

Recall from discrete Hodge theory \cite{Crane:2013:DGP,Lim2015,DHLMJ2005,JLYY2011,JohnsonGoldring2013} that discrete $0$-forms and $1$-forms on a graph $\Gamma$ are defined as
\begin{equation*}
  \Omega^0 \left( \Gamma \right) := \left\{ f: V\rightarrow \mathbb{K} \right\},\quad \Omega^1 \left( \Gamma \right) := \left\{ \omega:E\rightarrow \mathbb{K}\mid\omega_{ij}=-\omega_{ji}\,\,\forall \left( i,j \right)\in E \right\},
\end{equation*}
where $\mathbb{K}=\mathbb{C}$ or $\mathbb{R}$, which we will assume to be the number field for the vector space $F$. Let us define a local version of $\Omega^1 \left( \Gamma \right)$ by
\begin{equation}
  \label{eq:local-one-forms-graph}
  \Omega_i^1 \left( \Gamma \right):=\left\{ \omega:N_i\rightarrow \mathbb{K}\mid\omega_{jk}=-\omega_{kj}\,\,\forall \left( j,k \right)\in N_i \right\},\quad\textrm{where }N_i:= \left\{ \left( j,k \right)\in E\mid j=i\textrm{ or }k=i \right\}.
\end{equation}
In other words, elements of $\Omega^1_i \left( \Gamma \right)$ are restrictions of elements of $\Omega^1 \left( \Gamma \right)$ to $U_i$. By a partition of unity argument, it is straightforward to identify $\Omega^1 \left( \Gamma \right)$ with
\begin{equation}
  \label{eq:graph-one-form-identification}
  \left\{ \left( \omega^{\left( 1 \right)},\cdots,\omega^{\left( n \right)} \right)\in \prod_{i\in V}\Omega_i^1 \left( \Gamma \right)\,\Big|\, \omega^{\left( i \right)}_{ij}=\omega^{\left( j \right)}_{ij}(=-\omega_{ji}^{\left( j \right)}=-\omega_{ji}^{\left( i \right)}) \right\}.
  %\left\{ \left( \omega^{\left( 1 \right)},\cdots,\omega^{\left( n \right)} \right)\in \bigoplus_{i\in V}\Omega_i^1 \left( \Gamma \right)\,\Big|\, \omega^{\left( i \right)}_{ij}=\omega^{\left( j \right)}_{ij}(=-\omega_{ji}^{\left( j \right)}=-\omega_{ji}^{\left( i \right)}) \right\}.
\end{equation}

\begin{definition}[Constant Local $1$-forms]
  A \emph{constant twisted local $1$-form} on open set $U_i\in\mathfrak{U}$ is a local section of  $\Omega^0_i \left( \Gamma; \mathscr{B}_{\rho}\left[ F \right] \right)\otimes \Omega^1_i \left( \Gamma \right)$. Equivalently, a constant twisted local $1$-form on $U_i$ is a map $\omega:U_i\times N_i\rightarrow \mathscr{B}_{\rho}\left[ F \right]$ such that:
\begin{enumerate}[(i)]
\item For any $\left( j,k \right)\in N_i$, $\omega_{jk}:U_i\rightarrow\mathscr{B}_{\rho}\left[ F \right]$ is a constant local section on $U_i$, i.e.
  \begin{equation*}
    p_{i,x}\left(\omega_{ij} \left( x \right)\right)\equiv\textrm{const.};
  \end{equation*}
\item For any $x\in U_i$ and any $U_j\in\mathfrak{U}$ such that $U_i\cap U_j\neq \emptyset$, $\omega_{ij}\left( x \right)=-\omega_{ji}\left( x \right)$.
\end{enumerate}
We denote the linear space of all constant twisted local $1$-forms on $U_i$ as $\Omega_i^1 \left( \Gamma; \mathscr{B}_{\rho}\left[ F \right] \right)$.
\end{definition}

A similar notion of globally defined twisted $1$-forms will also be of interest. In the discrete setting of synchronization problems, it suffices to consider twisted global $1$-forms that are locally constant under fibre projections.

\begin{definition}[Locally Constant Global $1$-forms]
  A \emph{locally constant twisted global $1$-form} is a section of the tensor product bundle $\Omega^0 \left( \Gamma; \mathscr{B}_{\rho}\left[ F \right] \right)\otimes\Omega^1 \left( \Gamma \right)$. In other words, a locally constant twisted global $1$-form is a map
  \begin{equation*}
    \omega: \left\{ \left( i, \left( j,k \right) \right)\mid i\in V, \left( j,k \right)\in N_i \right\}\rightarrow\mathscr{B}_{\rho}\left[ F \right]
  \end{equation*}
such that:
\begin{enumerate}[(i)]
\item For any $U_i\in \mathfrak{U}$, $\omega|_{U_i}$ is a constant twisted local $1$-form on $U_i$;
\item For any $x\in U_i\cap U_j\neq \emptyset$, $p_{i,x}\left(\omega|_{U_i} \left( x \right)\right)=\rho_{ij}p_{j,x}\left(\omega|_{U_j} \left( x \right)\right)$.
\end{enumerate}
We denote the linear space of all locally constant twisted global $1$-forms on $\Gamma$ as $\Omega^1 \left( \Gamma; \mathscr{B}_{\rho}\left[ F \right] \right)$.
\end{definition}

The definition of $\Omega^1 \left( \Gamma; \mathscr{B}_{\rho}\left[ F \right] \right)$ already characterized the condition under which a given collection of constant twisted local $1$-forms $\left\{\omega^{\left( i \right)}\in\Omega_i^1 \left( \Gamma; \mathscr{B}_{\rho}\left[ F \right] \right)\right\}$, one for each $U_i\in\mathfrak{U}$, can be patched to form a locally constant twisted global $1$-form. Similar to \eqref{eq:constant-local-sections-patch-together}, it suffices to check the compatibility under ``change of coordinates''.

\begin{lemma}
  A collection of constant twisted local $1$-forms $\left\{\omega^{\left( i \right)}\in\Omega_i^1 \left( \Gamma; \mathscr{B}_{\rho}\left[ F \right] \right)\right\}$ defines a locally constant twisted global $1$-form if and only if
\begin{equation}
\label{eq:global-one-form}
  p_{i,x}\left(\omega_{ij}^{\left( i \right)} \left( x \right)\right)=\rho_{ij}p_{j,x}\left(\omega_{ij}^{\left( j \right)}\left( x \right)\right),\quad\forall \left( i,j \right)\in E.
\end{equation}
Since both sides of equality \eqref{eq:global-one-form} are constants, we shall simplify \eqref{eq:global-one-form} as
\begin{equation}
\label{eq:global-one-form-simplified}
  p_i\left(\omega_{ij}^{\left( i \right)} \right)=\rho_{ij}p_j\left(\omega_{ij}^{\left( j \right)}\right),\quad\forall \left( i,j \right)\in E.
\end{equation}
\end{lemma}

A significant difference between $\Omega^1 \left( \Gamma; \mathscr{B}_{\rho}\left[ F \right] \right)$ and $\Omega^0 \left( \Gamma; \mathscr{B}_{\rho}\left[ F \right] \right)$ is that a locally constant twisted global $1$-form does not naturally arise from an $F$-valued $1$-cochain in $C^1 \left( \Gamma; F \right):=\left\{ \omega:E\rightarrow F\mid \omega_{ij}=-\omega_{ji} \right\}$,
and these cochains play an essential role in the discrete Hodge theory. For instance, if we set $\omega_{ij}^{\left( i \right)} \left( x \right):=p_{i,x}^{-1}\left( \omega_{ij} \right)=:-\omega^{\left( i \right)}_{ji} \left( x \right)$ for all $x\in U_i$, then $\left\{ \omega^{\left( i \right)}\mid i\in V \right\}$ gives rise to a twisted global $1$-form if and only if
\begin{equation*}
  \omega_{ij} = p_{i,x}\left(\omega_{ij}^{\left( i \right)} \left( x \right)\right)=\rho_{ij}p_{j,x}\left(\omega_{ij}^{\left( j \right)}\left( x \right)\right)=-\rho_{ij}p_{j,x}\left( \omega_{ji}^{\left( j \right)}\left( x \right) \right)=-\rho_{ij}\omega_{ji},\quad\forall \left( i,j \right)\in E,
\end{equation*}
a condition that is generally not satisfied unless $\rho_{ij}\equiv e\in G$ for all $\left( i,j \right)\in E$. This observation indicates that $C^1 \left( \Gamma; F \right)$ is not a geometric object naturally associated with the structure of the vector bundle $\mathscr{B}_{\rho}\left[ F \right]$, but rather a special case of $\Omega^1 \left( \Gamma; \mathscr{B}_{\rho}\left[ F \right] \right)$ when the vector bundle $\mathscr{B}_{\rho}\left[ F \right]$ is trivial. In this case $\rho_{ij}=g_{i}g_j^{-1},\,\forall \left( i,j \right)\in E$ for a $G$-valued vertex potential $g:V\rightarrow G$ and the ``gauge-transformed'' constant twisted local $1$-forms $\left\{ p_{i,x}^{-1} \left( g_i\omega_{ij} \right) \mid i\in V, \left( i,j \right)\in E \right\}$ satisfy the compatibility condition \eqref{eq:global-one-form}. Exploring the action of the gauge group on twisted forms will be considered in the future.

With an appropriate notion of twisted $1$-forms, we are ready to define the twisted differential operator on twisted $0$-forms. This operation is a discrete analog of the covariant derivatives for smooth fibre bundles, and in the meanwhile, a fibre bundle analog of the discrete exterior derivative in discrete Hodge theory.

\begin{definition}[Twisted Differential on Twisted $0$-Cochains]
For $\rho\in C^1 \left( \Gamma; G \right)$ and $U_i\in \mathfrak{U}$, the \emph{$\rho$-twisted differential} is a linear operator taking any $f\in C^0 \left( \Gamma; F \right)=\prod_{i\in V}\Omega_i^0 \left( \Gamma; \mathscr{B}_{\rho}\left[ F \right] \right)$ to a collection of $n$ constant twisted local $1$-forms, one for each $U_i\in\mathfrak{U}$ %\JB{Note that $d_\rho$ is defined on the direct product, whereas before we talked about direct sums. It does not matter for a finite graph, but it would be good to be consistent.}
\begin{equation*}
  \begin{aligned}
    d_{\rho}:\prod_{i\in V}\Omega_i^0 \left( \Gamma; \mathscr{B}_{\rho}\left[ F \right] \right) &\longrightarrow \prod_{i\in V}\Omega_i^1 \left( \Gamma; \mathscr{B}_{\rho}\left[ F \right] \right)\\
    f &\longmapsto \left(\left( d_{\rho}f \right)^{\left( 1 \right)},\cdots, \left( d_{\rho}f \right)^{n}\right)
  \end{aligned}
\end{equation*}
where each $\left( d_{\rho}f \right)^{\left( i \right)}\in\Omega^1_i \left( \Gamma; \mathscr{B}_{\rho}\left[ F \right] \right)$ is defined as
\begin{equation}
  \label{eq:twisted_differential_vertical_projection_defn}
  \left(d_{\rho} f\right)^{\left( i \right)}_{ij}\left( x \right) := p_{i,x}^{-1}\left(f_i-\rho_{ij}f_j\right) =: -\left(d_{\rho} f\right)^{\left( i \right)}_{ji}\left( x \right),\quad\forall U_i\in\mathfrak{U},\,\,\forall x\in U_i \cap U_j \neq \emptyset,\,\,f\in\Omega_{\rho}^0 \left( \Gamma; F \right).
\end{equation}
\end{definition}

Though $d_{\rho}$ is defined as a linear operator mapping into a collection of constant twisted local $1$-forms, a somewhat surprising fact is that these constant twisted local $1$-forms do patch together to form an element of $\Omega^1 \left( \Gamma; \mathscr{B}_{\rho}\left[ F \right] \right)$.
\begin{proposition}
  The twisted differential $d_{\rho}$ maps $C^0 \left( \Gamma; F \right)$ into $\Omega^1 \left( \Gamma; \mathscr{B}_{\rho}\left[ F \right] \right)$.
\end{proposition}
\begin{proof}
  It suffices to check \eqref{eq:global-one-form-simplified} for the collection of constant twisted local $1$-forms $\left\{ \left( d_{\rho}f \right)^{\left( i \right)}\,\big|\,i=1,\cdots,n\right\}$. In fact,
  \begin{equation*}
    p_i\left( \left(d_{\rho} f\right)^{\left( i \right)}_{ij} \right)=f_i-\rho_{ij}f_j=-\rho_{ij} \left( f_j-\rho_{ji}f_i \right)=-\rho_{ij}p_j\left( \left(d_{\rho} f\right)^{\left( j \right)}_{ji} \right)=\rho_{ij}p_j\left( \left(d_{\rho} f\right)^{\left( j \right)}_{ij}\right)\quad\forall \left( i,j \right)\in E.
  \end{equation*}
\end{proof}

Since the graph $\Gamma$ (viewed as a simplicial complex) does not contain any $2$-simplices, $d_{\rho}$ is the only differential needed for specifying the \emph{$\rho$-twisted chain complex}
\begin{equation}
\label{eq:twisted_chain_complex}
  0\longrightarrow C^0 \left( \Gamma;F \right) \stackrel{\textrm{\footnotesize $d_{\rho}$}}{\longrightarrow} \Omega^1 \left( \Gamma; \mathscr{B}_{\rho}\left[ F \right] \right) \longrightarrow 0.
\end{equation}
The only non-trivial cohomology group in this de Rham-type chain complex is at the $0$-th order
\begin{equation*}
  H_{\rho}^0 \left( \Gamma; \mathscr{B}_{\rho}\left[ F \right] \right) := \ker d_{\rho}.
\end{equation*}
\begin{proposition}
\label{prop:twisted-diff-ker-global-sections}
  $\ker d_{\rho}=\Omega^0 \left( \Gamma; \mathscr{B}_{\rho}\left[ F \right] \right)$.
\end{proposition}
\begin{proof}
  Note in the definition \eqref{eq:twisted_differential_vertical_projection_defn} that
  \begin{equation*}
    f\in\ker d_{\rho}\quad\Leftrightarrow\quad f_i=\rho_{ij}f_j,\,\,\forall \left( i,j \right)\in E.
  \end{equation*}
The conclusion then follows from Lemma~\ref{lem:local-compatibility-form-global}.
\end{proof}

By Proposition~\ref{prop:triviality_of_principal_and_associated_bundles}, detecting the synchronizability of $\rho\in C^{1} \left( \Gamma;G \right)$ now reduces to checking if $\dim\ker d_{\rho}=\dim F$ holds. Furthermore, in scenarios where this dimension equality does not hold, $\dim\ker d_{\rho}$ still provides a quantitative measure for the extent to which synchronizability fails. In this sense, the cohomology group $H_{\rho}^0 \left( \Gamma;F \right)$ serves as the opposite of a ``topological obstruction'' to the synchronizability of $\rho\in C^1 \left( \Gamma; G \right)$.

\subsubsection{Twisted Hodge Theory and Synchronizability}
\label{sec:twisted-laplacians}

In the remainder of this section, we will focus on flat associated bundles $\mathscr{B}_{\rho}\left[ F \right]$ with the vector space $F$ equipped with an inner product $\left\langle \cdot,\cdot \right\rangle_F:F\times F\rightarrow \mathbb{K}$, where $\mathbb{K}=\mathbb{C}$ or $\mathbb{R}$ depending on the vector space $F$. This inner product on $F$ will be further assumed with \emph{$G$-invariance}, in the sense that
\begin{equation*}
  \langle gx, gy \rangle_F = \langle x, y \rangle_F \quad\forall x,y\in F, g\in G.
\end{equation*}
In the terminology of representation theory, we assume that the representation of $G$ on $F$ is \emph{unitary} (c.f. \cite[\S II.1]{Broecker2003GTM98}). This inner product introduces other related concepts into the geometric framework:

\begin{itemize}
%\begin{enumerate}[(1)]
\item $F$ is equipped with a $G$-invariant norm $\left\|x\right\|_F = \langle x, x \rangle_F$ for all $x\in F$, which further induces an operator norm on $G$ via duality
\begin{equation*}
  \left\| g \right\| := \sup_{\substack{m\in F\\ \left\| m \right\|\neq 0}} \frac{\left\| gm \right\|}{\left\| m \right\|},\quad\forall g\in G.
\end{equation*}
For simplicity, we will use the same notation for the norm on $F$ and the dual norm on $G$.

\item\label{item:formal_adjoint} For any $g\in G$, its \emph{formal adjoint} with respect to the inner product $\left\langle \cdot,\cdot \right\rangle_F$, denoted as $g^{*}$, is defined as
\begin{equation*}
  \langle gx, y \rangle = \langle x, g^{*}y \rangle\quad \forall x,y\in F.
\end{equation*}
Note that $\left\| g^{*} \right\|=\left\| g \right\|$ for any $g\in G$.

\item The twisted $0$-cochains $C^0 \left( \Gamma; F \right)$ and locally constant twisted global $1$-forms $\Omega^1 \left( \Gamma; \mathscr{B}_{\rho}\left[ F \right] \right)$ are equipped with inner products and norms induced from the $G$-invariant inner product $\left\langle \cdot,\cdot \right\rangle_F$, as follows:
%$\Omega_{\rho}^0 \left( \Gamma;F \right)$ and $\Omega_{\rho}^1 \left( \Gamma;F \right)$
  % \begin{equation}
  %   \label{eq:inner_product_zero_forms}
  %   \left\langle f, g \right\rangle := \sum_{i\in V} \left\langle p_if \left( v_i \right), p_ig \left( v_i \right) \right\rangle_F,\quad\forall f, g\in \Omega_F^0 \left( \Gamma;F \right)
  % \end{equation}
% \begin{align}
%      \left\langle f, g \right\rangle := \sum_{i\in V} d_i\left\langle f_i, g_i \right\rangle_F,\quad \left\| f \right\| &:= \left(\sum_{i\in V} d_i\left\| f_i \right\|_F^2\right)^{\frac{1}{2}}, \quad\forall f, g\in C^0 \left( \Gamma;F \right)\label{eq:vertex-potential-norm}\\
%      \left\langle \omega, \eta \right\rangle := \sum_{\left( i,j \right)\in E}w_{ij} \left\langle p_i\left(\omega^{\left( i \right)}_{ij}\right), p_i\left(\eta^{\left( i \right)}_{ij}\right) \right\rangle_F, \quad \left\| \omega \right\| &:= \left(\sum_{\left( i,j \right)\in E}w_{ij} \left\|p_i\left(\omega^{\left( i \right)}_{ij}\right) \right\|_F^2\right)^{\frac{1}{2}},\quad\forall \omega,\eta\in \Omega^1 \left( \Gamma; \mathscr{B}_{\rho}\left[ F \right] \right)\label{eq:edge-potential-norm}
%  \end{align}
\begin{align}
   \left\langle f, g \right\rangle := \sum_{i\in V} d_i\left\langle f_i, g_i \right\rangle_F,&\quad\forall f, g\in C^0 \left( \Gamma;F \right) \label{eq:vertex-potential-inner-product}\\
   \left\langle \omega, \eta \right\rangle := \frac{1}{2}\sum_{i\in V}\sum_{j:\left( i,j \right)\in E}w_{ij} \left\langle p_i\left(\omega^{\left( i \right)}_{ij}\right), p_i\left(\eta^{\left( i \right)}_{ij}\right) \right\rangle_F,&\quad\forall \omega,\eta\in\Omega^1 \left( \Gamma; \mathscr{B}_{\rho}\left[ F \right] \right)\label{eq:edge-potential-inner-product}\\
   \left\| f \right\| := \left\langle f,f \right\rangle^{\frac{1}{2}}, \quad \left\| \omega \right\| := \left\langle \omega,\omega \right\rangle^{\frac{1}{2}}, &\quad\forall f\in C^{0}\left( \Gamma; F \right),\omega\in \Omega^1 \left( \Gamma; \mathscr{B}_{\rho}\left[ F \right] \right)\label{eq:vert-egde-potentials-norm}
 \end{align}
where $w_{ij}$ is the weight on edge $\left( i,j \right)\in E$ and $d_i=\sum_{j:\left( i,j \right)\in E}w_{ij}$ is the weighted degree of vertex $v_i$. Note that by the $G$-invariance of $\left\langle \cdot,\cdot \right\rangle_F$ the sum in \eqref{eq:edge-potential-inner-product} can be equivalently written as (see \ref{sec:appendix-proof-fibre-bundle-sync} for a quick calculation)
\begin{equation}
  \label{eq:edge-potential-inner-product-equivalent}
  \left\langle \omega, \eta \right\rangle = \sum_{\left( i,j \right)\in E}w_{ij} \left\langle p_i\left(\omega^{\left( i \right)}_{ij}\right), p_i\left(\eta^{\left( i \right)}_{ij}\right) \right\rangle_F.
\end{equation}
\end{itemize}

Through local trivializations, an inner product on the vector space $F$ also induces a \emph{bundle metric} on $\mathscr{B}_{\rho}\left[ F \right]$, i.e., a section of the second symmetric power of the dual bundle of $\mathscr{B}_{\rho}\left[ F \right]$ which restricts to each fibre as a symmetric positive definite quadratic form. As is well known (see e.g. \cite[Chapter 7]{Taubes2011DG}), a bundle metric can be used to reduce the structure group of a vector bundle from $\GL \left( F \right)$ to $O \left( d \right)$ or $U \left( d \right)$, where $d=\dim \left( F \right)$. It suffices to consider global sections of $\mathscr{B}_{\rho}\left[ F \right]$ for $\rho \in C^1 \left( \Gamma;O\left(d \right)\right)$ or $\rho \in C^1 \left( \Gamma; U\left(d \right)\right)$ for many synchronization problems of practical interest \cite{BSS2013,BKS2016,CKS2015}, instead of requiring $\rho \in C^1 \left( \Gamma;\GL \left( d;\mathbb{R}\right) \right)$ or $\rho \in C^1 \left( \Gamma;\GL \left( d;\mathbb{C} \right) \right)$. Other important types of synchronization problem involving $\SO \left( d \right)$ and $\SU \left( d \right)$ can be treated in this geometric framework as determining global sections of \emph{orientable} vector bundles (see e.g. \cite[Chapter 7]{Taubes2011DG} or \cite[Proposition 6.4]{BottTu1982}). Also note that $\rho_{ij}^{-1}=\rho_{ij}^{*}$ for edge potentials in all these special matrix groups. Since the bundle reduction allows us to focus only on synchronization problems with orthogonal or unitary matrices, without loss of generality, we will always assume the edge potentials satisfy $\rho_{ji}=\rho_{ij}^{-1}=\rho_{ij}^{*}$ for all $\left( i,j \right)\in E$. The same is assumed in \cite{BSS2013,BKS2016}.

The inner product structures on $C^0 \left( \Gamma; F \right)$ and $\Omega^1 \left( \Gamma; \mathscr{B}_{\rho}\left[ F \right] \right)$ enable us to define the \emph{$\rho$-twisted codifferential} $\delta_{\rho}:\Omega^1 \left( \Gamma; \mathscr{B}_{\rho}\left[ F \right] \right)\rightarrow C^0 \left( \Gamma; F \right)$, the formal adjoint operator of the twisted differential $d_{\rho}: C^0 \left( \Gamma; F \right)\rightarrow \Omega^1 \left( \Gamma; \mathscr{B}_{\rho}\left[ F \right] \right)$ in the chain complex \eqref{eq:twisted_chain_complex}, eventually leading to a twisted Hodge theory for synchronization problems. The definition of $\delta_{\rho}$ is consistent with the discrete divergence operator in discrete Hodge theory \cite{JLYY2011,JohnsonGoldring2013}:
\begin{equation}
\begin{aligned}
  \delta_{\rho}:\Omega^1 \left( \Gamma; \mathscr{B}_{\rho}\left[ F \right] \right)&\longrightarrow C^0 \left( \Gamma; F \right)\\
  \theta &\longmapsto \left(\left( \delta_{\rho}\theta \right)\big|_{U_1},\cdots, \left( \delta_{\rho}\theta \right)\big|_{U_{n}}\right)
\end{aligned}
%\left(\delta_{\rho}\theta\right)_i \left( x \right) = \phi_i \left( x, \frac{1}{d_i}\sum_{j: \left( i,j \right)\in E}w_{ij}p_i\left(\theta^{\left( i \right)}_{ij}\right) \right)\quad \forall x\in U_i,\,\,\theta\in\Omega^1 \left( \Gamma; \mathscr{B}_{\rho}\left[ F \right] \right).
\end{equation}
where each $\left( \delta_{\rho}\theta \right)\big|_{U_i}\in $ is defined by
\begin{equation}
  \label{eq:twisted_codifferential_defn}
  \left(\delta_{\rho}\theta\right)\big|_{U_i} \left( x \right) = p_{i,x}^{-1}\left( \frac{1}{d_i}\sum_{j: \left( i,j \right)\in E}w_{ij}p_i\left(\theta^{\left( i \right)}_{ij}\right)\right)\quad \forall x\in U_i,\,\,\theta\in\Omega^1 \left( \Gamma; \mathscr{B}_{\rho}\left[ F \right] \right)
\end{equation}
or equivalently
\begin{equation}
  \label{eq:twisted_codifferential_defn_equivalent}
  \left( \delta_{\rho}\theta \right)_i = p_i\left(\left(\delta_{\rho}\theta\right)\big|_{U_i} \left( i \right)\right) = \frac{1}{d_i}\sum_{j: \left( i,j \right)\in E}w_{ij}p_i\left(\theta^{\left( i \right)}_{ij}\right)\quad \forall i\in V,\,\,\theta\in\Omega^1 \left( \Gamma; \mathscr{B}_{\rho}\left[ F \right] \right).
\end{equation}

\begin{proposition}
\label{prop:codifferentail_well_defined}
With respect to the inner products \eqref{eq:vertex-potential-inner-product} and \eqref{eq:edge-potential-inner-product}, the twisted codifferential $\delta_{\rho}:\Omega^1 \left( \Gamma; \mathscr{B}_{\rho}\left[ F \right] \right)\rightarrow C^0 \left( \Gamma; F \right)$ defined by \eqref{eq:twisted_codifferential_defn} is the formal adjoint of the twisted differential $d_{\rho}:C^0 \left( \Gamma; F \right)\rightarrow \Omega^1 \left( \Gamma; \mathscr{B}_{\rho}\left[ F \right] \right)$ defined by \eqref{eq:twisted_differential_vertical_projection_defn}.
\end{proposition}
\begin{proof}
  Note that for any $f\in C^0 \left( \Gamma; F \right)$, $\theta\in\Omega^1 \left( \Gamma; \mathscr{B}_{\rho}\left[ F \right] \right)$,
  \begin{equation*}
    \begin{aligned}
      \left\langle f,\delta_{\rho}\theta \right\rangle &= \sum_{i\in V}d_i\left\langle f_i, \frac{1}{d_i}\sum_{j:\left( i,j \right)\in E}w_{ij}p_{i}\left(\theta^{\left( i \right)}_{ij}\right) \right\rangle_F = \sum_{i\in V}\sum_{j:\left( i,j \right)\in E}\left\langle f_i, w_{ij}p_i \left(\theta^{\left( i \right)}_{ij}\right) \right\rangle_F\\
      &=\sum_{\left( i,j \right)\in E} \left[ \left\langle f_i, w_{ij} p_i\left( \theta^{\left( i \right)}_{ij}\right) \right\rangle_F+\left\langle f_j,w_{ji} p_j\left(\theta^{\left( j \right)}_{ji}\right) \right\rangle_F \right]\\
      &=\sum_{\left( i,j \right)\in E}\left\langle f_i, w_{ij}p_{i}\left(\theta^{\left( i \right)}_{ij}\right) \right\rangle_F+\sum_{\left( i,j \right)\in E}\left\langle f_j,w_{ji}p_{j}\left(\theta^{\left( j \right)}_{ji}\right) \right\rangle_F =: \left( I \right) + \left( II \right).
    \end{aligned}
  \end{equation*}
% The term $\left( I \right)$ can be re-assembled into
% \begin{equation*}
%   \left( I \right) = \sum_{i\in V} \left\langle f_i, \sum_{j:\left( i,j \right)\in E}w_{ij}\theta_{ij}^{\left( i \right)} \right\rangle_F,
% \end{equation*}
We keep the term $\left( I \right)$ intact and manipulate term $\left( II \right)$ using $w_{ij}=w_{ji}$ and the $G$-invariance of $\left\langle \cdot,\cdot \right\rangle_F$:
\begin{equation*}
  \left( II \right) = \sum_{\left( i,j \right)\in E}\left\langle \rho_{ij}f_j,w_{ji}\rho_{ij}p_{j}\left(\theta^{\left( j \right)}_{ji}\right) \right\rangle_F\stackrel{(*)}{=\!=}\sum_{\left( i,j \right)\in E}\left\langle \rho_{ij}f_j,w_{ij}p_{i}\left(\theta^{\left( i \right)}_{ji}\right) \right\rangle_F\stackrel{(**)}{=\!=}-\sum_{\left( i,j \right)\in E}\left\langle \rho_{ij}f_j,w_{ij}p_{i}\left(\theta^{\left( i \right)}_{ij}\right) \right\rangle_F,
\end{equation*}
where we used $\rho_{ij}p_j\left(\theta_{ji}^{\left( j \right)}\right)=p_{i}\left(\theta_{ji}^{\left( i \right)}\right)$ (the compatibility condition \eqref{eq:global-one-form-simplified}) at $(*)$, and the skew-symmetry $\theta_{ji}^{\left( i \right)}=-\theta_{ij}^{\left( i \right)}$ at $(**)$. Re-combining $\left( I \right)$ and $\left( II \right)$, we conclude that
\begin{equation*}
  \begin{aligned}
    \left\langle f,\delta_{\rho}\theta \right\rangle &= \left( I \right) + \left( II \right)=\sum_{\left( i,j \right)\in E}\left\langle f_i, w_{ij}p_i\left(\theta^{\left( i \right)}_{ij}\right) \right\rangle_F-\sum_{\left( i,j \right)\in E}\left\langle \rho_{ij}f_j,w_{ij} p_i\left( \theta^{\left( i \right)}_{ij}\right) \right\rangle_F\\
    &=\sum_{\left( i,j \right)\in E}\left\langle f_i-\rho_{ij}f_j, w_{ij} p_i\left(\theta_{ij}^{\left( i \right)}\right) \right\rangle_F = \sum_{\left( i,j \right)\in E}w_{ij}\left\langle p_i\left(\left(d_{\rho}f\right)^{\left( i \right)}_{ij}\right), p_i\left(\theta_{ij}^{\left( i \right)}\right) \right\rangle_F=\left\langle d_{\rho}f, \theta \right\rangle.
  \end{aligned}
\end{equation*}
\end{proof}

The chain complex \eqref{eq:twisted_chain_complex} is now also equipped with formal adjoints:
\begin{equation}
\label{eq:twisted_chain_complex_with_adjoint}
  0\mathrel{\substack{\xrightarrow[\hspace{0.15in}]{} \\ \vspace{-0.15in}  \\ \xleftarrow[]{\hspace{0.15in}}}}C^0 \left( \Gamma;F \right) \mathrel{\substack{\xrightarrow[\hspace{0.15in}]{\textrm{\footnotesize $d_{\rho}$}} \\ \vspace{-0.15in}  \\ \xleftarrow[\textrm{\footnotesize $\delta_{\rho}$}]{\hspace{0.15in}}}}\Omega^1 \left( \Gamma;\mathscr{B}_{\rho}\left[ F \right] \right) \mathrel{\substack{\xrightarrow[\hspace{0.15in}]{} \\ \vspace{-0.15in}  \\ \xleftarrow[]{\hspace{0.15in}}}} 0.
%  0\mathrel{\substack{\xrightarrow[\hspace{0.15in}]{} \\ \vspace{-0.15in}  \\ \xleftarrow[]{\hspace{0.15in}}}}\Omega_{\rho}^0 \left( \Gamma;F \right) \mathrel{\substack{\xrightarrow[\hspace{0.15in}]{\textrm{\footnotesize $d_{\rho}$}} \\ \vspace{-0.15in}  \\ \xleftarrow[\textrm{\footnotesize $\delta_{\rho}$}]{\hspace{0.15in}}}}\Omega_{\rho}^1 \left( \Gamma;F \right) \mathrel{\substack{\xrightarrow[\hspace{0.15in}]{} \\ \vspace{-0.15in}  \\ \xleftarrow[]{\hspace{0.15in}}}} 0.
\end{equation}
Two \emph{twisted Hodge Laplacians} can be constructed from this chain complex:
\begin{align}
  \Delta_{\rho}^{\left( 0 \right)}:=\,\,\delta_{\rho}d_{\rho}:\, &C^0 \left( \Gamma; F \right)\longrightarrow C^0\left( \Gamma; F \right),\label{eq:twisted_hodge_laplacians_zero}\\
  \Delta_{\rho}^{\left( 1 \right)}:=\,\,d_{\rho}\delta_{\rho}:\, &\Omega^1 \left( \Gamma; \mathscr{B}_{\rho}\left[ F \right] \right)\longrightarrow \Omega^1\left( \Gamma; \mathscr{B}_{\rho}\left[ F \right] \right).\label{eq:twisted_hodge_laplacians_one}
\end{align}
It is straightforward to see from these definitions that both twisted Laplacians are positive definite. In view of Hodge theory, it would be of interest to investigate the \emph{harmonic forms} in the complex \eqref{eq:twisted_chain_complex_with_adjoint}, the kernels of $\Delta_{\rho}^{\left( 0 \right)}$ and $\Delta_{\rho}^{\left( 1 \right)}$.

\begin{lemma}
\label{lem:equal_kernels}
$\ker d_{\rho}=\ker\Delta_{\rho}^{\left( 0 \right)}$ and $\ker \delta_{\rho}=\ker\Delta_{\rho}^{\left( 1 \right)}$.
\end{lemma}
\begin{proof}
  Clearly $\ker d_{\rho}\subset \ker\Delta_{\rho}^{\left( 0 \right)}$. For the reverse inclusion, note that by adjointness
  \begin{equation*}
    0=\left\langle f, \Delta_{\rho}^{\left( 0 \right)}f \right\rangle=\left\| d_{\rho}f \right\|^2\quad \forall f\in\ker \Delta_{\rho}^{\left( 0 \right)},
  \end{equation*}
which implies $d_{\rho}f=0$. The equality involving $\ker\Delta_{\rho}^{\left( 1 \right)}$ follows from a similar argument.
\end{proof}

The following decomposition results follow from standard Hodge-theoretic arguments.
\begin{theorem}
\label{thm:hodge_decomp}
$C^0\left( \Gamma; F \right)=\ker\Delta_{\rho}^{\left( 0 \right)}\oplus\im \delta_{\rho}=\ker d_{\rho}\oplus\im \delta_{\rho}$,$\quad$ $\Omega^1 \left( \Gamma; \mathscr{B}_{\rho}\left[ F \right] \right)=\im d_{\rho}\oplus \ker\Delta_{\rho}^{\left( 1 \right)}=\im d_{\rho}\oplus\ker \delta_{\rho}$.
\end{theorem}
\begin{proof}
  We only present the proof for the decomposition of $C^0 \left( \Gamma;F \right)$; the decomposition for $\Omega^1 \left( \Gamma; F \right)$ is similar. First note that both $C^0 \left( \Gamma; F \right)$ and $\Omega^1 \left( \Gamma; \mathscr{B}_{\rho}\left[ F \right] \right)$ are finite dimensional. The subspace $\ker d_{\rho}$ and $\im \delta_{\rho}$ are orthogonal with respect to the inner product \eqref{eq:vertex-potential-inner-product}, since if $f\in \ker d_{\rho}$ and $\delta_{\rho}\theta\in\im\delta_{\rho}$,
  \begin{equation*}
    \left\langle f, \delta_{\rho}\theta \right\rangle=\left\langle d_{\rho}f, \theta \right\rangle=0.
  \end{equation*}
It remains to prove that each $f\in C^0 \left( \Gamma; F \right)$ can be decomposed into a linear combination of elements in $\ker\Delta^{\left( 0 \right)}_{\rho}$ and $\im \delta_{\delta}$. If $d_{\rho}f=0$, the decomposition is trivial. Otherwise, consider the following Poisson equation:
\begin{equation}
\label{eq:poission-fredholm-argument}
  \Delta_{\rho}^{\left( 1 \right)}\theta=d_{\rho}f.
\end{equation}
We claim that equation \eqref{eq:poission-fredholm-argument} has a solution $\theta\in \Omega^1 \left( \Gamma;\mathscr{B}_{\rho}\left[ F \right] \right)$ as long as $d_{\rho}f\neq 0$. In fact, by Fredholm alternative (see e.g. an exposition for the finite dimensional case in \cite{Lim2015} which suffices for our purpose), if $d_{\rho}f\notin\im\Delta^{\left( 1 \right)}_{\rho}$, then $d_{\rho}f\in\ker \Delta_{\rho}^{\left( 1 \right)}=\ker \delta_{\rho}$; however, $d_{\rho}f\perp \ker \delta_{\rho}$ since
\begin{equation*}
  \left\langle d_{\rho} f, \omega \right\rangle=\left\langle f, \delta_{\rho}\omega \right\rangle=0\quad\forall \omega\in \ker\delta_{\rho}.
\end{equation*}
This proves that \eqref{eq:poission-fredholm-argument} has a solution $\theta\in \Omega^1 \left( \Gamma;\mathscr{B}_{\rho}\left[ F \right] \right)$ for $d_{\rho}f\neq 0$. We can thus split $f\in C^0 \left( \Gamma; F \right)$ into
\begin{equation*}
  f = \left( f-\delta_{\rho}\theta \right) + \delta_{\rho}\theta,
\end{equation*}
in which $\delta_{\rho}\theta\in\im\delta_{\rho}$, and $f-\delta_{\rho}\theta\in\ker d_{\rho}$ since
\begin{equation*}
  d_{\rho} \left( f-\delta_{\rho}\theta \right) = d_{\rho}f-d_{\rho}\delta_{\rho}\theta=d_{\rho}f-\Delta_{\rho}^{\left( 1 \right)}\theta=0.
\end{equation*}
\end{proof}

\begin{remark}
  Proposition~\ref{prop:twisted-diff-ker-global-sections} and Theorem~\ref{thm:hodge_decomp} completely characterized the embedding \eqref{eq:lcwg-zero-forms-embed-vert-potentials}: the orthogonality complement (with respect to the inner product \eqref{eq:vertex-potential-inner-product}) of the linear space of solutions to the $F$-synchronization problem on $\Gamma$ with respect to $\rho\in\Omega^1 \left( \Gamma; G \right)$ is exactly the image of the twisted codifferential \eqref{eq:twisted_codifferential_defn}. In fact, one recognizes from $C^0\left( \Gamma; F \right)=\ker\Delta_{\rho}^{\left( 0 \right)}\oplus\im \delta_{\rho}$, $\ker\Delta_{\rho}^{\left( 0 \right)}=\ker d_{\rho}$, and $H_{\rho}^0 \left( \Gamma; \mathscr{B}_{\rho}\left[ F \right] \right)=\ker d_{\rho}$ the well-known Hodge theorem $H_{\rho}^0 \left( \Gamma; \mathscr{B}_{\rho}\left[ F \right] \right)=\ker \Delta_{\rho}^{\left( 0 \right)}$. 
\end{remark}

\subsubsection{Graph Connection Laplacian and Cheeger-Type Inequalities for Graph Frustration}
\label{sec:cheeg-ineq-twist}

In this section, we connect our geometric framework to the computational aspects of synchronization algorithms. As pointed out in Section~\ref{sec:twisted-de-rham-coho}, for cases where $G=O\left( d \right)$ or $G=U\left( d \right)$, the synchronizability of an edge potential $\rho\in C^1 \left( \Gamma; G \right)$ is equivalent to whether or not the equality $\dim\ker d_{\rho}= d$ holds. Lemma~\ref{lem:equal_kernels} reduces the synchronizability further to the dimension of $\ker\Delta_{\rho}^{\left( 0 \right)}$. With identification \eqref{eq:zero-cochain-identification}, it can be noticed that $\Delta_{\rho}^{\left( 0 \right)}$ is exactly the \emph{graph connection Laplacian} (GCL) in the literature of synchronization problems, random matrix theory, and manifold learning (see e.g. \cite{SingerWu2012VDM,BSS2013,KLPSS2016,ElKarouiWu2016}). Recall from \cite{BSS2013} that the graph connection Laplacian for graph $\Gamma$ and edge potential $\rho\in C^1\left(\Gamma;G\right)$ is defined as
\begin{equation}
  \label{eq:GCL}
  L_1 = D_1-W_1
\end{equation}
where $W_1\in \mathbb{K}^{n d\times n d}$ is a $n\times n$ block matrix with $w_{ij}\rho_{ij}\in\mathbb{K}^{d \times d}$ at its $\left( i,j \right)$th block, and $D_1\in \mathbb{K}^{nd\times nd}$ is block diagonal with $d_i I_{d\times d}\in\mathbb{K}^{d\times d}$ at its $\left( i,i \right)$th block. To see that $\Delta_{\rho}^{\left( 0 \right)}$ coincides with $L_1$, notice that
\begin{equation*}
  \left\langle f, \Delta_{\rho}^{\left( 0 \right)}f \right\rangle=\left\| d_{\rho}f \right\|^2=\sum_{\left( i,j \right)\in E}w_{ij}\left\| f_i-\rho_{ij}f_j \right\|^2=\frac{1}{2}\sum_{i,j\in V}w_{ij}\left\| f_i-\rho_{ij}f_j \right\|^2 = \frac{1}{2} \left[ f \right]^{\top}L_1 \left[ f \right],\quad\forall f\in C^0 \left( \Gamma;F \right).
\end{equation*}
Theorem~\ref{thm:hodge_decomp} translates into this combinatorial setting as a decomposition result for the matrix $L_1$, as presented below in Proposition~\ref{prop:gcl-decomp}. We denote $n=\left| V \right|$ and $m=\left| E \right|$ for the graph $\Gamma= \left( V,E \right)$.
\begin{proposition}
\label{prop:gcl-decomp}
  The graph connection Laplacian $L_1\in\mathbb{K}^{nd\times nd}$ admits a decomposition
  \begin{equation}
    \label{eq:hodge-decom-matrix-form}
    L_1 = \left[ \delta_{\rho} \right]\left[ d_{\rho} \right],\quad\left[\delta_{\rho}\right]\in\mathbb{K}^{nd\times md}, \left[ d_{\rho} \right]\in\mathbb{K}^{md\times nd},
  \end{equation}
where $\left[ d_{\rho} \right]$ is an $m$-by-$n$ block matrix in which the $\left( i,j \right)$th block is given by
\begin{equation}
  \label{eq:twisted-differential-matrix-form}
  \left[ d_{\rho} \right]_{ij}=
  \begin{cases}
    I_{d\times d} & \textrm{if edge $i$ starts at vertex $j$},\\
    -w_{kj}\rho_{kj} & \textrm{if edge $i$ starts at vertex $k$ and ends at vertex $j$},\\
    0 & \textrm{otherwise},
  \end{cases}
\end{equation}
and $\left[ \delta_{\rho} \right]$ is an $n$-by-$m$ block matrix in which the $\left( i,j \right)$th block is given by
\begin{equation}
  \label{eq:twisted-codifferential-matrix-form}
  \left[\delta_{\rho}\right]_{ij}=
  \begin{cases}
    \displaystyle\frac{w_{ik}}{d_i}I_{d\times d} & \textrm{if edge $j$ starts at vertex $i$ and ends at vertex $k$},\\
    0 & \textrm{otherwise}.
  \end{cases}
\end{equation}
Note that here each edge $\left( i,j \right)$ appears twice in $E$ with opposite orientations.
\end{proposition}

The Hodge decomposition \eqref{eq:hodge-decom-matrix-form} immediately leads to the following observation, which reflects the geometric fact that there do not exist more than $n=\dim F$ linearly independent global sections on the vector bundle $\mathscr{B}_{\rho}\left[ F \right]$.
\begin{proposition}
  The dimension of the null eigenspace of $L_1$ can not exceed $n$, the dimension of both the column space of $\left[d_{\rho}\right]$ and the row space of $\left[ \delta_{\rho} \right]$.
\end{proposition}

By Lemma~\ref{lem:equivalence-linear-independence}, if there are $d$ linearly independent vectors in the kernel space of $L_1$, then they give rise to $d$ locally constant global sections on $\mathscr{B}_{\rho}\left[ F \right]$ that are also linearly independent as global sections, which indicates the triviality of the vector bundle $\mathscr{B}_{\rho}\left[ F \right]$ and the synchronizability of $\rho\in C^1\left(\Gamma;G\right)$. Note that an analogy of this result for graphs with multiple connected components also holds, though we assumed $\Gamma$ is connected throughout this paper: a graph with $k\geq 1$ connected components and a prescribed $O \left( d \right)$-valued edge potential is synchronizable if and only if the dimension of the null eigenspace of $L_1$ is $kd$. This geometric picture is consistent with the main spectral relaxation algorithm \cite[Algorithm 2.5]{BSS2013} when the edge potential is synchronizable. Basically, the spectral relaxation procedure works as follows: first, extract $d$ eigenvectors $x^1,\cdots,x^d$ corresponding to the smallest $d$ eigenvalues of $L_1$; second, form the $n d\times d$ matrix $X=\left[ x^1,\cdots,x^d \right]$ and split it vertically into $n$ blocks $X_1,\cdots,X_{n}$ of equal size $d\times d$; finally, find the closest orthogonal matrix $O_i$ to each $X_i$ by polar decomposition, and construct the desired synchronizing vertex potential $f\in C^0 \left( \Gamma;O \left( d \right) \right)$ by setting $f_i=O_i$. Since $\ker \Delta_{\rho}^{\left( 0 \right)}=d$ for any synchronizable edge potential$\rho$, the $d$ eigenvectors $x^1,\cdots,x^d$ of $L_1$ all lie in the null eigenspace of $L_1$, which provide exactly the $d$ linearly independent global sections needed to trivialize the vector bundle; all that remains for obtaining a desired synchronizing vertex potential is to rescale the columns of each block $X_i\in\mathbb{K}^{d\times d}$ to achieve orthonormality, which is exactly what is done in the polar decomposition step when $\rho$ is synchronizable. The twisted cohomology framework developed in this section suggests the following improvements when applying the spectral relaxation algorithm for determining synchronizability of a given edge potential:

\begin{enumerate}[(1)]
\item Instead of checking $\dim\ker\Delta_{\rho}^{\left( 0 \right)}$, one can simply check $\dim\ker d_{\rho}$ or $\dim\ker \delta_{\rho}$ (which gives $\dim\im\delta_{\rho}$). Both $\left[ d_{\rho} \right]$ and $\left[ \delta_{\rho} \right]$ matrices are much smaller in size compared with $L_1$, and the dimension can be determined by QR decomposition rather than the more expensive eigen-decomposition;
\item Instead of performing polar decomposition for each $d\times d$ block $X_i$, which involves the relatively more expensive SVD, it suffices to invoke a Gram-Schmidt orthonormalization. If synchronizability of $\rho$ is confirmed by the dimension test in a previous step, the Gram-Schmidt procedure can be performed for the entire matrix $X\in\mathbb{K}^{nd\times d}$ in one pass (with a minor modification of keeping the columns to have norm $n$ instead of $1$), as opposed to being carried out for each individual block $X_i$.
\end{enumerate}
\begin{remark}
  The Hodge decomposition \eqref{eq:hodge-decom-matrix-form} also suggests an alternative approach to obtaining $n$ linearly independent locally constant global sections on $\mathscr{B}_{\rho}\left[ F \right]$: instead of directly solving for the null eigenspace of $L_1$, we can look for the orthogonal complement of $\im\left[\delta_{\rho}\right]$. Note, however, that the domain of $\left[\delta_{\rho}\right]$ should not be taken as the entire $\mathbb{K}^{md}$, since $\delta_{\rho}$ is defined on $\Omega^1 \left( \Gamma; \mathscr{B}_{\rho}\left[ F \right] \right)$, in which elements satisfy the compatibility condition \eqref{eq:global-one-form-simplified}. Constructing such a basis matrix $B\in\mathbb{K}^{md\times md}$ and computing the orthogonal complement of the column space of $\left[ \delta_{\rho} \right]B$ turns out not to be much simpler than finding the orthogonal complement of $L_1$ (i.e. finding the null eigenspace of $L_1$ directly).
\end{remark}

In the more general setting where the edge potential $\rho$ is not assumed synchronizable, the geometric picture becomes much more involved. Of central importance to the relaxation algorithms and Cheeger inequalities in \cite{BSS2013} is to minimize the \emph{frustration} of a graph $\Gamma$ with respect to a prescribed group potential:
\begin{equation}
\label{eq:frustration-group-valued-vert-potential-defn}
\begin{aligned}
  \nu \left( \Gamma \right)&=\inf_{g\in C^0 \left( \Gamma; O \left( d \right) \right)}\nu \left( g \right)\\
  &=\inf_{g\in C^0 \left( \Gamma; O \left( d \right) \right)}\frac{1}{2d}\frac{1}{\mathrm{vol}\left( \Gamma \right)}\sum_{i,j\in V}w_{ij}\left\| g_i-\rho_{ij}g_j \right\|_{\mathrm{F}}^2,\quad\textrm{where $\mathrm{vol}\left( \Gamma \right) = \sum_{i\in V}d_i$.}
\end{aligned}
\end{equation}
As shown in (the proof of) Proposition~\ref{prop:triviality_of_principal_and_associated_bundles}, an $O \left( d \right)$-valued edge potential $\xi\in C^1 \left( \Gamma; O \left( d \right) \right)$ is synchronizable if and only if there exists $g\in C^0 \left( \Gamma; O \left( d \right) \right)$ such that $\xi_{ij}=g_ig_j^{-1}$ for all $\left( i,j \right)\in E$. The frustration $\nu \left( \Gamma \right)$ defined in \eqref{eq:frustration-group-valued-vert-potential-defn} can thus be rewritten as
\begin{equation*}
  \begin{aligned}
    \nu \left( \Gamma \right) = \frac{1}{2d}\frac{1}{\mathrm{vol}\left( \Gamma \right)}\inf_{g\in C^0 \left( \Gamma; O \left( d \right) \right)}\sum_{i,j\in V}w_{ij}\left\| g_ig_j^{-1}-\rho_{ij} \right\|_{\mathrm{F}}^2 = \frac{1}{2d}\frac{1}{\mathrm{vol}\left( \Gamma \right)}\inf_{\xi\in C_{\textrm{sync}}^1 \left( \Gamma; O \left( d \right) \right)}\sum_{i,j\in V}w_{ij}\left\| \xi_{ij}-\rho_{ij} \right\|_{\mathrm{F}}^2,
  \end{aligned}
\end{equation*}
where we define
\begin{equation*}
  C_{\textrm{sync}}^1 \left( \Gamma; O \left( d \right) \right) := \left\{ \xi\in C^1 \left( \Gamma; O \left( d \right) \right)\,\big|\,\xi\textrm{ synchronizable} \right\}=\left\{ \xi\in C^1 \left( \Gamma; O \left( d \right) \right)\,\big|\, \Hol_{\xi} \left( \Gamma \right)\textrm{ is trivial}\right\}.
\end{equation*}
Therefore, in the fibre bundle framework, a synchronization problem asks for a synchronizable edge potential that is ``as close as possible'' to a prescribed edge potential, or geometrically speaking, for a trivial flat bundle ``as close as possible'' to a given flat bundle. One approach, from the point of view of Proposition~\ref{prop:triviality_of_principal_and_associated_bundles}, is to find $d$ linearly independent cochains in $C^0 \left( \Gamma; \mathbb{R}^d \right)$ that are ``as close as possible'' to being a global frame of $\mathscr{B}_{\rho}\left[ \mathbb{R}^d \right]$ in the sense of minimizing the \emph{frustration of a $\mathbb{S}^{d-1}$-valued cochain}
\begin{equation*}
  \eta \left( f \right) = \frac{\left\langle f, \Delta_{\rho}^{\left( 0 \right)}f \right\rangle}{\left\| f \right\|^2}  = \frac{\displaystyle \frac{1}{2}\sum_{i,j\in V}w_{ij}\left\| f_i-\rho_{ij}f_j \right\|^2}{\displaystyle\sum_{i\in V} d_i \left\| f_i \right\|^2}=\frac{1}{2\mathrm{vol}\left( \Gamma \right)}\left[ f \right]^{\top}L_1 \left[ f \right],\quad\forall f\in C^0 \left( \Gamma; \mathbb{S}^{d-1} \right),\,\,\left\| f \right\|\neq 0,
\end{equation*}
which equals zero if and only if $f$ defines a global section on $\mathscr{B}_{\rho}\left[ \mathbb{R}^d \right]$ (the constraint $\left\| f \right\|\neq 0$ is also indispensable from a geometric point of view, as any vector bundle trivially admits the constant zero global section). This provides a geometric interpretation of the spectral relaxation algorithm in \cite{BSS2013}. From a perturbation point of view, the magnitudes of the smallest $n$ eigenvalues of $\Delta_{\rho}^{\left( 0 \right)}$ measure the deviation from degeneracy of the $d$-dimensional eigenspace of lowest frequencies, and can thus be interpreted as the extent to which $\mathscr{B}_{\rho}\left[ \mathbb{R}^d \right]$ deviates from admitting $d$ linearly independent global sections and being a trivial bundle. The Cheeger-type inequality established in \cite{BSS2013} quantitatively confirms this geometric intuition relating $\nu \left( \Gamma \right)$ to the magnitude of $d$ smallest eigenvalues of $D_1^{-1}L_1$ (the random walk version of the graph connection Laplacian):
\begin{equation}
  \label{eq:cheeger-like-inequality-frustration}
  \frac{1}{d}\sum_{k=1}^d\lambda_k \left( D_1^{-1}L_1 \right)\leq \nu \left( \Gamma \right) \leq \frac{Cd^3}{\lambda_2 \left( L_0 \right)}\sum_{k=1}^d\lambda_k \left( D_1^{-1}L_1 \right),
\end{equation}
where $C>0$ is a constant, $\lambda_2 \left( L_0 \right)$ is the spectral gap of $\Gamma$ associated with the graph Laplacian $L_0$, and $\lambda_k \left( D_1^{-1}L_1 \right)$ is the $k$th smallest eigenvalue of $D_1^{-1}L_1$. (The actual version stated in \cite{BSS2013} is for the smallest $d$ eigenvalues of the normalized graph connection Laplacian $D_1^{-1/2}L_1 D_1^{-1/2}$, but note that $D_1^{-1/2}L_1D_1^{-1/2}$ has the same eigenvalues as $D_1^{-1}L_1$.)

Classical Cheeger inequalities \cite{Cheeger1970,Alon1995,Chung2007FourProofs} relate isoperimetric constants or cuts on graphs and manifolds to the spectral gap of a graph Laplacian or Laplace-Beltrami operator. There have been Cheeger-type inequalities for simplicial complexes with the objective of understanding high-dimensional generalization of expander graphs \cite{MS2016,PR2016,PRT2015,SKM2014}. These results are all concerned with partitioning graphs, manifolds, or simplicial complexes. The Cheeger-type inequality in equation \eqref{eq:cheeger-like-inequality-frustration} differs from standard Cheeger inequalities in that the cochains are group- or vector-valued. In addition, the frustration $\nu \left( \Gamma \right)$ is not related with any optimalily for graph partitioning --- in the sense of Proposition~\ref{prop:triviality_of_principal_and_associated_bundles}, $\nu \left( \Gamma \right)$ measures the triviality of a fibre bundle as a whole. The algorithm we will propose in Section~\ref{sec:learn-group-acti} is an attempt to address the graph cut problem based on the synchronizability of the partitions resulted.

\section{Learning Group Actions by Synchronization}
\label{sec:learn-group-acti}

In this section we specify an algorithm for learning group actions from observations based on synchronization. We also use simulations to provide some insight towards the performance of the algorithm.

\subsection{Motivation and General Formulation}
\label{sec:motiv-gener-form}

We first state some basic terminology from the general theory of group actions that will be used extensively. If $G$ is a group and $X$ is a set, a \emph{left group action of $G$ on $X$} is a map $\phi:G\times X\rightarrow X:\left( g,x \right)\mapsto \phi \left( g,x \right)$ such that
\begin{equation*}
  \phi \left( e,x \right)=x,\,\forall x\in X\quad\textrm{if $e$ is the identity element of $G$}
\end{equation*}
and
\begin{equation*}
  \phi \left( g,\phi \left( h,x \right) \right)=\phi \left( gh,x \right),\quad\forall x\in X,\,\,\forall g,h\in G.
\end{equation*}
To simplify notation, we will abbreviate $\phi \left( g,x \right)$ as $g.x$. The \emph{orbit} of any element $x\in X$ under the action of $G$ is defined as the set $G.x:=\left\{ g.x\mid g\in G \right\}$. If we introduce an equivalence relation on $X$ by setting
\begin{equation*}
  x\sim y\quad\Leftrightarrow\quad x=g.y\,\,\textrm{for some $g\in G$},
\end{equation*}
then clearly $x\sim y$ if and only if $G.x=G.y$. The set $X$ is naturally partitioned into the disjoint unions of orbits, and each orbit $Y$ is an \emph{invariant subset} of $X$ under the action of $G$ in the sense that $G.Y\subset Y$. If for any pair of distinct elements $x,y$ of $X$ there exists $g\in G$ such that $g.x=y$, we say that the action of $G$ on $X$ is \emph{transitive}. Note that the total space $X$ is an invariant subset in its own right, and the action of $G$ on each orbit is obviously transitive. If the set $X$ is finite and there exists a constant time procedure to verify whether any two elements are equivalent under transformations, the problem of partitioning $X$ into disjoint subsets of orbits can be solved in polynomial time complexity with respect to the size of $X$. 

In practice we are often interested in classification or clustering tasks which can be framed as follows: given a dataset $X=\left\{ x_1,\cdots,x_n \right\}$ of $n$ objects, find a correspondence or transformation between each pair of distinct objects. We will see these pairwise correspondences often play the role of nuisance variables and one needs to ``quotient out'' the influence of these variables in downstream analysis (e.g. for most practical applications of synchronization problems \cite{SZSH2011,TSR2011,BCSZ2014} and alignment problems in statistical shape analysis \cite{Auto3dGM2015}). The intuition as to why some of these pairwise correspondences are nuisance variables one can often with greater fidelity transform one object into another via intermediary transformations to other objects rather than a direct transformation between objects. Sometimes, for instance in the analysis of a collection of shapes in computer graphics \cite{HuangZhangGHBG2012,NBWYG2011,HuangGuibas2013,CGH2014,KKSBL2015,MDKKL2016} and group-wise registration in automated geometric morphometrics \cite{PNAS2011,CP13,LipmanDaubechies2011,LipmanPuenteDaubechies2013,KoehlHass2015}, the pairwise transformations contains crucial information and are important on their own right. A common challenge in both of the above problems is that the fidelity of pairwise comparisons can be extremely variable over the data. We illustrate this challenge using
the example of computing continuous Procrustes distances between disk-type shapes \cite{CP13} in automated geometric morphometrics. The core of the algorithm is an efficient strategy for searching the M\"obius transformation group of the unit disk to obtain a diffeomorphism between the shapes that minimizes an energy functional. It has been observed that for similar shapes
(in the sense of having a small pairwise distance), the resulting diffeomorphism is often of high quality and can reflect the correspondence of biological traits.
If the shape pair is highly dissimilar, the diffeomorphism tends to suffer from various structural errors (see e.g. \cite{GYDMB2018} and \cite[Chapter 5]{Gao2015Thesis}). Similar issues have also been observed in the field of non-rigid shape registration in geometric processing --- successful feature extraction and matching techniques for near-isometric shapes abound \cite{SunOvsGuibas2009HKS,WKS2011,SIHKS2010,VHKS2010,ISCD2012,LaiZhao2014}, whereas registering shape pairs with large deformation is still considered a difficult open problem \cite{BronsteinKimmel2008,KABL2014,APL2015}. Recently, a series of works \cite{NBWYG2011,HuangZhangGHBG2012,HuangGuibas2013,CGH2014,KKSBL2015,MDKKL2016} proposed to jointly compute all pairwise correspondences within a collection subject to ``consistency constraints'' that require the composition of resulted maps along any cycle within the collection be approximately the identity map. The idea in this approach is that pairwise correspondences between dissimilar shapes are implicitly approximated by concatenating many correspondences between similar shapes with the individual correpondences have high fidelity,
thus avoiding directly solving non-convex optimization problems with large numbers of local minimizers. Similar ideas can also be found in recent progress in automated geometric morphometrics where a 
\emph{Minimum Spanning Tree} (MST) 
provides the concatenating of correspondences \cite{Auto3dGM2015,vitek2017semi,GYDMB2018}. It has been observed by morphologists that 
cycle-consistent constraints are more often satisfied for a collection of samples within a species versus samples across a variety of species, suggesting that
inconsistency may be used for species clustering.

Motivated by the above algorithms and approaches, we propose to study the following general problem of \emph{Learning Group Actions} (LGA):
\begin{framed}
  \begin{problem}[Learning Group Actions]
\label{prob:lga-defn}
Given a group $G$ acting on a set $X$, simultaneously learn a new action of $G$ on $X$ and a partition of $X$ into disjoint subsets $X_1,\cdots,X_K$, such that the new action is as close as possible to the given action and cycle-consistent on each $X_i$ ($1\leq i\leq K$).
  \end{problem}
\end{framed}
The LGA problem can also be understood as a variant of the classical clustering problem, in which the coarse-graining is based on the cycle-consistency of group actions rather than pairwise similarity or spatial configuration of elements in the dataset. A solution of the LGA problem provides not only a partition of the input dataset but also cycle-consistent group actions within each cluster. It is useful to notice that all group elements implemented as pairwise actions within the same partition $X_i$ form a subgroup of $G$; the LGA problem can thus also be considered as ``learning'' subgroups of a prescribed ``ambient group'' that optimally fit a given dataset $X$. In other words, by solving an LGA problem we identify the ``correct'' transformation group for a dataset, which in most practical situations are much more tightly adapted to the given data than the potentially massive group of all possible transformations $G$.

\begin{example-numbered}
\label{example:irredu-rep}
  If the set $X$ is a vector space and we seek a direct sum decomposition $X=\bigoplus_{i=1}^KX_i$ instead of a partition $X=\bigcup_{i=1}^K X_i$, the LGA problem reduces to the search for all irreducible $G$-subrepresentations of $X$.
\end{example-numbered}
\begin{example-numbered}
\label{example:spins}
  Consider a point set $X=\left\{ x_1,\cdots,x_n \right\}$ equipped with a labeling map $S:X\rightarrow \left\{ \pm 1 \right\}$ that assigns to each $x_i$ either value $+1$ or $-1$. We say $x_i$ has \emph{positive spin} if $S \left( x_i \right)=1$ and has \emph{negative spin} if $S \left( x_i \right)=-1$. Let $G=\left\{ \pm 1 \right\}$ act on $X$ transitively as $\left( g_{ji},x_i \right)\mapsto x_j,\,g_{ji}=S ( x_j )S ( x_i )$. Suppose the spin of each point in $X$ (i.e. the label map $S$) is unknown, but we have full access to the group actions $\{ g_{ij}\}$, we can reconstruct $S$ --- up to flipping labels $\pm 1$ --- by spectral clustering the dataset $X$, viewed as vertices of a complete graph $\Gamma$ with weight $w_{ij}=g_{ij}$ on the edge connecting $x_i$ and $x_j$. Under circumstances where some group actions $g_{ji}$ are subject to a sign-flip error (noisy measurements), or/and the graph $\Gamma$ is not complete (incomplete measurements), spectral or semi-definite programming relaxation techniques can still be used to recover $S$ up to permuting labels $\pm 1$ (see e.g. \cite{CKS2015}). With $X$ and (potentially noisy and incomplete) $\{g_{ji}\}$ as input, this spectral clustering example can be considered as an instance of LGA: the output consists of a partition of $X$ into positive/negative spin subsets, as well as the trivial subgroup $\left\{ +1 \right\}$ of $G=\left\{ \pm 1 \right\}$ acting in a cycle-consistent manner on both partitions.
\end{example-numbered}

Example~\ref{example:spins} provides further motivation to consider a version of LGA in the context of synchronization problems. We are given  a graph $\Gamma= \left( V,E \right)$ and the data $X$, where the vertex set $V$ is identified with observations in $X$ and the edges in $E$ representing pairwise relations between elements of $X$. It is natural to consider a \emph{partition of the graph} $\Gamma$ in this setup  decomposition of $\Gamma$ into connected subgraphs such that the vertices of the subgraphs form a partition of the set of vertices of $\Gamma$. %Denote the collection of all partitions of $\Gamma$ into $K$ nonempty connected subgroups ($K\leq n$) as $\mathscr{X}_K$.

\begin{framed}
\begin{problem}[Learning Group Actions by Synchronization]
\label{prob:lgas-defn}
  Let $\Gamma = \left( V,E \right)$ be an undirected weighted graph, $G$ a topological group, and $\rho\in C^{1} \left( \Gamma;G \right)$ a given edge potential on $\Gamma$. Furthermore, assume the vertex set $V$ is equipped with a cost function $\mathrm{Cost}_G:G\times G\rightarrow \left[ 0,\infty \right)$. Denote $\mathscr{X}_K$ for all partitions of $\Gamma$ into $K$ nonempty connected subgroups ($K\leq n$) and 
\begin{equation*}
  \nu \left( S_i \right)=\inf_{f\in C^{0}\left(\Gamma;G\right)}\sum_{j,k\in S_i}w_{jk}\mathrm{Cost}_G \left( f_j, \rho_{jk}f_k\right),\,\,\mathrm{vol}\left( S_i \right)=\sum_{j\in S_i}d_j,\quad 1\leq i\leq K.
\end{equation*}
Solve the optimization problem
  \begin{equation}
    \label{eq:learning-group-actions-formulation}
    \min_{\left\{ S_1,\cdots,S_K \right\}\in \mathscr{X}_K}\frac{\displaystyle\max_{1\leq i\leq K} \nu \left( S_i \right)}{\displaystyle \min_{1\leq i\leq K} \mathrm{vol}\left( S_i \right)}
  \end{equation}
and output an optimal partition $\left\{ S_1,\cdots,S_K \right\}$ together with the minimizing $f\in C^{0}\left(\Gamma;G\right)$.
\end{problem}
\end{framed}

In the following, we shall refer to Problem~\ref{prob:lgas-defn} as \emph{Learning Group Actions by Synchronization} (LGAS). When $G=O \left( d \right)$ or $U\left( d \right)$ and $\mathrm{Cost}_G$ is the squared Frobenius norm on $d\times d$ matrices, $\nu \left( S_i \right)$ is clearly the frustration \eqref{eq:frustration-group-valued-vert-potential-defn} of the subgraph of $\Gamma$ with vertices in $S_i$, up to a multiplicative constant depending only on $\Gamma$ and dimension $d$. The minimizing vertex potential defines a synchronizable edge potential on the entire graph $\Gamma$, thus also gives rise to a cycle-consistent action on each partition. Note that the objective function \eqref{eq:learning-group-actions-formulation} does not account for the discrepancy between the realized synchronizable edge potential and the original $\rho$ on edges across partitions --- intuitively, solving Problem~\ref{prob:lgas-defn} amounts to forming partitions by economically ``dropping out'' appropriate edges in $\Gamma$ to minimize the total frustration.

\subsection{SynCut: A Heuristic Algorithm for Learning Group Actions by Synchronization}
\label{sec:heur-algor-learn}

In this subsection, we will investigate Problem~\ref{prob:lgas-defn} (LGAS) in the context of $O \left( d \right)$-synchronization problems, focusing on the simpler 
setting where $K=2$. In this case, \eqref{eq:learning-group-actions-formulation} simplifies to
\begin{equation*}
  \min_{S\subset V}\frac{\max \left\{ \nu \left( S \right), \nu \left( S^c \right) \right\}}{\min \left\{ \mathrm{vol}\left( S \right), \mathrm{vol}\left( S^c \right) \right\}}.
\end{equation*}
Note that
\begin{equation*}
  \max \left\{ \nu \left( S \right), \nu \left( S^c \right) \right\}\leq \nu \left( S \right) + \nu \left( S^c \right) \leq 2\max \left\{ \nu \left( S \right), \nu \left( S^c \right) \right\},
\end{equation*}
we can thus consider --- drawing an analogy with the standard approach of studying Cheeger numbers through normalized cuts --- the following optimization problem closely related with \eqref{eq:learning-group-actions-formulation}:
\begin{equation}
  \label{eq:frustration-cut-measure}
  \begin{aligned}
    \xi_{\Gamma}&:=\min_{S\subset V}\xi \left( S \right)\\
    &:=\min_{S\subset V}\,\,\left[\nu \left( S \right) + \nu \left( S^c \right) \right]\left( \frac{1}{\mathrm{vol}\left( S \right)}+\frac{1}{\mathrm{vol}\left( S^c \right)} \right).
  \end{aligned}
\end{equation}
Recall from \eqref{eq:frustration-group-valued-vert-potential-defn} that $\xi_{\Gamma}$ further simplifies into
\begin{equation}
  \label{eq:lgas-constant}
  \begin{aligned}
    \xi_{\Gamma}&=\min_{S\subset V}\inf_{g\in C^0 \left( \Gamma;O \left( d \right) \right)}\frac{1}{2d}\frac{1}{\mathrm{vol}\left( \Gamma \right)}\sum_{\substack{i,j\in V\\ \left( i,j \right)\notin \partial S}}w_{ij}\left\| g_i-\rho_{ij}g_j \right\|_{\mathrm{F}}^2\cdot \frac{\mathrm{vol}\left( \Gamma \right)}{\mathrm{vol}\left( S \right)\mathrm{vol}\left( S^c \right)}\\
    &=\min_{S\subset V}\inf_{g\in C^0 \left( \Gamma;O \left( d \right) \right)}\frac{1}{2d}\frac{1}{\mathrm{vol}\left( S \right)\mathrm{vol}\left( S^c \right)}\sum_{\substack{i,j\in V\\ \left( i,j \right)\notin \partial S}}w_{ij}\left\| g_i-\rho_{ij}g_j \right\|_{\mathrm{F}}^2,
  \end{aligned}
\end{equation}
where
\begin{equation*}
  \partial S:= \left\{ \left( u,v \right)\in E\mid u\in S, v\in S^c\textrm{ or }u\in S^c, v\in S \right\}.
\end{equation*}
In other words, the goal of solving the optimization problem \eqref{eq:lgas-constant} is to lower the total frustration of the graph $\Gamma$ by dropping out a minimum set of edges under the constraint that the residual graph consists of two connected components; this is equivalent to say that we seek a most economic graph cut in terms of reducing total frustration. To simplify statements, we shall refer to $\left\| g_i-\rho_{ij}g_j \right\|_{\mathrm{F}}^2$ as the \emph{frustration on edge $\left( i,j \right)\in E$} of vertex potential $g$ with respect to the edge potential $\rho$, and call the collection of frustrations on all edges the \emph{edge-wise frustrations}. The sum of all edge-wise frustrations will be referred to as the \emph{total frustration}.
% This quantity quantitatively measures the incompatibility of $g\in C^0 \left( \Gamma;G \right)$ and $\rho\in C^1 \left( \Gamma; G \right)$ on an individual edge.

Formulation \eqref{eq:lgas-constant} motivates a greedy algorithm that alternates between minimizing graph cuts and vertex potentials. We shall refer to this algorithm as \emph{Synchronization Cut}, or \verb|SynCut| for short; see Algorithm~\ref{alg:learning-group-actions}. We describe the main steps in \verb|SynCut| below:
\begin{enumerate}
{\setlength\itemindent{25pt} \item[\bf Step 1.] \emph{Initialization:} Input data include the weighted graph $\Gamma=\left(V,E,w\right)$, edge potential $\rho\in C^1 \left( \Gamma;G \right)$, and parameters required for
the spectral clustering subroutine plus termination conditions for the main loop. Initialize iteration counter $t=0$, and dynamic graph weights $\epsilon$ to be the input graph weights $w$;}

{\setlength \itemindent{25pt} \item[\bf Step 2.] \emph{Global Synchronization:} Synchronize the edge potential $\rho$ on the entire graph $\Gamma$ with respect to edge weights. Any synchronization algorithm can be used in this step, e.g. spectral relaxation \cite{BSS2013,CKS2015} or SDP relaxation \cite{BCSZ2014,Singer2011a,CKS2015,BKS2016,NaorRegevVidick2013}. Note that in this step the synchronization is performed on $\Gamma$ with dynamic weights $\epsilon$ instead of the original weights $w$. Denote $f^{\left( t \right)}$ as the
edge potential on $\Gamma$ at the edge potential at the $t$-th iteration;}

{\setlength \itemindent{25pt} \item[\bf Step 3.] \emph{Spectral Clustering (First Pass):} Update dynamic weights $\epsilon$ based on the frustration of $f^{\left( t \right)}$ on each edge by $$\epsilon_{ij}=w_{ij}\exp \left( -\frac{1}{\sigma}\left\|f^{\left( t \right)}_i-\rho_{ij}f^{\left( t \right)}_j\right\|^2_{\mathrm{F}}\right),\quad\textrm{where $\sigma>0$ is the average of all non-zero edge-wise frustrations}$$then partition the vertex set $V$ of graph $\Gamma$ into $K$ clusters $S_1,\cdots,S_K$ using spectral clustering based on the updated dynamic weights $\epsilon$. The goal is to cut the graph $\Gamma$ into more synchronizable clusters; edges causing large frustration are assigned relatively smaller weights $\epsilon_{ij}$ to increase the chance of being cut. To simplify notation, we will also use $S_{\ell}$ ($1\leq\ell\leq  K$) to denote the subgraph of $\Gamma$ spanned by the vertices in $S_{\ell}$;}

{\setlength \itemindent{25pt} \item[\bf Step 4.] \emph{Local Synchronization:} Synchronize the edge potential $\rho$ within each partition $S_{\ell}$, $1\leq\ell\leq K$. If we denote $\rho|_{S_{\ell}}$, $\epsilon|_{S_{\ell}}$ for the restrictions of $\rho$, $\epsilon$ to $S_{\ell}$, respectively, then this step solves the synchronization problem on weighted graph $\left(S_{\ell}, \epsilon|_{S_{\ell}}\right)$ for prescribed edge potential $\rho|_{S_{\ell}}$. Again, any synchronization algorithm can be used in this step. Denote $g^{\left( \ell \right)}$ for the resulting vertex potential on $S_{\ell}$;}

{\setlength \itemindent{25pt} \item[\bf Step 5.] \emph{Collage:} After obtaining $g^{\left( \ell \right)}$ for each local synchronization on $S_{\ell}$, we make a ``collage'' from these local solutions to form a global vertex potential defined on the entire graph $\Gamma$. Since each $g^{\left( \ell \right)}$ is obtained from synchronizing within $S_{\ell}$, the collected local solutions $\left\{ g^{\left( \ell \right)} \right\}_{\ell=1}^K$ generally incur large incompatibility (frustration) on edges across partitions. Our strategy is to find $K$ elements $h_1,\cdots,h_K\in G$, where each $h_{\ell}$ acts on $g^{\left( \ell \right)}$ by $g^{\left( \ell \right)}_u\mapsto g^{\left( \ell \right)}_uh_{\ell}$, $\forall u\in S_{\ell}$, such that the \emph{total cross-partition frustration}
\begin{equation}
  \label{eq:cross-partition-frustration}
  C \left(h_1,\cdots,h_K\,\big|\,\left\{ S_{\ell} \right\}_{1\leq\ell\leq K}, \left\{ g^{\left( \ell \right)} \right\}_{1\leq \ell \leq K}  \right):=\sum_{1\leq p\neq q\leq K}\sum_{\substack{\left( u,v \right)\in E\\u\in S_p, v\in S_q}}w_{uv}\left\| g^{\left( p \right)}_uh_p-\rho_{uv}g^{\left( q \right)}_vh_q \right\|_{\mathrm{F}}^2
\end{equation}
is minimized. Note that this is essentially synchronizing an edge potential
\begin{equation}
  \label{eq:cross-partition-edge-potential}
  \widetilde{\rho}_{pq}=
  \begin{cases}
    \displaystyle\sum_{\substack{\left( u,v \right)\in E\\u\in S_p, v\in S_q}}w_{uv}\left(g_u^{\left( p \right)}\right)^{-1}\rho_{uv}g_v^{\left( q \right)} & \textrm{if partitions $S_p$, $S_q$ are connected}\\
    \displaystyle 0 & \textrm{otherwise}
  \end{cases}
\end{equation}
on a reduced complete graph $\widetilde{\Gamma}_K$ consisting of $K$ vertices where each vertex represent one of the $K$ partitions $S_1,\cdots,S_K$. It thus simply requires calling the synchronization routine again to obtain $h_1,\cdots,h_K$, but this time the scale of the synchronization problem is often much smaller than the previous global and local synchronization steps. Also note that for the binary cut case $K=2$ and $G=O \left( d \right)$ this collage step is even simpler: it suffices to perform an single SVD on the $d\times d$ matrix
\begin{equation*}
  \sum_{\substack{\left( u,v \right)\in E\\u\in S_1,v\in S_2}}w_{uv}\left(g_u^{\left( p \right)}\right)^{-1}\rho_{uv}g_v^{\left( q \right)}=U\Sigma V^{\top}
\end{equation*}
and set $h_1=UV^{\top}$, $h_2=I_{d\times d}$.}

{\setlength \itemindent{25pt} \item[\bf Step 6.] \emph{Spectral Clustering (Second Pass):} Update dynamic weights $\epsilon$ based on the frustration of $f^{*}$ on each edge by $$\epsilon_{ij}=w_{ij}\exp \left( -\frac{1}{\sigma}\left\|f^{*}_i-\rho_{ij}f^{*}_j\right\|^2_{\mathrm{F}} \right),\quad\textrm{where $\sigma>0$ is the average of all non-zero edge-wise frustrations}$$then partition $\Gamma$ into $K$ clusters $S_1,\cdots,S_K$ using spectral clustering for a second time, based on the updated dynamic weights $\epsilon$.
}

{\setlength \itemindent{25pt} \item[\bf Step 7.] \emph{Repeat Step 2 to Step 6 Until Convergence.} The termination condition can be specified either by a maximum number of iterations or monitoring the change of the quantity
\begin{equation}
  \label{eq:objective-xi}
  \xi \left( \left\{ S_1,\cdots,S_K \right\} \right):=\left( \sum_{\ell=1}^K\nu \left( S_{\ell} \right) \right)\left( \sum_{k=1}^K \frac{1}{\mathrm{vol}\left( S_{k} \right)} \right).
\end{equation}
At the end of the procedure, return the partitions $\left\{ S_1,\cdots,S_K \right\}$ and the final edge potential $f^{*}$ from the most recent updates. The cycle-consistent edge potential on partition $S_{\ell}$ is encoded in the restriction of $f^{*}$ to $S_{\ell}$.}

\end{enumerate}

\begin{algorithm}
\caption{{\sc Synchronization Cut}: Learning Group Actions by Synchronization}
\label{alg:learning-group-actions}
 \begin{algorithmic}[1]
   \Procedure{SynCut}{$\Gamma$, $\rho$, $K$}\Comment{weighted graph $\Gamma=\left( V,E,w \right)$, $\rho\in C^1\left( \Gamma; G \right)$, number of partitions $K$}
   \State{$t=0$}
   \State{$\epsilon=w$}
   \While{not converge}
     \State $f^{\left( t \right)}\in C^{0}\left(\Gamma;G\right)\gets$ {\sc Synchronize}$\left( \Gamma,\rho,\epsilon \right)$
     \State $\displaystyle\sigma\gets\textrm{ average non-zero edge-wise frustrations of $f^{\left( t \right)}$}$
     \For{$\left( i,j \right)\in E$}\Comment{calculate weights $\epsilon$ on graph $\Gamma$ for spectral clustering}
       \State $\epsilon_{ij}\gets w_{ij}\exp\left(-\displaystyle\frac{1}{\sigma}\left\| f^{\left( t \right)}_i-\rho_{ij}f^{\left( t \right)}_j \right\|_{\mathrm{F}}^2\right)$
     \EndFor
     \State $\left\{S_1,\cdots,S_K\right\}\gets$ {\sc SpectralClustering}$\left( \Gamma,\epsilon \right)$
     \For{$\ell=1,\cdots,K$}
       \State $g^{\left( \ell \right)}\in\Omega^0 \left( S_{\ell}; G \right)\gets$ {\sc Synchronize}$\left( S_{\ell}, \rho|_{S_{\ell}}, \epsilon|_{S_{\ell}}\right)$
     \EndFor
     \State $f^{*}\in\Omega^0 \left( \Gamma; G \right)\gets$ {\sc Collage}$\left( \left\{ S_{\ell} \right\}_{\ell=1}^K,\left\{ g^{\left( \ell \right)} \right\}_{\ell=1}^K \right)$
     \State $\displaystyle\sigma\gets\textrm{ average non-zero edge-wise frustrations of $f^{*}$}$
     \For{$\left( i,j \right)\in E$}\Comment{update weights $\epsilon$ on graph $\Gamma$ for next iteration}
       \State $\epsilon_{ij}\gets w_{ij}\exp\left(-\displaystyle \frac{1}{\sigma}\left\| f^{*}_i-\rho_{ij}f^{*}_j \right\|_{\mathrm{F}}^2\right)$
     \EndFor
     \State $\left\{S_1,\cdots,S_K\right\}\gets$ {\sc SpectralClustering}$\left( \Gamma,\epsilon \right)$
     \State $t\gets t+1$
   \EndWhile
 \State \Return $\left\{ S_1,\cdots,S_K \right\}$, $f^{*}$\Comment{$f^{*}$ defines a cycle-consistent edge potential on each partition}
   \EndProcedure
 \end{algorithmic}
\end{algorithm}

\subsection{Results on Simulated Random Synchronization Networks}
\label{sec:numer-exper-synth}

In this subsection, we use simulations to provide some intuition for the behavior of SynCut under the setting $K=2$ (two partitions). We first specify a random procedure to simulate input data --- a connected random graph with a prescribed edge potential --- for synchronization problems. In addition, the random graph generation procedure will be controlled by a parameter that allows us to adjust the level of obstruction to the synchronizability of the prescribed edge potential over the generated graph. We then specify the metrics used for performance measure. We conclude by demonstrating that the partition generated from SynCut recovers the two synchronizable connected components with high accuracy and within relatively few numbers of iterations. For the simplicity of statements, we refer to each pair of generated graph and edge potential an instance of a \emph{random synchronization network}.

\subsubsection{Random Synchronization Network Simulation}
\label{sec:random-sync-network}

We first specify the procedure to generate the random graphs. Our intention is to sample random graphs with sufficiently variable spectral gaps, based on the intuition that a large spectral gap of the underlying graph results in greater obstruction to the synchronizability of the edge potential constructed by the procedures that will soon be described in this subsection. We first generate two partitions $S_1,S_2$ with an equal number of vertices. Each partition is a connected component built from a vertex degree sequence of random integers uniformly distributed in an interval (say $5$ to $8$), adapting an algorithm first proposed in \cite{BlitzsteinDiaconis2010}; when the interval is a single integer, the connected component is a regular graph. Random edges are than added to link the two partitions $S_1,S_2$. The number of inter-component random edges positively correlates with the spectral gap, as shown in Figure~\ref{fig:simulation-histograms-supp}, suggesting that this number can be used as a parameter to adjust the level of obstruction to cutting the graph into two connected components $S_1$ and $S_2$.

A subtlety in this random network generation procedure is that a uniform distribution on the number of inter-component links does not induce a uniform distribution on the spectral gaps of the generated random graphs, due to concentration effects. A precise characterization of the distribution of spectral gaps in our random graph model is interesting on its own right but beyond the scope of this paper. We refer interested readers to the existing literature on the spectral gaps of random graphs such as \cite{ChungLuVu2003,CoL2006,HoffmanKahlePaquette2012}. In practice, we simply use a large number of random trials to generate sufficiently many sample graphs with spectral gaps within desirable ranges; see Figure~\ref{fig:simulation-histograms}a.

After drawing an instance of the random graph, we randomly construct an edge potential that is synchronizable within $S_1$ and $S_2$, but not necessarily synchronizable on the inter-component links. The procedure to generate the random edge potentials proceeds as follows:
\begin{enumerate}[(1)]
\item Randomly generate a vertex potential $g\in C^0 \left( \Gamma;G \right)$ for the entire graph $\Gamma$;
\item Set the value of $\rho$ on edge $\left( i,j \right)$ according to
  \begin{equation*}
    \rho_{ij}=
    \begin{cases}
      g_ig_j^{-1} & \textrm{if both $i,j\in S_1$ or $i,j\in S_2$,}\\
      \textrm{a random matrix in $O \left( d \right)$} & \textrm{otherwise.}
    \end{cases}
  \end{equation*}
%\item Set $\rho_{ij}=$ on each edge $\left( i,j \right)$;
%\item For each edge between the partitions, substitute $\rho_{ij}$ with a randomly generated orthogonal matrix in $O \left( d \right)$.
\end{enumerate}
The vertex potential $g\in C^0 \left( \Gamma;G \right)$ will no longer attain the minimum frustration for the entire graph with respect to the prescribed edge potential $\rho$, due to the edges added between the partitions that are much less likely synchronizable.

\begin{figure}[htbp]
\centering
\includegraphics[width=0.75\textwidth]{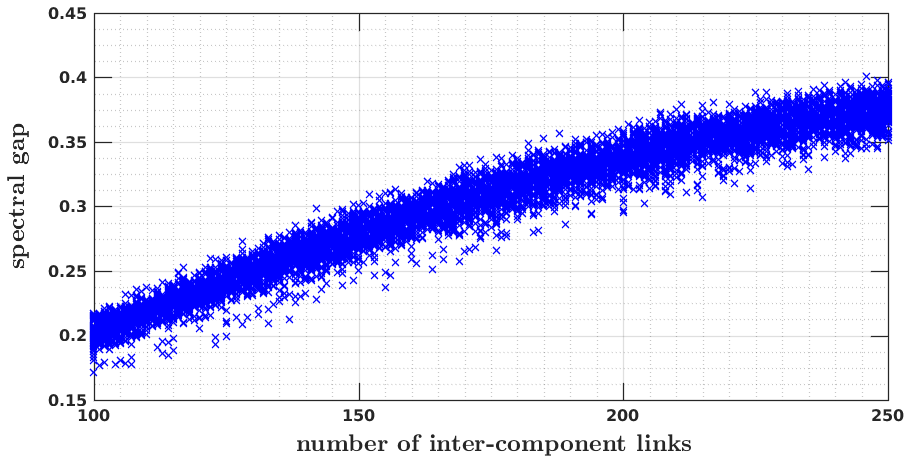}
\caption{\small A scatter plot displaying the correlation between the number of inter-component links and the spectral gap in our random graph model, with $N=100$ vertices and the (integer) number of inter-component links uniformly distributed between $100$ and $250$.}
\label{fig:simulation-histograms-supp}
\end{figure}

We consider each run of SynCut as successful if both output partitions are synchronizable connected components, i.e. if SynCut recovers the original partitions $S_1, S_2$. The performance of SynCut is measured using the \emph{error ratio} computed by dividing the number of erroneously clustered vertices by the total number of vertices. If SynCut successfully recovers $S_1,S_2$, this error ratio is $0$; if the output partition is close to a random guess, or if the algorithm fails to separate the vertices into distinct clusters, the error ratio is $0.5$. The error ratios of the partitions output from SynCut are then compared with a baseline graph cutting algorithm using normalized graph cut (NCut) that does not utilize any information in the prescribed edge potential; see e.g. \cite{ShiMalik2000,vonLuxburg2007}. We refer interested readers to \cite{BGHHL2018} for an initial attempt at analyzing this phenomenon for the scenario where the group $G$ is a symmetric group (group of permutations) and the underlying graph is generated from a stochastic block model.

\subsubsection{Simulation Results}
\label{sec:simulation-results}

\begin{figure}[htbp]
    \centering
    \includegraphics[width=1.0\textwidth]{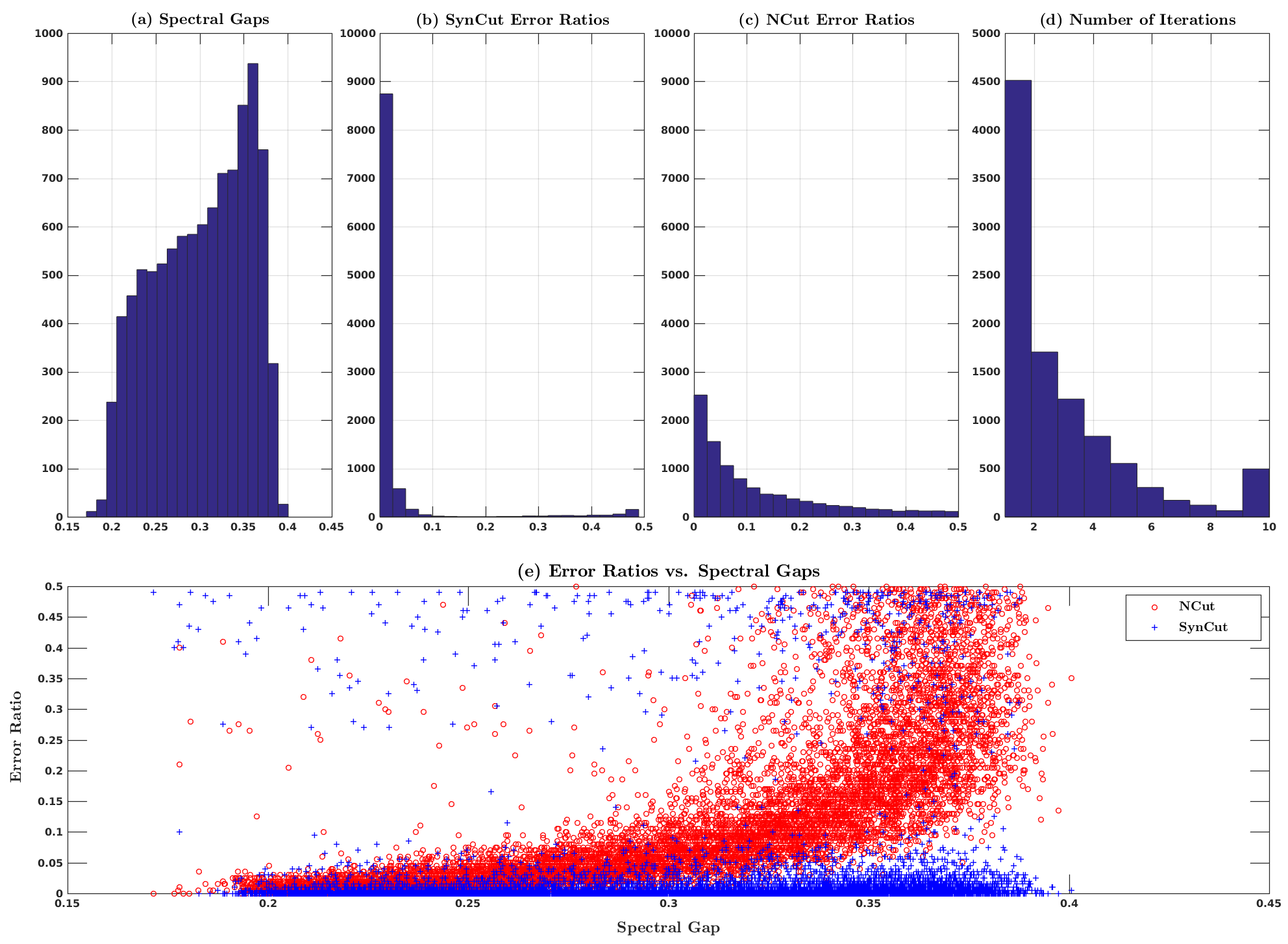}
\caption{\small\emph{(a)} Histogram of the spectral gap of the $10,000$  random graphs drawn from our model. \emph{(b)} Histogram of the error ratios of the SynCut  clustering results. \emph{(c)} Histogram of the error ratios of the baseline NCut clustering results. \emph{(d)} Histogram of the number of iterations for SynCut. \emph{(e)} Scatter plots of the error ratios of SynCut and NCut versus spectral gap. }
\label{fig:simulation-histograms}
\end{figure}

In this simulation study, we set the number of vertices in each of the two synchronizable components to $N=100$, and the entries of the vertex degree sequence of random integers are independently uniformly distributed between $4$ and $8$. The number of inter-component links between the two synchronizable components is drawn uniformly between $100$ and $250$. The edge potentials are valued in the orthogonal group $O \left( d \right)$ with $d=5$. We terminate SynCut either after $10$ iterations or if the change in the value of the objective function $\xi$ [see \eqref{eq:objective-xi}] between consecutive iterations falls below a preset tolerance of $10^{-8}$. We plot in Figure~\ref{fig:simulation-histograms}(a) the spectral gaps of $10,000$ realizations of our random network model. In Figure~\ref{fig:simulation-histograms}(b) and Figure~\ref{fig:simulation-histograms}(c) we observed that the error ratios in these $10,000$ runs of SynCut tend to be much smaller than NCut, suggesting that SynCut outputs more accurate partitions with respect to synchronizability. In Figure~\ref{fig:simulation-histograms}(e), we again see that SynCut outperforms NCut and the amount of improvement increases with the magnitude of the spectral gap. Figure~\ref{fig:simulation-histograms}(d) shows that SynCut converges quickly.

\begin{figure}[htbp]
\centering
\includegraphics[width = 1.0\textwidth]{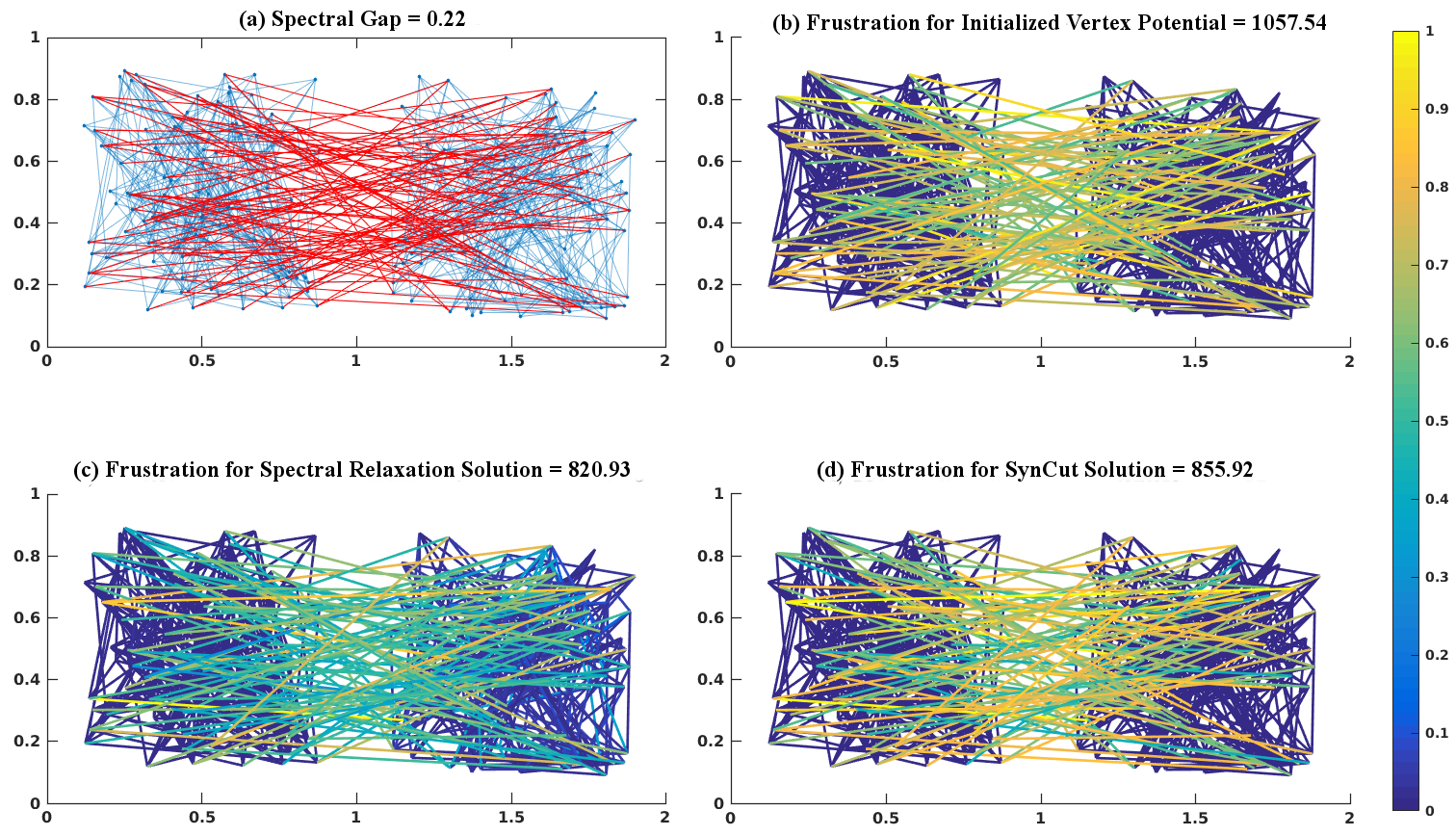}
\caption{\small \emph{(a)} A random graph consisting of two synchronizable connected components, each with $100$ vertices and $250$ edges (in blue), and $100$ non-synchronizable inter-component edges (in red). The edge potential takes value in the orthogonal group $O \left( 5 \right)$. All vertex degrees in each synchronizable component are set to $5$. \emph{(b)} Edge-wise frustration for the vertex potential $g$ used to generate the prescribed edge potential $\rho$. As expected, frustration is small within each connected component but large between components. \emph{(c)} Edge-wise frustration for the vertex potential obtained from spectral relaxation \cite{BSS2013}. The total frustration is lower than that in the top right figure, but the inter-component edges carries relatively lower frustration since the relaxation procedure tends to ``spread'' the non-synchronizability across the entire graph. \emph{(d)} Edge-wise frustration for the vertex potential obtained from SyncCut. The total frustration is higher than that for the spectral relaxation solution, but the distribution of frustrations on the edges is closer to that of vertex $g$ and can thus be used to recover the synchronizable connected components.}
\label{fig:synthetic-datasets}
\end{figure}

We now focus on a particular instance  of a random synchronization network to better understand SynCut in comparison with the spectral relaxation algorithm proposed in \cite{BSS2013}. Each synchronizable component in the random synchronization network shown in Figure~\ref{fig:synthetic-datasets}(a) is a regular graph containing $N=100$ vertices and $250$ edges, generated with a constant vertex degree sequence of $5$. We color the edges within and between synchronizable components in blue and red, respectively. In Figure~\ref{fig:synthetic-datasets}(b) we plot the edge-wise frustration for the vertex potential $g$ used to generate the edge potential $\rho$ prescribed to the network. As expected, the frustration is zero within each synchronizable component but large on the edges across components. Figure~\ref{fig:synthetic-datasets}(c) and Figure~\ref{fig:synthetic-datasets}(d) show the edge-wise frustrations for two vertex potentials obtained from the spectral relaxation algorithm \cite{BSS2013} and SynCut, respectively. Though the total frustration is larger for SynCut than spectral relaxation, the SynCut solution concentrates most of the frustration on the non-synchronizable inter-component edges, with a distribution of edge-wise frustrations closer to the distribution for the initial vertex potential $g$. This suggests that applying a spectral graph cut algorithm using the edge-wise frustration of the SynCut solution as a dissimilarity measure is advantageous, as the distribution in Figure~\ref{fig:synthetic-datasets}(d) identifies the obstructions to synchronizability in the synchronization network more accurately.

\section{Application to Automated Geometric Morphometrics}
\label{sec:example}

In this section we formulate a problem in automated geometric morphometrics in terms of LGAS, then apply the SynCut algorithm to provide a solution. In Section~\ref{sec:morph-datas} we provide some background in geometric morphometrics and its relation to synchronization problems. In Section~\ref{sec:clust-synchr} we apply SynCut to a collection of second mandibular molars of prosimian primates and non-primate close relatives. The morphological hypothesis is that the geometric traits of second mandibular molars cluster into $3$ dietary regimens: \emph{folivorous} (herbivores that eat leaves), \emph{frugivorous} (herbivores or omnivores that prefer fruit), and \emph{insectivorous} (a carnivore that eats insects). We will show the SynCut result, which is based on the synchronizability of pairwise correspondences, and compare it with a distance-based clustering result using diffusion maps \cite{CoifmanLafon2006}.

\subsection{Geometric Morphometrics and Synchronization}
\label{sec:morph-datas}

The classic tool in geometric morphometrics is \emph{Procrustes analysis}. The basic assumption underlying this analysis framework is that most of the geometric information on each shape can be efficiently encoded in a set of landmark points carefully picked to highlight the morphometrical phenotypes (variation in the geometric shape of an organism). The \emph{Procrustes distance} between two shapes is the average Euclidean distance between corresponding landmarks, after applying a rigid motion (rotations, reflections, translations, and their compositions) to optimally align the two sets of landmarks. If all the shapes are marked with an equal number of landmarks but the landmark correspondence is not known \emph{a priori}, a combinatorial search can be performed over all permutations of one-to-one landmark correspondences, and the minimum average Euclidean distance between corresponding landmarks can be taken as a dissimilar measure between the two shapes. Comparing a pair of shapes in this framework thus yields abundant pairwise information, including a scalar dissimilarity score, a rigid motion, and a permutation matrix encoding the one-to-one landmark correspondence.

In automated geometric morphometrics, landmark points are not used to represent the shapes, and algorithms search for an ``optimal'' transformation between a pair of whole shapes directly by minimizing energy functionals over a set of admissible transformations. Depending on the specific class of transformations and energy functional, the pairwise comparisons produce different types of correspondences between surfaces, such as conformal/quasiconformal transformations, isometries, area-preserving diffeomorphisms, or even transport-plans between surface area measures in a Wasserstein framework. Regardless of the type of admissible transformations, the algorithm can output a rigid motion for the optimal alignment between two shapes, as well as a dissimilarity or similarity score for such an alignment. See Figure~\ref{fig:molars} for an example of representing a collection of shapes using landmarks versus triangular meshes.

\begin{figure}[htbp]
\centering
\includegraphics[width = 0.9\textwidth]{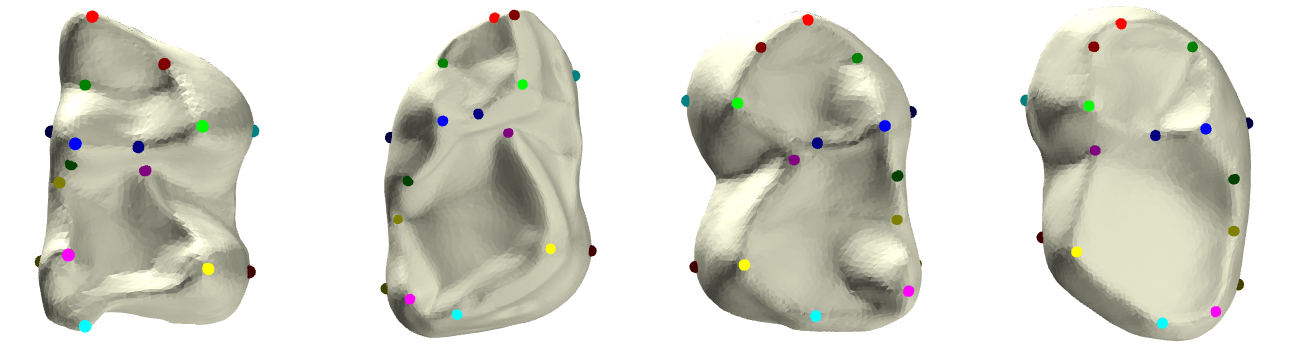}
\caption{\small The meshes of four lemur molars from an an anatomical surface dataset  first published in \cite{PNAS2011}. The colored dots on the molars are landmark points where 
identical colors indicate corresponding landmarks.}
\label{fig:molars}
\end{figure}

When the analysis is extended from comparing a single pair to a large collection of shapes, a crucial premise for downstream statistical analysis (e.g. \emph{General Procrustes Analysis} (GPA) \cite{Gower1975,DrydenMardia1998SSA,GowerDijksterhuis2004GPA}) is that the pairwise correspondences be \emph{cycle-consistent}, meaning that propagating any landmark on any shape by consecutive correspondences along a close cycle of shapes should land exactly at the original landmark. Traditional landmark-based Procrustes analysis begins with consistently picking an equal number of landmarks on each shape, resulting in a large amount of pairwise correspondence relations that are cycle-consistent by construction. This is, however, not the situation with automated geometric morphometrics, where the correspondence transformations produced by automated algorithms are rarely cycle-consistent, even when one localizes the transformations within relatively ``stable'' regions where landmarks are affixed with the knowledge of an experienced geometric morphometrician. The necessity of cycle-consistent correspondences links automated geometric morphometrics to synchronization problems. An automated geometric morphometric algorithm will output for each pair of shapes a triplet consisting of a dissimilarity score, a rigid motion, and a pairwise transformation. We can use the dissimilarity scores to define a weighted graph $\Gamma$ that captures the similarities within the collection, both qualitatively and quantitatively, by adjusting the number of nearest neighbors of each vertex and the weights on each edge. The rigid motions and pairwise transformations define two edge potentials on $\Gamma$, taking values in different groups. We list below some interesting synchronization problems arising from this formulation:

\textbf{Three-Dimensional Euclidean Group.} The rigid motions $R_{ij}$ between shapes $S_i$ and $S_j$ that share an edge in $\Gamma$ define an edge potential $R\in C^1 \left( \Gamma; E \left( 3 \right) \right)$, where $E \left( 3 \right)$ is the three-dimensional Euclidean group. Solving an $E \left( 3 \right)$-synchronization problem over $\Gamma$ with respect to $R$ results in a globally consistent alignment for a collection of shapes, which is often crucial for initializing geometric morphometrical analysis algorithms such as \emph{Dirichlet Normal Energy}~\cite{BBLSJD2011}, \emph{Orientation Patch Analysis}~\cite{EWFJ2007}, and \emph{Relief Index}~\cite{Boyer2008}. Algorithms that automatically align a collection of anatomical shapes in a globally consistent manner can also be viewed as primitive approaches for solving $E \left( 3 \right)$-synchronization problems; see e.g. \cite{PuenteThesis2013,Auto3dGM2015,MolaR2016,GBJLBP2016}.

\textbf{Orthogonal Group and Orientability Detection.} If the shapes are preprocessed to superimpose the centers of mass at the same point, the translation component of each $R_{ij}$ output from a pairwise landmark-based Procrustes analysis vanishes\footnote{Note this is not the case for jointly analyzing a collection of shapes in a landmark-based Procrustes analysis framework; see e.g. \cite{CKS2015}.}. This reduces the global alignment problem to standard synchronization problems over the compact Lie group $O \left( 3 \right)$. Spectral and semidefinite programming (SDP) relaxation methods can then be applied directly to solve the global alignment problem. If we consider the edge potential $\rho\in C^1 \left( \Gamma; \mathbb{Z}_{2} \right)$ defined by $\rho_{ij}=\det R_{ij}$, a $\mathbb{Z}_2$-synchronization solution can be used to either partition the dataset into ``left-handed'' and ``right-handed'' subsets or conclude that such an orientation-based partition does not exist. We stated a similar situations in Example~\ref{example:spins}; other examples in this setting can be found in applications of \emph{Orientable Diffusion Maps}~\cite{SingerWu2011ODM}.

\textbf{Automorphism Groups.} Certain classes of transformations $C_{ij}$ between each pair of shapes $S_i$, $S_j$ give rise to an edge potential on the graph $\Gamma$ valued in an automorphism group of a canonical domain. For instance, algorithms such as \emph{M\"obius Voting} and the \emph{Continuous Procrustes Distance}~\cite{CP13} between disk-type surfaces rely on the computation of conformal maps between two shapes, based on uniformization parametrization techniques \cite{PinkallPolthier1993,AHTK1999} that map each surface conformally to a canonical unit disk on the plane. By intertwining $C_{ij}$ with the parametrizations of the source and the target shape, the correspondence between $S_i$and $S_j$ can be equivalently considered as an element of the conformal automorphism group $\Aut \left( \mathcal{D} \right)$ of the planar unit disk $\mathcal{D}$. The group $\Aut \left( \mathcal{D} \right)$ is isomorphic to the \emph{projective special linear group} $\PSL \left(2,\mathbb{R}  \right)$, a non-compact simple real Lie group that is equivalent to the quotient of the \emph{special linear group} $\SL \left( 2,\mathbb{R} \right)$ by $\left\{\pm I_2\right\}$, where $I_2$ denotes the $2\times 2$ identity matrix. Synchronization problems over $\PSL \left( 2,\mathbb{R} \right)$ or $\SL \left( 2,\mathbb{R} \right)$ require non-trivial extensions of the non-unique games (NUG) framework \cite{BCS2015NUG} over compact Lie groups.

\textbf{Groupoids.} Other types of transformations $C_{ij}$ between each pair of shapes $S_i,S_j$ require further generalizations of the synchronization framework to
edge potentials taking values in a \emph{groupoid} rather than in a group. As an example, consider surface registration techniques based on area-preserving maps \cite{CP13,ZhuHakerTannenbaum2003,AreaPresOT2013,SZSWSG2013}. These techniques use conformal or area-preserving parametrizations to push forward surface area measures on $S_i,S_j$ to measures $\mu_i,\mu_j$ on the planar unit disk $\mathcal{D}$, respectively, then solve for a \emph{transport map} on $\mathcal{D}$ that pulls back $\mu_j$ to $\mu_i$ (or equivalently $\mu_i$ to $\mu_j$). To formulate such ``transport-map-valued'' edge potentials in a synchronization framework, an edge potential should be allowed to take values in different classes of maps on different edges, with the only constraint that maps on consecutive edges can be composed; these ingredients have much in common in spirit with \emph{fundamental groupoids} \cite{Brown1967,BHS2011} and Haefliger's \emph{complexes of groups} \cite{Haefliger1991,Haefliger1992}. Such a generalized framework for synchronization problems can also be used to analyze correspondences $\left\{C_{ij}\right\}$ that are \emph{soft maps} \cite{RangarajanChuiBookstein1997,SNBBG2012SoftMaps} or \emph{transport plans} \cite{LipmanDaubechies2011,LipmanPuenteDaubechies2013,LaiZhao2014}, where one replaces the set of transport maps between $\mu_i$ and $\mu_j$ with (probabilistic) couplings $\Pi \left( \mu_i,\mu_j \right)$ as in Kantorovich's relaxation to the Monge optimal transport problem \cite{Villani2003,Villani2008}. The \emph{Horizontal Diffusion Maps} (HDM) framework \cite{HDM2016} and the application in automated geometric morphometrics \cite{Gao2015Thesis} are among the initial attempts in this direction.% in the context of synchronization problems.

\subsection{Clustering Lemurs by Dietary Regimens using Synchronizability of Molar Surfaces}
\label{sec:clust-synchr}

We focus on a real anatomical surface mesh dataset of second mandibular molars from $5$ genera of prosimian primates and nonprimate close relatives. There are a total of $50$ molars with $10$ specimens from each genus; see Figure~\ref{fig:SyncAlignment}. The five genera divide into three dietary regimens: the \emph{Alouatta} and \emph{Brachyteles} are folivorous, the \emph{Ateles} and \emph{Callicebus} are frugivorous, and the \emph{Saimiri} are insectivorous. In Figure~\ref{fig:molars} we display four lemur molars from an anatomical surface dataset first published in \cite{PNAS2011}, together with landmarks on each molar placed by evolutionary anthropologists. Similar datasets have been studied in a series of papers developing algorithms for automatic geometric morphometrics \cite{PNAS2011,CP13,LipmanDaubechies2011,LipmanPuenteDaubechies2013,KoehlHass2015,GKD2019,GKBD2019}. The chewing surface of each molar is digitized as a two-dimensional triangular mesh in $\mathbb{R}^3$ of disk-type topology (i.e. conformally equivalent with a planar disk).  We will apply SynCut to these $50$ molars and examine if the clustering is consistent with dietary regimens.

\begin{figure}[htbp]
\centering
\includegraphics[width=1.0\textwidth]{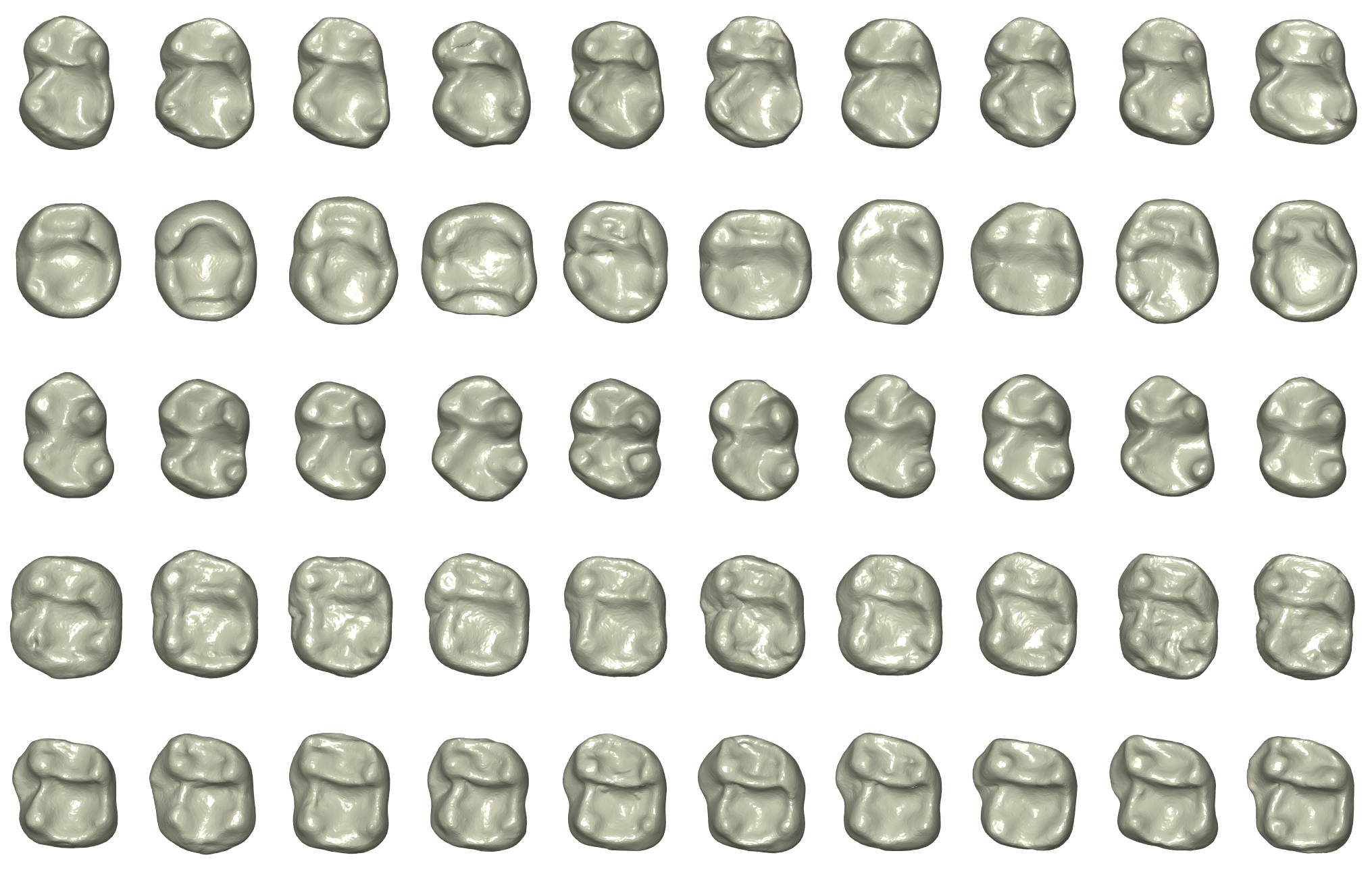}
\caption{\small Consistent alignment of $50$ lemur teeth based on applying SynCut to all pairwise alignments from the continuous Procrustes analysis~\cite{CP13}. 
Each row corresponds to teeth from a genus, from top to bottom:  \emph{Alouatta}, \emph{Ateles}, \emph{Brachyteles}, \emph{Callicebus}, \emph{Saimiri}.}
\label{fig:SyncAlignment}
\end{figure}

\paragraph{Method} We first pre-process the dataset by translating and scaling each shape so that all surface meshes center at the origin and enclose unit surface area. We then apply the continuous Procrustes distance algorithm for each pair of teeth, generating a a distance score $d_{ij}$, a diffeomorphism $C_{ij}$, and an orthogonal matrix $R_{ij}\in O \left( 3 \right)$ that optimally aligns $S_j$ to $S_i$ with respect to the diffeomorphism $C_{ij}$. We use the distance scores to construct a weighted $K$-nearest-neighbor graph $\Gamma$. The  weights are defined as $w_{ij}=\exp \left( -d_{ij}^2\Big/\sigma^2 \right)$ with the bandwidth parameter  $\sigma>0$ set to be of the order of the average smallest non-zero distances.
%an empirical criterion suggested in \cite{LafonThesis2004} and used in spectral clustering.
We apply SynCut to the edge potential $\rho\in C^1 \left( \Gamma;O \left( 3 \right) \right)$ defined by the alignments $R_{ij}$ on the weighted graph $\left( \Gamma,w \right)$. Finally, we compare the clustering performance of SynCut with applying diffusion maps and spectral clustering directly to the weighted graph without the alignment information.

\paragraph{Results}  SynCut and diffusion maps both require the choice of a parameter $K$ determining the number of nearest neighbors in the construction of the graph $\Gamma$. When $6\leq K \leq 10$ both procedures accurately cluster the $50$ molars in the dataset into the three distinct dietary regimens, see Figure~\ref{fig:hdm-diet-tsne} for the two-dimensional embedding plots for $K=7$.  SynCut produces slightly tighter and more distinguishable species clusters. Not surprisingly, for $K>10$ --- when the number of nearest neighbors exceeds the number of specimens in each genus --- both algorithms are less accurate as $K$ increases, with the accuracy of SynCut dropping faster than diffusion maps. This empirical observation is
 consistent with our intuition that the performance of SynCut is more sensitive to increased spectral gaps than diffusion maps.

% \begin{table}[htbp]
% \centering
% \caption{SynCut Clustering Result for $50$ Lemur Teeth}
% \label{table:clustering-confusion-matrix}
% \begin{tabular}{|c|c|c|c|}
% \hline
%  & Folivore & Frugivore & Insectivore \\
% \hline
% Cluster 1 & $20$ & $1$ & $0$ \\
% \hline
% Cluster 2 & $0$ & $19$ & $0$ \\
% \hline
% Cluster 3 & $0$ & $0$ & $10$ \\
% \hline
% \end{tabular}
% \end{table}

\begin{figure}[htbp]
\centering
\includegraphics[width=1.0\textwidth]{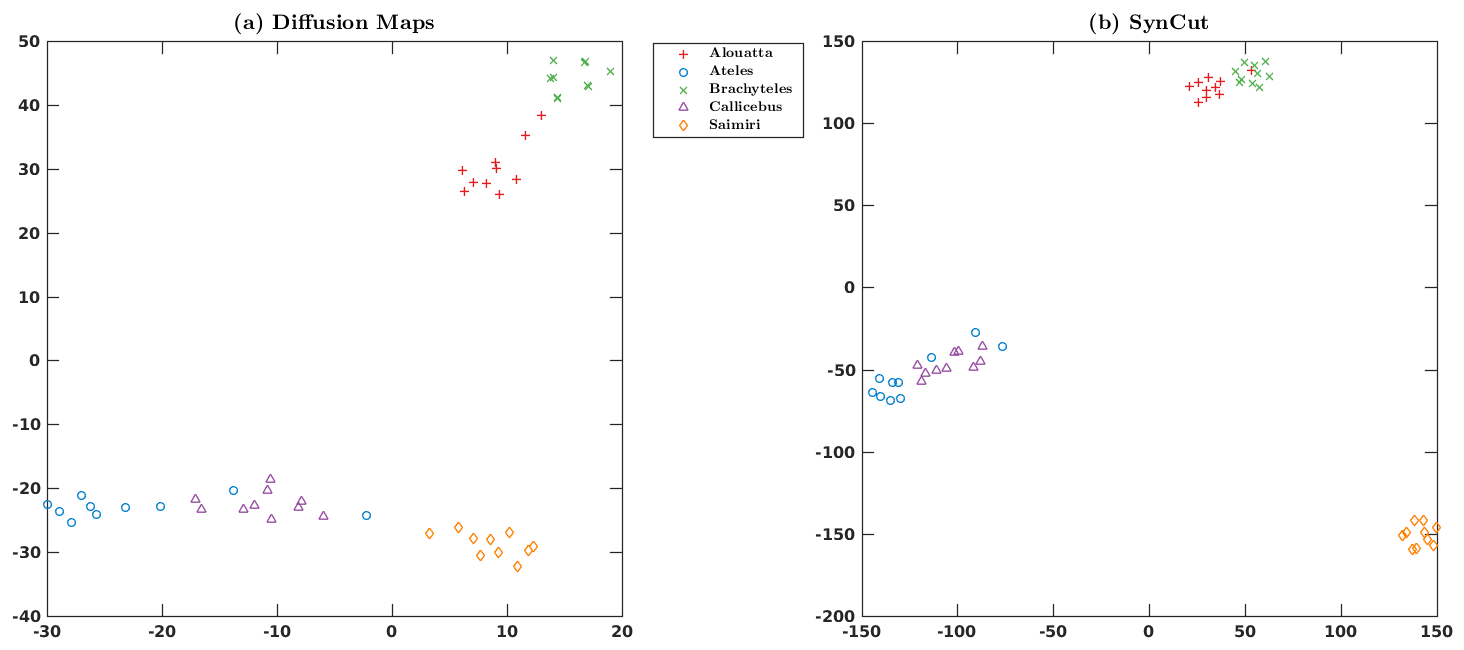}
\caption{\small Embeddings of the $50$ lemur teeth dataset into $\mathbb{R}^2$, obtained by applying diffusion maps (left) and SynCut (right) to the $7$-nearest-neighbor graph. Both plots are post-processed using t-SNE \cite{MH2008tSNE}. \emph{(a)} Diffusion maps applied to the weighted graph $\left(\Gamma,w\right)$ successfully distinguishes three diet groups, but the genera are less distinguishable. \emph{(b)} SynCut produces an edge-wise frustration matrix after the final iteration that can be used by diffusion maps to generate a low-dimensional embedding, in which both dietary groups and genera are more distinguishable.}
\label{fig:hdm-diet-tsne}
\end{figure}

\section{Conclusion and Discussion}
\label{sec:concl-discu}

We provided in this paper a geometric framework for synchronization problems. We first related the synchronizability of an edge potential on a connected graph to the triviality of a flat principal bundle over the topological space underlying the graph, then characterized synchronizability from two aspects: the holonomy of the principal bundle, and the twisted cohomology of an associated vector bundle. On the holonomy side, we established a correspondence between two seemingly distant objects on a connected graph $\Gamma$, namely, the orbit space of the action of $G$-valued vertex potentials on $G$-valued edge potentials, and the representation variety of the fundamental group of $\Gamma$ into $G$; on the cohomology side, we built a twisted de Rham cochain complex on an associated vector bundle $\mathscr{B}_{\rho}\left[ F \right]$ of the synchronization principal bundle $\mathscr{B}_{\rho}$, of which the zero-th degree cohomology group characterizes the obstruction to the synchronizability of the prescribed edge potential.

With the presence of a metric on the associated vector bundle $\mathscr{B}_{\rho}\left[ F \right]$, we also developed a twisted Hodge theory on graphs. Independent of the contribution to synchronization problems, this theory is both a discrete version of the Hodge theory of elliptic complexes and a fibre bundle analogue of the discrete Hodge theory on graphs. Specifically for synchronization problems, this twisted Hodge theory realizes the graph connection Laplacian operator as the zero-th degree Hodge Laplacian in the twisted de Rham cochain complex. A Hodge-type decomposition theorem is also proven, stating that the image of the twisted codifferential is the orthogonal complement of the linear space of $F$-valued synchronization solutions, with respect to the bundle metric.

Motivated by the geometric intuitions gained from these theoretical results, we coined the problem of learning group actions (LGA), and proposed a heuristic algorithm, which we referred to as SynCut, based on iteratively applying synchronization and spectral graph techniques. Numerical simulations on synthetic and real datasets indicated that SynCut has the potential to cluster a collection of objects according to the synchronizability of a subset of partially observed pairwise transformations.

We conclude this paper by listing several problems of interest for future exploration. These are only a subset of a vast collection of potential directions concerning the mathematical, statistical, and algorithmic aspects of synchronization problems:
\begin{enumerate}[1)]
\item \emph{The Representation Variety of Synchronization Problems.} When a prescribed edge potential $\rho$ is not synchronizable over graph $\Gamma$, the goal of the synchronization problem is to find a synchronizable edge potential $\tilde{\rho}$ that is as close as possible to $\rho$ in a sense that has been made clear in this paper. The point of view adopted in Section~\ref{sec:synchr-probl} is that the synchronization problem essentially concerns the orbits of $\rho$ and $\tilde{\rho}$ under the action of all vertex potentials. It is natural to conceive a synchronization algorithm based on the geometry of the orbit space $C^1 \left( \Gamma;G \right)/C^0 \left( \Gamma;G \right)$ that enables efficiently ``moving across'' the orbits. Since the fundamental group of any connected graph is simply a free product of copies of $\mathbb{Z}$, we expect the representation variety $\Hom \left( \pi_1 \left( \Gamma \right), G \right)/G$ to possess relatively simple structures that could be used for guiding the design of novel synchronization algorithms with provable guarantees.
\item \emph{Higher-order Synchronization Problems.} As a simplicial complex, the graph $\Gamma$ only has $0$- and $1$-simplices, which results in only one cohomology group of interest in the de Rham cochain complex \eqref{eq:twisted_chain_complex_with_adjoint}. By extending the twisted de Rham and Hodge theory developed in Section~\ref{sec:geom-synchr-probl} to simplicial complexes of higher dimensions, we expect higher-order synchronization problems can be formulated and studied using tools and insights from high-dimensional expanders and the Hodge theory of elliptic complexes. Generalizing the current regime of synchronization problems, in which only pairwise transformations are considered, the higher-order analogies would enable the study of relations and interactions among multiple vertices in the graph $\Gamma$, which potentially opens doors towards higher-order graphical models and related statistical inference questions as well.
\item \emph{Hierarchical Partial Synchronization Algorithms with Provable Guarantees.} The SynCut algorithm we proposed in this paper can be understood as an iterative hierarchical partial synchronization algorithm, based on the assumption that edge-wise synchronization is an indicator of the synchronizability of a prescribed edge potential over a proper subgraph. The numerical experiments on synthetic and real datasets suggested the validity of this intuition under our random graph model, but no provable guarantees exist either for the convergence or the effectiveness of algorithms similar or related to SynCut, to the best of our knowledge. Building a Cheeger-type inequality as the performance guarantee for SynCut attracted our attention, but even the analogy of Cheeger number (or graph conductance) in the setting of SynCut or LGA is not clear --- whereas the Cheeger number depends only on the graph weights, which are fixed numbers on each edge independent of the graph cut, the notion of edge-wise frustration is highly non-local as the frustration depends on the behavior of the synchronization solution on the entire graph. We conjecture that a Cheeger-type inequality for SynCut, if exists, will reflect the global geometry information encoded by geometric quantities associated with the fibre bundle.
\item \emph{Statistical Framework for Learning Group Actions.} The LGA problem presented in this paper is not formulated with a natural generative model for the dataset of objects with pairwise transformations; nor is assumed any concrete noise models. It would be of interest to provide a systematic, statistical framework under which the problem of LGA and LGAS can be quantitatively analyzed and understood; we believe such a framework also has the potential to bridge statistical inference with synchronization problems.
\end{enumerate}
%\vspace{.05in}

\appendix
\renewcommand*{\thesection}{Appendix~\Alph{section}}
\normalsize

\section{Proofs of Proposition \ref{prop:synchronizability_flat_bundle} and Formula \eqref{eq:edge-potential-inner-product-equivalent}}
\label{sec:appendix-proof-fibre-bundle-sync}

\begin{proof}[Proof of Proposition \ref{prop:synchronizability_flat_bundle}]
  The construction of $\mathfrak{U}$ using the stars of the vertices of $\Gamma$ ensures that
\begin{enumerate}[(1)]
\item\label{item:1} $U_i\cap U_j\neq \emptyset$ if and only if $\left( i,j \right)\in E$;
\item\label{item:2} $U_i\cap U_j\cap U_k \neq \emptyset$ if and only if the $2$-simplex $\left( i,j,k \right)$ is in $X$.
\end{enumerate}
For such pair $\left( i,j \right)$, define constant map $g_{ij}:U_i\cap U_j\rightarrow G$ as
\begin{equation*}
  g_{ij} \left( x \right)=\rho_{ij}\qquad \forall x\in U_i\cap U_j.
\end{equation*}
Set $g_{ii}=e$ for all $1\leq i\leq \left| V \right|$, and note that $g_{ij} \left( x \right)=g_{ji}^{-1} \left( x \right)$ for all $x\in U_i\cap U_j$ by our assumption on $\rho$. If $\rho$ is synchronizable over $G$, let $f:V\rightarrow G$ be a vertex potential satisfying $\rho$, then $\rho_{ij}=f_i f_j^{-1}$ for all $\left( i,j \right)\in E$ from \eqref{eq:vertex-pot-satisfies-edge-pot}. Thus $\rho_{kj}\rho_{ji}=\rho_{ki}$ for any triangle $\left( i,j,k \right)$ in $\Gamma$, or equivalently that $g_{kj} \left( x \right)g_{ji} \left( x \right)=g_{ki} \left( x \right)$ for all $x\in U_i\cap U_j \cap U_k$. Therefore, $\left\{ g_{ij}\mid 1\leq i,j\leq \left| V \right|\right\}$ defines a system of \emph{coordinate transformations} \cite[\S2]{Steenrod1951} with values in $G$. These data determine a principal fibre bundle $\mathscr{P}_{\rho}$ with base space $\mathcal{X}$ and structure group $G$ --- by a standard construction in the theory of fibre bundles (see e.g. \cite[\S3.2]{Steenrod1951}) --- of which local trivializations are defined on the open sets in $\mathfrak{U}$ with constant transition functions $g_{ij}$; this principal bundle is thus flat by definition. Furthermore, the vertex potential $f$ and the compatibility constraints \eqref{eq:vertex-pot-satisfies-edge-pot} ensure that the following global section $s:\mathcal{X}\rightarrow \mathscr{P}_{\rho}$ is well-defined on this bundle:
\begin{equation*}
  s \left( x \right)=\phi_i\left( x, f_i\right),\quad x\in U_i
\end{equation*}
where $\phi_i:U_i\times G\rightarrow \mathscr{P}_{\rho}$ is the local trivialization of $\mathscr{P}_{\rho}$ over $U_i$. The triviality of this principal bundle then follows from the existence of such a global section; see e.g. \cite[\S8.3]{Steenrod1951}. The other direction of the proposition follows immediately from this triviality criterion for principal bundles.
\end{proof}

\begin{proof}[Proof of Formula \eqref{eq:edge-potential-inner-product-equivalent}]
  %\begin{equation*}
  %\begin{aligned}
  \begin{align*}
  \left\langle \omega, \eta \right\rangle &= \frac{1}{2}\sum_{\left( i,j \right)\in E}\left[w_{ij} \left\langle p_i\left(\omega^{\left( i \right)}_{ij}\right), p_i\left(\eta^{\left( i \right)}_{ij}\right) \right\rangle_F+w_{ji} \left\langle p_j\left(\omega^{\left( j \right)}_{ji}\right), p_j\left(\eta^{\left( j \right)}_{ji}\right) \right\rangle_F\right]\\
  &=\frac{1}{2}\sum_{\left( i,j \right)\in E}\left[w_{ij} \left\langle p_i\left(\omega^{\left( i \right)}_{ij}\right), p_i\left(\eta^{\left( i \right)}_{ij}\right) \right\rangle_F+w_{ij} \left\langle\rho_{ij} p_j\left(\omega^{\left( j \right)}_{ji}\right), \rho_{ij}p_j\left(\eta^{\left( j \right)}_{ji}\right) \right\rangle_F\right]\\
  &=\frac{1}{2}\sum_{\left( i,j \right)\in E}\left[w_{ij} \left\langle p_i\left(\omega^{\left( i \right)}_{ij}\right), p_i\left(\eta^{\left( i \right)}_{ij}\right) \right\rangle_F+w_{ij} \left\langle p_i\left(\omega^{\left( i \right)}_{ji}\right), \rho_{ij}p_i\left(\eta^{\left( i \right)}_{ji}\right) \right\rangle_F\right]\qquad\textrm{(see compatibility condition \eqref{eq:global-one-form-simplified})}\\
  &=\frac{1}{2}\sum_{\left( i,j \right)\in E}\left[w_{ij} \left\langle p_i\left(\omega^{\left( i \right)}_{ij}\right), p_i\left(\eta^{\left( i \right)}_{ij}\right) \right\rangle_F+w_{ij} \left\langle p_i\left(\omega^{\left( i \right)}_{ij}\right), \rho_{ij}p_i\left(\eta^{\left( i \right)}_{ij}\right) \right\rangle_F\right]\qquad\textrm{(skew-symmetry)}\\
  &=\sum_{\left( i,j \right)\in E}w_{ij} \left\langle p_i\left(\omega^{\left( i \right)}_{ij}\right), p_i\left(\eta^{\left( i \right)}_{ij}\right) \right\rangle_F.
  \end{align*}
  %\end{aligned}
%\end{equation*}
\end{proof}

\section{Graph Laplacian in Discrete Hodge Theory}
\label{sec:graph-lapl-discr}

Define $\mathbb{K}$-valued $0$- and $1$-forms on weighted graph $\Gamma=\left( V,E,w \right)$ as
\begin{equation*}
  \Omega^0 \left( \Gamma \right) := \left\{ f: V\rightarrow \mathbb{K} \right\},\quad \Omega^1 \left( \Gamma \right) := \left\{ \omega:E\rightarrow \mathbb{K}\mid\omega_{ij}=-\omega_{ji}\,\,\forall \left( i,j \right)\in E \right\},
\end{equation*}
equipped with natural inner products
\begin{equation*}
  \begin{aligned}
    \left\langle f,g \right\rangle&:=\sum_id_i \left\langle f_i, g_i  \right\rangle_{\mathbb{K}}, \quad\forall f,g\in \Omega^0 \left( \Gamma \right),\\
    \left\langle \omega,\eta \right\rangle&:=\sum_{\left( i,j \right)\in E}w_{ij}\left\langle \omega_{ij}, \eta_{ij} \right\rangle_{\mathbb{K}},\quad\forall \omega,\eta\in \Omega^1 \left( \Gamma \right),
  \end{aligned}
\end{equation*}
where $\left\langle \cdot,\cdot \right\rangle_{\mathbb{K}}$ is an inner product on $\mathbb{K}$, and $d_i=\sum_{j:\left( i,j \right)\in E}w_{ij}$ is the weighted degree at vertex $i\in V$. Analogous to the study of differential forms on a smooth manifold, one can define the \emph{differential} $d:\Omega^0 \left( \Gamma \right)\rightarrow\Omega^1 \left( \Gamma \right)$ and \emph{codifferential} $\delta:\Omega^1 \left( \Gamma \right)\rightarrow\Omega^0 \left( \Gamma \right)$ operators that are formal adjoints of each other:
\begin{equation*}
  \left(df\right)_{ij}=f_i-f_j,\quad\forall f\in\Omega^0 \left( \Gamma \right),\qquad \left(\delta\omega\right)_i:=\frac{1}{d_i}\sum_{j:\left( i,j \right)\in E}w_{ij}\omega_{ij},\quad\forall\omega\in\Omega^1 \left( \Gamma \right).
\end{equation*}
These constructions can be encoded into a de Rham cochain complex
\begin{equation*}
  0\mathrel{\substack{\xrightarrow[\hspace{0.15in}]{} \\ \vspace{-0.15in}  \\ \xleftarrow[]{\hspace{0.15in}}}}\Omega^0 \left( \Gamma \right) \mathrel{\substack{\xrightarrow[\hspace{0.15in}]{\textrm{\footnotesize $d$}} \\ \vspace{-0.15in}  \\ \xleftarrow[\textrm{\footnotesize $\delta$}]{\hspace{0.15in}}}}\Omega^1 \left( \Gamma \right) \mathrel{\substack{\xrightarrow[\hspace{0.15in}]{} \\ \vspace{-0.15in}  \\ \xleftarrow[]{\hspace{0.15in}}}} 0,
\end{equation*}
which realizes $L_0^{\mathrm{rw}}$, the \emph{graph random walk Laplacian}, as the Hodge Laplacian of degree zero:
\begin{equation*}
  \left(\Delta^{\left( 0 \right)}f\right)_i:=\left(\delta df\right)_i=\frac{1}{d_i}\sum_{j:\left( i,j \right)\in E}w_{ij} \left( f_i-f_j \right)=\left(L_0^{\mathrm{rw}}f\right)_i,\quad\forall i\in V,\,\,\forall f\in\Omega^0 \left( \Gamma \right).
\end{equation*}
It is well known that $L_0^{\mathrm{rw}}$ differs from the \emph{normalized graph Laplacian} $L_0$ by a similarity transform $L_0=D^{-1/2}L_0^{\mathrm{rw}}D^{1/2}$, where $D$ is a diagonal matrix with weighted degrees of each vertex on its diagonal.
\newline

%\newline\newline
%\clearpage
%\noindent\textbf{Acknowlegement} TG and SM would like to thank Ingrid Daubechies for many useful discussions. TG would also like to thank Ingrid Daubechies for her generous and continuous support, as well as Robert Bryant for discussions and feedback. JB would like to thank . SM would like to thank Vidit Nanda, Shmuel Weinberger, Steve Smale, Lek-Heng Lim for discussions.
\noindent\textbf{Software} MATLAB code implementing SynCut for the numerical simulations and application in automated geometric morphometrics is publicly available at \url{https://github.com/trgao10/GOS-SynCut}.
\newline%\newline

\noindent\textbf{Acknowlegement} The authors would like to thank Pankaj Agarwal, Douglas Boyer, Robert Bryant, Ingrid Daubechies, Pawe\l{} D\l{}otko, Kathryn Hess, Lek-Heng Lim, Vidit Nanda, Steve Smale, and Shmuel Weinberger for many inspirational discussions. %TG gratefully acknowledges the support from Simons Math+X Investigators Award \#400837; JB acknowledges (partial) support for this work by the EPSRC grants EP/I016945/1 and EP/N014189/1; SM would like to acknowledge NSF DMS 16-13261,  NSF IIS 1546331, NSF DMS-1418261, NSF IIS-1320357, and  NSF DMS-1045153 for support.

% TG would also like to thank Ingrid Daubechies for her generous and continuous support.

%\bibliographystyle{alpha}
\bibliographystyle{plain}
%\section*{References}
\bibliography{bib/references}

\end{document}